\documentclass{amsart}

\usepackage{enumerate, amsmath, amsfonts, amssymb, amsthm, wasysym, graphics, graphicx, xcolor, url, hyperref, hypcap, a4wide, pdflscape, multido, xargs, colortbl, multicol, multirow, calc, shuffle}
\hypersetup{colorlinks=true, citecolor=PineGreen, linkcolor=PineGreen}
\usepackage{tikz}\usetikzlibrary{trees,snakes,shapes,arrows,matrix,calc}
\usepackage{etex}
\graphicspath{{figures/}}
\makeatletter
\def\input@path{{figures/}}
\makeatother

\title{Hopf dreams and diagonal harmonics}

\author[N.~Bergeron]{Nantel Bergeron$^{\diamond}$} 
\address[N.~Bergeron]{Department of Mathematics and Statistics, York University, Toronto}
\email{bergeron@mathstat.yorku.ca}
\urladdr{http://www.math.yorku.ca/bergeron/}

\author[C.~Ceballos]{Cesar Ceballos$^{\star}$} 
\address[C.~Ceballos]{Institute of Geometry, TU Graz, Graz}
\email{cesar.ceballos@tugraz.at}
\urladdr{http://www.geometrie.tugraz.at/ceballos/}

\author[V.~Pilaud]{Vincent Pilaud$^{\ddagger}$} 
\address[V.~Pilaud]{CNRS \& LIX, \'Ecole Polytechnique, Palaiseau}
\email{vincent.pilaud@lix.polytechnique.fr}
\urladdr{http://www.lix.polytechnique.fr/~pilaud/}

\thanks{This project was partially supported by the project ``Austria/France Scientific \& Technological Cooperation'', which is co-financed by the Austrian Federal Ministry of Science, Research and Economy BMWFW (Project No. FR 10/2018) and by the French Ministry of Foreign Affairs and International Development (PHC Amadeus 2018 Project No. 39444WJ).
$^{\diamond}$NB was supported by NSERC and a York Research Chair.
$^\star$CC was supported by the Austrian Science Foundation FWF, grant F 5008-N15, in the framework of the Special Research Program Algorithmic and Enumerative Combinatorics.
$^\ddagger$VP was partially supported by the french ANR grants SC3A~(15\,CE40\,0004\,01) and CAPPS~(17\,CE40\,0018).
}

\newtheorem{theorem}{Theorem}[section]
\newtheorem{corollary}[theorem]{Corollary}
\newtheorem{proposition}[theorem]{Proposition}
\newtheorem{lemma}[theorem]{Lemma}
\newtheorem{definition}[theorem]{Definition}
\newtheorem{conjecture}[theorem]{Conjecture}

\theoremstyle{definition}
\newtheorem{example}[theorem]{Example}
\newtheorem{remark}[theorem]{Remark}
\newtheorem{question}[theorem]{Question}

\newcommand{\N}{\mathbb{N}} 
\newcommand{\Z}{\mathbb{Z}} 
\newcommand{\C}{\mathbb{C}} 
\newcommand{\bk}{\mathbf{k}} 

\newcommand{\set}[2]{\left\{ #1 \;\middle|\; #2 \right\}} 
\newcommand{\bigset}[2]{\big\{ #1 \;|\; #2 \big\}} 
\newcommand{\ssm}{\smallsetminus} 
\newcommand{\eqdef}{\mbox{\,\raisebox{0.2ex}{\scriptsize\ensuremath{\mathrm:}}\ensuremath{=}\,}} 

\newcommand{\fref}[1]{Figure~\ref{#1}} 
\newcommand{\ie}{\textit{i.e.}~} 
\definecolor{PineGreen}{RGB}{2,120,120} 
\definecolor{DarkGreen}{RGB}{57,181,74} 
\renewcommand{\b}[1]{{\color{blue} #1}} 
\renewcommand{\r}[1]{{\color{red} #1}} 
\newcommand{\g}[1]{{\color{DarkGreen} #1}} 
\renewcommand{\o}[1]{{\color{orange} #1}} 
\newcommand{\defn}[1]{\textbf{\textsf{\color{PineGreen} #1}}} 
\usepackage{todonotes}

\newcommand{\fS}{\mathfrak{S}} 
\newcommand{\gdproduct}{\bullet} 
\newcommand{\gdshuffle}{\shuffle_\gdproduct} 
\newcommand{\gddeconcat}{\coproduct_\gdproduct} 
\newcommand{\HS}{{\bk\fS}} 

\newcommand{\boxsize}{.35}
\newlength{\verticalOffset}
\newlength{\verticalShift}
\tikzstyle{dashed} = [dash pattern = on 2pt off 1pt]
\tikzstyle{dashdotted} = [dash pattern=on .5pt off .5pt on 1.7pt off .5pt]

\newcounter{length}
\newcommand{\length}[1]{%
	\setcounter{length}{0}%
	\foreach \x in {#1} {%
		\stepcounter{length}%
	}%
}

\newcommand{\pipeDreamMonoColor}[3]{
	\length{#3}%
	\begin{tikzpicture}[baseline = \value{length}*\verticalShift+\verticalOffset, scale=1]
		\coordinate (origin) at (0,0);
		\newcount{\y} \y=0
		\newcount{\x}
		\foreach \line in {#3} {
			\x=0
			\foreach \t in \line {
				\coordinate (W) at ($ (origin) + ( \boxsize * \x , -\boxsize * \y ) + ( 0      , \boxsize / 2 ) $);
				\coordinate (E) at ($ (origin) + ( \boxsize * \x , -\boxsize * \y ) + ( \boxsize     , \boxsize / 2 ) $);
				\coordinate (N) at ($ (origin) + ( \boxsize * \x , -\boxsize * \y ) + ( \boxsize / 2 , \boxsize     ) $);
				\coordinate (S) at ($ (origin) + ( \boxsize * \x , -\boxsize * \y ) + ( \boxsize / 2 , 0 ) $);
				\coordinate (C) at ($ (origin) + ( \boxsize * \x , -\boxsize * \y ) + ( \boxsize / 2 , \boxsize / 2 ) $);
				\ifthenelse{\equal{\t}{e}}{
					\draw[rounded corners=\boxsize * 8, thick, #1] (W) -- (C) -- (N);
					\draw[rounded corners=\boxsize * 8, thick, #1] (S) -- (C) -- (E);			
				}{
				\ifthenelse{\equal{\t}{c}}{
					\draw[color=#1, thick] (W) -- (E);
					\draw[color=#1, thick] (S) -- (N);
				}{
				\ifthenelse{\equal{\t}{t}}{
					\draw[rounded corners=\boxsize * 8, #2] (W) -- (C) -- (N);
					\draw[rounded corners=\boxsize * 8, thick, #1] (S) -- (C) -- (E);			
				}{
				\ifthenelse{\equal{\t}{b}}{
					\draw[rounded corners=\boxsize * 8, thick, #1] (W) -- (C) -- (N);
					\draw[rounded corners=\boxsize * 8, #2] (S) -- (C) -- (E);			
				}{
				\ifthenelse{\equal{\t}{tb}}{
					\draw[rounded corners=\boxsize * 8, #2] (W) -- (C) -- (N);
					\draw[rounded corners=\boxsize * 8, #2] (S) -- (C) -- (E);			
				}{
				\ifthenelse{\equal{\t}{n}}{}{\node at (C) {$\small \t$};
				}}}}}}
				\global\advance\x by 1
			}
			\global\advance\y by 1
		}
	\end{tikzpicture}%
}

\newcommand{\pipeDreamBiColor}[4]{ 
	\length{#4}%
	\begin{tikzpicture}[baseline = \value{length}*\verticalShift+\verticalOffset, scale=1]
		\coordinate (origin) at (0,0);
		\newcount{\y} \y=0
		\newcount{\x}
		\foreach \line in {#4} {
			\x=0
			\foreach \t/\colorW/\colorS in \line {
				\coordinate (W) at ($ (origin) + ( \boxsize * \x , -\boxsize * \y ) + ( 0      , \boxsize / 2 ) $);
				\coordinate (E) at ($ (origin) + ( \boxsize * \x , -\boxsize * \y ) + ( \boxsize     , \boxsize / 2 ) $);
				\coordinate (N) at ($ (origin) + ( \boxsize * \x , -\boxsize * \y ) + ( \boxsize / 2 , \boxsize     ) $);
				\coordinate (S) at ($ (origin) + ( \boxsize * \x , -\boxsize * \y ) + ( \boxsize / 2 , 0 ) $);
				\coordinate (C) at ($ (origin) + ( \boxsize * \x , -\boxsize * \y ) + ( \boxsize / 2 , \boxsize / 2 ) $);
				\ifthenelse{\equal{\t}{e}}{
					\ifthenelse{\equal{\colorW}{l}}{\draw[rounded corners=\boxsize * 8, thick, #1] (W) -- (C) -- (N);}{}
					\ifthenelse{\equal{\colorW}{r}}{\draw[rounded corners=\boxsize * 8, thick, #2] (W) -- (C) -- (N);}{}
					\ifthenelse{\equal{\colorW}{b}}{\draw[rounded corners=\boxsize * 8, #3] (W) -- (C) -- (N);}{}
					\ifthenelse{\equal{\colorS}{l}}{\draw[rounded corners=\boxsize * 8, thick, #1] (S) -- (C) -- (E);}{}
					\ifthenelse{\equal{\colorS}{r}}{\draw[rounded corners=\boxsize * 8, thick, #2] (S) -- (C) -- (E);}{}
					\ifthenelse{\equal{\colorS}{b}}{\draw[rounded corners=\boxsize * 8, #3] (S) -- (C) -- (E);}{}
				}{
				\ifthenelse{\equal{\t}{c}}{
					\ifthenelse{\equal{\colorW}{l}}{\draw[thick, #1] (W) -- (E);}{}
					\ifthenelse{\equal{\colorW}{r}}{\draw[thick, #2] (W) -- (E);}{}
					\ifthenelse{\equal{\colorS}{l}}{\draw[thick, #1] (S) -- (N);}{}
					\ifthenelse{\equal{\colorS}{r}}{\draw[thick, #2] (S) -- (N);}{}
				}{
				\ifthenelse{\equal{\t}{h}}{
					\ifthenelse{\equal{\colorW}{l}}{\draw[thick, #1] (W) -- (E);}{}
					\ifthenelse{\equal{\colorW}{r}}{\draw[thick, #2] (W) -- (E);}{}
					\ifthenelse{\equal{\colorS}{l}}{\draw[thick, #1] ($(S)+(0,\boxsize / 3)$) -- ($(N)-(0,\boxsize / 3)$);}{}
				}{
				\ifthenelse{\equal{\t}{v}}{
					\ifthenelse{\equal{\colorS}{l}}{\draw[thick, #1] (S) -- (N);}{}
					\ifthenelse{\equal{\colorS}{r}}{\draw[thick, #2] (S) -- (N);}{}
					\ifthenelse{\equal{\colorW}{r}}{\draw[thick, #2] ($(W)+(\boxsize / 3,0)$) -- ($(E)-(\boxsize / 3,0)$);}{}
				}{
				\ifthenelse{\equal{\t}{n}}{}{\node at (C) {$\small \t$};}}}}}
				\global\advance\x by 1
			}
			\global\advance\y by 1
		}
	\end{tikzpicture}%
}

\newcommand{\pipeDreamTriColor}[5]{ 
	\length{#5}%
	\begin{tikzpicture}[baseline = \value{length}*\verticalShift+\verticalOffset, scale=1]
		\coordinate (origin) at (0,0);
		\newcount{\y} \y=0
		\newcount{\x}
		\foreach \line in {#5} {
			\x=0
			\foreach \t/\colorW/\colorS in \line {
				\coordinate (W) at ($ (origin) + ( \boxsize * \x , -\boxsize * \y ) + ( 0      , \boxsize / 2 ) $);
				\coordinate (E) at ($ (origin) + ( \boxsize * \x , -\boxsize * \y ) + ( \boxsize     , \boxsize / 2 ) $);
				\coordinate (N) at ($ (origin) + ( \boxsize * \x , -\boxsize * \y ) + ( \boxsize / 2 , \boxsize     ) $);
				\coordinate (S) at ($ (origin) + ( \boxsize * \x , -\boxsize * \y ) + ( \boxsize / 2 , 0 ) $);
				\coordinate (C) at ($ (origin) + ( \boxsize * \x , -\boxsize * \y ) + ( \boxsize / 2 , \boxsize / 2 ) $);
				\ifthenelse{\equal{\t}{e}}{
					\ifthenelse{\equal{\colorW}{l}}{\draw[rounded corners=\boxsize * 8, thick, #1] (W) -- (C) -- (N);}{}
					\ifthenelse{\equal{\colorW}{m}}{\draw[rounded corners=\boxsize * 8, thick, #2] (W) -- (C) -- (N);}{}
					\ifthenelse{\equal{\colorW}{r}}{\draw[rounded corners=\boxsize * 8, thick, #3] (W) -- (C) -- (N);}{}
					\ifthenelse{\equal{\colorW}{b}}{\draw[rounded corners=\boxsize * 8, #4] (W) -- (C) -- (N);}{}
					\ifthenelse{\equal{\colorS}{l}}{\draw[rounded corners=\boxsize * 8, thick, #1] (S) -- (C) -- (E);}{}
					\ifthenelse{\equal{\colorS}{m}}{\draw[rounded corners=\boxsize * 8, thick, #2] (S) -- (C) -- (E);}{}
					\ifthenelse{\equal{\colorS}{r}}{\draw[rounded corners=\boxsize * 8, thick, #3] (S) -- (C) -- (E);}{}
					\ifthenelse{\equal{\colorS}{b}}{\draw[rounded corners=\boxsize * 8, #4] (S) -- (C) -- (E);}{}
				}{
				\ifthenelse{\equal{\t}{c}}{
					\ifthenelse{\equal{\colorW}{l}}{\draw[thick, #1] (W) -- (E);}{}
					\ifthenelse{\equal{\colorW}{m}}{\draw[thick, #2] (W) -- (E);}{}
					\ifthenelse{\equal{\colorW}{r}}{\draw[thick, #3] (W) -- (E);}{}
					\ifthenelse{\equal{\colorS}{l}}{\draw[thick, #1] (S) -- (N);}{}
					\ifthenelse{\equal{\colorS}{m}}{\draw[thick, #2] (S) -- (N);}{}
					\ifthenelse{\equal{\colorS}{r}}{\draw[thick, #3] (S) -- (N);}{}
				}{
				\ifthenelse{\equal{\t}{n}}{}{\node at (C) {$\small \t$};}}}
				\global\advance\x by 1
			}
			\global\advance\y by 1
		}
	\end{tikzpicture}%
}

\newcommand{\crossMonoColor}[1]{\raisebox{.1cm}{\pipeDreamMonoColor{#1}{black}{c}}\,}
\newcommand{\cross}{\raisebox{-.15cm}{\includegraphics[scale=.8]{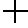}}\,}
\newcommand{\crossBiColor}[2]{\raisebox{.1cm}{\pipeDreamBiColor{#1}{#2}{black}{c/l/r}}}

\newcommand{\elbow}{\raisebox{-.15cm}{\includegraphics[scale=.8]{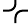}}\,}
\newcommand{\elbowBiColor}[2]{\raisebox{.1cm}{\pipeDreamBiColor{#1}{#2}{black}{e/l/r}}}
\newcommand{\transpose}{\mathsf{transp}}
\newcommand{\pipeDreams}{\Pi} 
\newcommand{\HP}{{\bk\pipeDreams}} 
\newcommand{\HT}{{\bk\nuTrees}} 
\newcommand{\indecomposablePipeDreams}{\Lambda} 
\newcommand{\lrot}{\rotatebox{90}{$\mathsf{L}$}}
\newcommand{\lmir}{\reflectbox{$\mathsf{L}$}}
\newcommand{\contact}{^\#} 
\newcommand{\gammainsertion}{\ll}  
\newcommand{\ltlex}{<_{\mathrm{lex}}} 
\newcommand{\duality}{^\star} 
\newcommand{\Ptop}{P^{\mathrm{top}}} 

\newcommand{\HA}{{\bk\acyclicPipeDreams}} 
\newcommand{\acyclicPipeDreams}{\Sigma} 

\newcommand{\Sid}{S^{\operatorname{Id}}} 
\newcommand{\Sdom}{\fS^{\operatorname{dom}}} 
\newcommand{\pipeDreamsDom}{\pipeDreams^{\operatorname{dom}}} 
\newcommand{\HSdom}{{\bk\fS}^{\operatorname{dom}}} 
\newcommand{\HPdom}{\bk\pipeDreamsDom} 
\newcommand{\nuTrees}{\Theta} 

\newcommand{\product}{\cdot} 
\newcommand{\coproduct}{\triangle} 
\newcommand{\columnReading}{\mathsf{CRW}} 
\newcommand{\leadingTerm}{\mathsf{LT}} 
\newcommand{\leadingCoefficient}{\mathsf{LC}} 
\newcommand{\leadingPipeDream}{\mathsf{LP}} 
\newcommand{\elbowLetter}{\elbow}
\newcommand{\crossLetter}{\cross}

\newcommand{\DiagHarm}[2]{\mathrm{DH}_{#1,#2}} 
\newcommand{\area}{\operatorname{area}} 
\newcommand{\dinv}{\operatorname{dinv}} 
\newcommand{\bounce}{\operatorname{bounce}} 
\newcommand{\steep}{\operatorname{steep}} 
\newcommand{\col}{\operatorname{col}} 
\newcommand{\type}{\operatorname{type}} 
\newcommand{\Sym}[2]{\mathrm{Sym}_{#1,#2}} 
\newcommand{\Alt}{\mathrm{Alt}} 
\newcommand{\TC}[2]{\mathrm{TC}_{#1,#2}} 
\newcommand{\HC}[2]{\mathrm{HC}_{#1,#2}} 
\newcommand{\bpi}{\boldsymbol \pi} 
\newcommand{\LLT}[2]{\mathbb{L}_{#1}(#2)} 

\makeatletter
\def\part{\@startsection{part}{1}%
\z@{.7\linespacing\@plus\linespacing}{.8\linespacing}%
{\LARGE\sffamily\centering}}
\@addtoreset{section}{part}
\makeatother

\makeatletter
\def\l@section{\@tocline{1}{5pt}{0pc}{}{}}
\makeatother
\let\oldtocpart=\tocpart
\renewcommand{\tocpart}[2]{\bf\large\oldtocpart{#1}{#2}}
\let\oldtocsection=\tocsection
\renewcommand{\tocsection}[2]{\bf\oldtocsection{#1}{#2}}

\begin{document}

\begin{abstract}
This paper introduces a Hopf algebra structure on a family of reduced pipe dreams. 
We show that this Hopf algebra is free and cofree, and construct a surjection onto a commutative Hopf algebra of permutations. 
The pipe dream Hopf algebra contains Hopf subalgebras with interesting sets of generators and Hilbert series related to subsequences of Catalan numbers. 
Three other relevant Hopf subalgebras include 
the Loday--Ronco Hopf algebra on complete binary trees,
a Hopf algebra related to a special family of lattice walks on the quarter plane, and 
a Hopf algebra on $\nu$-trees related to $\nu$-Tamari lattices. 
One of this Hopf subalgebras motivates a new notion of Hopf chains in the Tamari lattice, which are used to present applications and conjectures 
in the theory of multivariate diagonal harmonics. 

\medskip
\noindent
\textsc{MSC classes}: 16T05, 16T30, 05E10
\end{abstract}

\maketitle

\tableofcontents


\section*{Introduction}

Pipe dreams (or RC-graphs) are combinatorial objects closely related to reduced expressions of permutations in terms of simple transpositions.
They were introduced by N.~Bergeron and S.~Billey in~\cite{BergeronBilley} to compute Schubert polynomials and later revisited in the context of Gr\"obner geometry by A.~Knutson and E.~Miller~\cite{KnutsonMiller-GroebnerGeometry}, who coined the name \emph{pipe dreams} in reference to a game involving pipe connections. In brief, a pipe dream is an arrangement of pipes, each connecting an entry on the vertical axis to an exit on the horizontal axis, and remaining in a triangular shape of the grid.
Pipe dreams are grouped according to their exiting permutation, given by the order in which the pipes appear along the horizontal axis.
Pipe dreams have important connections and applications to various areas related to Schubert calculus and Schubert varieties~\cite{LascouxSchutzenberger-PolynomesSchubert, LascouxSchutzenberger-SchubertLittlewoudRichardson}. 
They have a rich combinatorial and geometric structure but their algebraic structure was less considered.

The objective of this paper is to introduce a Hopf algebra structure on pipe dreams.
Hopf algebras are rather rigid structures which often reveal deep combinatorial properties and connections.
This paper contributes to this general philosophy: the Hopf algebra of pipe dreams will give us insight  on a special family of lattice walks on the quarter plane studied by M.~Bousquet-M\'elou, S.~Melczer, M.~Mishna, and A.~Rechnitzer in a series of papers~\cite{BousquetMelouMishna,MishnaRechnitzer,MelczerMishna}, as well as applications to the still emerging theory of multivariate diagonal harmonics~\cite{Bergeron-multivariateDiagonalCoinvariantSpaces}.

The starting point of this project is the Hopf algebra of J.-L.~Loday and M.~Ronco on complete binary trees~\cite{LodayRonco}.
There is a strong correspondence between the complete binary trees with $n$ internal nodes and the reduced pipe dreams with exiting permutation~$0 n \dots 1$.
This correspondence preserves a lot of the combinatorial structure and allows to interpret the product and the coproduct of the Loday--Ronco Hopf algebra in terms of pipe dreams.
This interpretation yields to an extension of the Loday--Ronco Hopf algebra on a bigger family~$\pipeDreams$ consisting of reduced pipe dreams with an elbow in the top left corner.
We show that it results in a free and cofree Hopf algebra structure~$(\HP, \product, \coproduct)$.
We also show that mapping a pipe dream to its exiting permutation defines a surjective morphism from the Hopf algebra~$(\HP, \product, \coproduct)$ of pipe dreams to a commutative Hopf algebra~$(\HS, \gdshuffle, \gddeconcat)$ of permutations.

\medskip

The pipe dream Hopf algebra~$\HP$ has many Hopf subalgebras with interesting combinatorial structure and enumeration (Hilbert series).
These Hopf algebras are obtained using pipe dreams whose exiting permutations belong to a Hopf subalgebra of~$\HS$.
Even the first naive examples, obtained from permutations with restricted atom sets (\ie with prescribed decompositions in~$\HS$), give rise to relevant Hopf algebras whose sets of generators are counted by formulas involving Catalan numbers.
Three relevant Hopf subalgebras are:
\begin{enumerate}
\item the Loday--Ronco Hopf algebra on complete binary trees~\cite{LodayRonco}.
\item a Hopf algebra related to a special family of lattice walks on the quarter plane~\cite{BousquetMelouMishna, MishnaRechnitzer, MelczerMishna}.
\item a Hopf algebra on $\nu$-trees connected to the $\nu$-Tamari lattices of L.~F.~Pr\'eville-Ratelle and X.~Viennot~\cite{PrevilleRatelleViennot}.
\end{enumerate}

The Loday--Ronco Hopf algebra has a remarkable geometric interpretation which, together with~\cite{GelfandKrobLascouxLeclercRetakhThibon,MalvenutoReutenauer}, implies significant results about their algebraic structures as shown by M.~Aguiar and F.~Sottile in~\cite{AguiarSottile-MalvenutoReutenauer, AguiarSottile-LodayRonco} and by F.~Hivert, J.-C.~Novelli and J.-Y.~Thibon~\cite{HivertNovelliThibon-algebraBinarySearchTrees}.
It also inspired work on related Hopf algebras including the Cambrian Hopf algebra of G.~Chatel and V.~Pilaud~\cite{ChatelPilaud} related to the Cambrian lattices introduced by N.~Reading~\cite{Reading-CambrianLattices}, the Hopf algebra of N.~Bergeron and C.~Ceballos~\cite{BergeronCeballos} related to A.~Knutson and E.~Miller's theory of subword complexes~\cite{KnutsonMiller-subwordComplex}, and the Hopf algebra of V.~Pilaud~\cite{Pilaud-brickAlgebra} related to brick polytopes~\cite{PilaudSantos-brickPolytope, PilaudStump-brickPolytopes}.
On the other side, the enumeration of lattice walks in the quarter plane is a challenging problem of interest in combinatorics and computer science.
M.~Bousquet-M\'elou and M.~Mishna considered 79 different models in~\cite{BousquetMelouMishna}.
The family of walks that we consider has acquired special attention~\cite{MishnaRechnitzer, MelczerMishna}, and our Hopf algebra approach gives an alternative (conjectural) refinement of their enumeration (see Remark~\ref{rem:latticewalksEnumeration} and Corollary~\ref{cor:walksEnumerationDeterminants}).
It also leads to a new conjecture concerning a bijection between two families of pairs of nested Dyck paths related to the zeta map in $q,t$-Catalan combinatorics~\cite{Haglund-qt-catalan} (see Conjecture~\ref{conj:SteepBounce}).
Last but not least, L.-F.~Pr\'eville-Ratelle and X.~Viennot~\cite{PrevilleRatelleViennot} recently introduced the $\nu$-Tamari lattices inspired by F.~Bergeron's conjectural connections between interval enumeration 
in the classical and Fuss--Catalan Tamari lattices and dimension formulas of certain spaces in trivariate and higher multivariate diagonal harmonics~\cite{BergeronPrevilleRatelle, BousquetMelouChapuyPrevilleRatelle, BousquetMelouFusyPrevilleRatelle}.
Even though no application of $\nu$-Tamari lattices are known in this context, they possess remarkable enumerative and geometric properties~\cite{FangPrevilleRatelle, CeballosPadrolSarmiento-geometryNuTamari}.

\medskip

One of the most important contributions of this paper is the application of the Hopf algebra of dominant pipe dreams to the theory of multivariate diagonal harmonics.
The foundation of diagonal harmonics was inspired by famous conjectures and results related to the theory of Macdonald polynomials pioneered by I.~G.~Macdonald~\cite{Macdonald-newClassSymFunct}.
The discovery of these polynomials originated important developments including the proof of the Macdonald constant-term identities~\cite{Cherednik} and the resolution of the Macdonald positivity conjecture~\cite{Haiman-hilbertSchemes}, as well as many results in connection with representation theory of quantum groups~\cite{EtingofKirillov}, affine Hecke algebras~\cite{KirillovNoumi, Knop, Macconald-affineHeckeAlgebras}, and the Calogero--Sutherland model in particle physics~\cite{LapointeVinet}.
Diagonal harmonics is also connected to many areas in mathematics where they play a central role, including the rectangular and rational Catalan combinatorics in algebraic combinatorics~\cite{ArmstrongLoehrWarrington-parking, ArmstrongLoehrWarrington-sweep, Bergeron-rectangular, BergeronGarsiaSergelLevenXin-compositional, Mellit-rationalshuffle}, cohomology of flag manifolds and group schemes in algebraic topology~\cite{DalalMorse},
Hilbert schemes in algebraic geometry~\cite{Haiman-hilbertSchemes, Haiman-vanishingTheorems}, homology of torus links in knot theory~\cite{GorskyNegut, Mellit-Homology}, and more.

The space of diagonal harmonics is an $\mathfrak{S}_n$-module of polynomials in two sets of variables that satisfy some harmonic properties.
The dimensions of these spaces (and of their bigraded components) led to numerous important conjectures.
These include the $(n+1)^{n-1}$-conjecture by A.~Garsia and M.~Haiman~\cite{GarsiaHaiman-gradedRepresentationModel, GarsiaHaiman-BigradedModules}, proved by M.~Haiman using properties of the Hilbert scheme in algebraic geometry~\cite{Haiman-vanishingTheorems}, and the shuffle conjecture by J.~Haglund et al.~\cite{HaglundHaimanLoehrRemmelUlyanov}, recently proved by E.~Carlsson and A.~Mellit in~\cite{CarlssonMellit}.
The bigraded Hilbert series of the alternating component of the space of diagonal harmonics also gave rise to the now famous \mbox{$q,t$-Catalan} polynomials~\cite{Haglund-qt-catalan}.

The module of diagonal harmonics has natural generalizations in three or more sets of variables.
However, very little is known about these generalizations and the techniques from algebraic geometry do not straightforward apply.
Computational experiments by M.~Haiman from the early 1990's~\cite[Fact~2.8.1]{Haiman-conjectures} suggest explicit simple dimension formulas for the space of diagonal harmonics and its alternating component in the trivariate case.  
F.~Bergeron noticed that these formulas coincide with formulas counting labeled and unlabeled intervals in the classical Tamari lattice, and opened the door to a more systematic study of the multivariate case~\cite{BergeronPrevilleRatelle,Bergeron-multivariateDiagonalCoinvariantSpaces}. 
In particular, together with L.-F.~Pr\'eville-Ratelle, they established a connection between trivariate diagonal harmonics and intervals in the Tamari lattice, by presenting an explicit combinatorial conjecture for the Frobenius characteristic that involves two statistics $\dinv$ and length of the longest chain in the intervals~\cite[Conj.~1]{BergeronPrevilleRatelle}.  
This conjecture and the dimension formulas in the trivariate case remain widely open. In addition, not even a dimension formula is known yet for the four variable case.
F.~Bergeron suggested to us that there might be a way to understand the $r$-variate diagonal harmonics in terms of some suitable $(r-1)$-chains in the Tamari lattice.
In this paper we present a milestone towards this understanding by introducing a new class of chains in the Tamari lattice that was motivated by our Hopf algebra construction.

A very natural question is to ask about the number of Tamari chains related to intervals in the graded dimensions of the Hopf algebra of dominant pipe dreams.
This motivates a natural notion of \emph{Hopf chains} of Dyck paths.
A Hopf chain is a chain~$(\pi_1,\dots,\pi_r)$ in the classical Tamari lattice satisfying an extra property motivated from the Hopf algebra, and such that $\pi_1$ is the diagonal path. Surprisingly, they turn out to be closely related to multivariate diagonal harmonics, not only in the numerology related to dimension formulas, but also in their underlying representation theory. More explicitly, let us denote by~$\DiagHarm{n}{r}$ the diagonal harmonics space on $r$ sets of $n$ variables.
We will prove that, for degree~$n\leq 4$ and \emph{any number $r$} of sets of variables, the $q,t$-Frobenius characteristic of~$\DiagHarm{n}{r}$ can be obtained as a generating function of Hopf chains using a new \emph{collar} statistic and the LLT polynomials of A.~Lascoux, B.~Leclerc and J.-Y.~Thibon (see Theorem~\ref{thm:LLT}).
Our result generalizes the shuffle conjecture to the multivariate case (for $n\leq 4$), and has several immediate consequences. 
For instance, the dimensions of~$\DiagHarm{n}{r}$ and its alternating component~$\Alt(\DiagHarm{n}{r})$ are equal to the number of labeled and unlabeled Hopf chains, respectively, and the bigraded Hilbert series of~$\Alt(\DiagHarm{n}{r})$ is the generating function of Hopf chains with respect to the collar and $\dinv$ statistics (see Corollary~\ref{coro:bigcoro}).


\part{The pipe dream Hopf algebra}

This first part introduces a Hopf algebra structure on certain pipe dreams (Section~\ref{sec:HopfAlgebraPipeDreams}), shows its freeness (Section~\ref{sec:freeness}) and studies its connection to a Hopf algebra on permutations (Section~\ref{sec:HopfAlgebraPermutations}).


\section{A Hopf algebra on permutations}
\label{sec:HopfAlgebraPermutations}

Before we work on pipe dreams, we first introduce a Hopf structure on permutations. This Hopf algebra of permutations  is commutative and non-cocommutative. In particular it is different from the classical Malvenuto--Reutenauer Hopf algebra~\cite{MalvenutoReutenauer}. It is isomorphic to one of the commutative Hopf algebras introduced by Y.~Vargas in~\cite{Vargas}.
We denote by~$\fS_n$ the set of permutations of~$[n] \eqdef \{1,2,\ldots,n\}$, and we let~$\fS \eqdef \bigsqcup_{n \ge 0} \fS_n$. We consider the graded vector space~$\HS \eqdef \bigoplus_{n\ge 0} \HS_n$, where~$\HS_n$ is the~$\bk$-span of the permutations in~$\fS_n$.


\subsection{Global descents and atomic permutations}
\label{subsec:globalDescents}

Consider a permutation~$\omega \in \fS_n$. Index by~$\{0, \dots, n\}$ from left to right the \defn{gaps} before the first position, between two consecutive positions, or after the last position of~$\omega$. A gap~$\gamma$ is a \defn{global descent} if~$\omega([\gamma]) = [n] \ssm [n-\gamma]$. In other words, the first~$\gamma$ positions are sent to the last~$\gamma$ values. For example, the global descents in the permutation~$\omega = 635421$ are~$0, 1, 4, 5, 6$. Note that the gaps~$0$ and~$n$ are always global descents. A permutation with no other global descent is called \defn{atomic}.

For two permutations~$\mu \in \fS_m$ and~$\nu \in \fS_n$ we define the permutation~$\mu \gdproduct \nu \in \fS_{m+n}$ by 
\[
\mu \gdproduct \nu(i) \eqdef
\begin{cases}
	\mu(i)+n & \text{ if } 1   \le i \le m,  \\
	\nu(i-m) & \text{ if } m+1 \le i \le m+n.
\end{cases}
\]
The product~$\gdproduct$ is associative, and the unique permutation~$\epsilon$ of~$\fS_0$ is neutral for~$\gdproduct$. Observe that~$\mu \gdproduct \nu$ has a global descent in position~$m$. Conversely, given a permutation~$\omega \in \fS_{m+n}$ with a global descent in position~$m$, there exist a unique pair of permutations~${\mu \in \fS_m}$ and~${\nu \in \fS_{n}}$ such that~${\omega = \mu \gdproduct \nu}$. Therefore any permutation~$\omega \in \fS$ factorizes in a unique way as a product~${\omega = \omega_1 \gdproduct \omega_2 \gdproduct \cdots \gdproduct \omega_\ell}$ of atomic permutations~$\omega_i$. For example, we have~$\b{6}\r{354}\g{2}\o{1} = \b{1} \gdproduct \r{132} \gdproduct \g{1} \gdproduct \o{1}$. For such a factorization, we denote by~$\omega^{\gdproduct} \eqdef \{\omega_1, \omega_2, \dots, \omega_\ell\}$ the set of its atomic factors.


\subsection{Coproduct on permutations}
\label{subsec:coproductPermutations}

We first introduce a coproduct~$\gddeconcat$ on permutations. Consider a permutation~$\omega \in \fS$, and let~${\omega = \omega_1 \gdproduct \cdots \gdproduct \omega_\ell}$ be its unique factorization into atomic permutations. We define the coproduct~$\gddeconcat(\omega)$ by
\[
\gddeconcat(\omega) \eqdef \sum_{i=0}^\ell (\omega_1 \gdproduct \cdots \gdproduct \omega_i) \otimes (\omega_{i+1} \gdproduct \cdots \gdproduct \omega_\ell),
\]
where an empty $\gdproduct\,$-product is the neutral element~$\epsilon$ for~$\gdproduct$. This coproduct extends to~$\HS$ by linearity and is clearly coassociative.

\begin{example} 
For the permutation~$635421 \in \fS_6$, we have~$\b{6}\r{354}\g{2}\o{1} = \b{1} \gdproduct \r{132} \gdproduct \g{1} \gdproduct \o{1}$. Hence
\begin{align*}
\gddeconcat(\b{6}\r{354}\g{2}\o{1}) & = \epsilon \otimes \b{6}\r{354}\g{2}\o{1} + \b{1} \otimes (\r{132} \gdproduct \g{1} \gdproduct \o{1}) + (\b{1} \gdproduct \r{132}) \otimes (\g{1} \gdproduct \o{1}) + (\b{1} \gdproduct \r{132} \gdproduct \g{1}) \otimes \o{1} + \b{6}\r{354}\g{2}\o{1} \otimes \epsilon \\
& = \epsilon \otimes \b{6}\r{354}\g{2}\o{1} + \b{1} \otimes \r{354}\g{2}\o{1} + \b{4}\r{132} \otimes \g{2}\o{1} + \b{5}\r{243}\g{1} \otimes \o{1} + \b{6}\r{354}\g{2}\o{1} \otimes \epsilon.
\end{align*}
\end{example}


\subsection{Product on permutations}
\label{subsec:productPermutations}

We now introduce a commutative product~$\gdshuffle$ on permutations. Consider two permutations~$\b{\pi}, \r{\omega} \in \fS$. Define first~$\b{\pi} \gdshuffle \r{\epsilon} \eqdef \b{\pi}$ and~$\b{\epsilon} \gdshuffle \r{\omega} \eqdef \r{\omega}$. Assume now that~$\b{\pi} = \b{\mu} \gdproduct \b{\nu}$ and~$\r{\omega} = \r{\sigma} \gdproduct \r{\tau}$ where~$\b{\mu}$ and~$\r{\sigma}$ are non-trivial atomic permutations, and define
\[
\b{\pi} \gdshuffle \r{\omega} \eqdef \b{\mu} \gdproduct (\b{\nu} \gdshuffle \r{\omega}) + \r{\sigma} \gdproduct (\b{\pi} \gdshuffle \r{\tau}).
\]
This product extends to~$\HS$ by linearity and is clearly associative.
This is the standard shuffle product, but performed on the atomic factorizations of the factors.

\begin{example} 
 For~$\b{2431} = \b{132} \gdproduct \b{1} \in \fS_4$ and~$\r{312} = \r{1} \gdproduct \r{12} \in \fS_3$, we have 
\begin{align*}
\b{2431} \gdshuffle \r{312} & = \b{132} \gdproduct \b{1} \gdproduct \r{1} \gdproduct \r{12} + \b{132} \gdproduct \r{1} \gdproduct \b{1} \gdproduct  \r{12} + \b{132} \gdproduct  \r{1} \gdproduct \r{12} \gdproduct \b{1} \\
& \qquad + \r{1} \gdproduct \b{132} \gdproduct \b{1} \gdproduct \r{12} + \r{1} \gdproduct \b{132} \gdproduct \r{12} \gdproduct \b{1} + \r{1} \gdproduct \r{12} \gdproduct  \b{132} \gdproduct \b{1} \\
& = \b{576}\b{4}\r{3}\r{12} + \b{576}\r{4}\b{3}\r{12} + \b{576}\r{4}\r{23}\b{1} + \r{7}\b{465}\b{3}\r{12} + \r{7}\b{465}\r{23}\b{1} + \r{7}\r{56}\b{243}\b{1} \\
& = 2 \times 5764312 + 5764231 + 7465312 + 7465231 + 7562431.
\end{align*}
\end{example}

The product~$\gdproduct$ and the coproduct~$\gddeconcat$ are compatible in the sense that
\[
\gddeconcat(\b{\pi} \gdshuffle \r{\omega}) = \gddeconcat(\b{\pi}) \gdshuffle \gddeconcat(\r{\omega}),
\]
where the right hand side product has to be understood componentwise. This structure thus gives us a graded and connected commutative Hopf algebra~$(\HS, \gdshuffle, \gddeconcat)$. See~\cite{Vargas} for more details.
In the following sections we will see that this Hopf structure on permutations is directly linked to our Hopf structure on pipe dreams.


\section{A Hopf algebra on pipe dreams}
\label{sec:HopfAlgebraPipeDreams}

In this section, we construct a product and a coproduct on certain pipe dreams and show that they define a graded connected Hopf algebra. We also relate this structure to the Hopf algebra of Section~\ref{sec:HopfAlgebraPermutations}.


\subsection{Pipe dreams}
\label{subsec:pipeDreams}

A \defn{pipe dream}~$P$ is a filling of a triangular shape with crosses~\cross{} and elbows~\elbow{} so that all pipes entering on the left side exit on the top side~\cite{BergeronBilley, KnutsonMiller-GroebnerGeometry}. See \fref{fig:pipeDreams}. We only consider \defn{reduced} pipe dreams, where two pipes have at most one intersection. We also limit ourself to pipe dreams with an elbow~\elbow{} in the top left corner. In particular the pipe entering in the topmost row always exits in the leftmost column. We label this pipe with~$0$ and the other pipes with~$1, 2, \dots, n$ in the order of their entry points from top to bottom. We also label accordingly the rows and the columns of the pipe dream~$P$ from~$0$ to~$n$. We denote by~$\omega_P \in \fS_n$ the order of the exit points of the non-zero pipes of~$P$ from left to right. In other words, the pipe entering at row~$i > 0$ exits at column~$\omega^{-1}_P(i) > 0$. For a fixed permutation~$\omega \in \fS_n$, we denote by~$\pipeDreams(\omega)$ the set of reduced pipe dreams~$P$ with an elbow~\elbow{} in the top left corner and such that~$\omega_P = \omega$. We let~$\pipeDreams_n \eqdef \bigsqcup_{\omega \in \fS_n} \pipeDreams(\omega)$ and~$\pipeDreams \eqdef \bigsqcup_{n \in \N} \pipeDreams_n$. 

\begin{figure}[ht]
	\capstart
	\centerline{
		\input{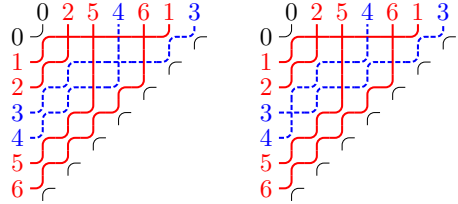}
	}
	\caption{Two pipe dreams of~$\pipeDreams(254613)$ connected by a flip exchanging an elbow with the crossing on the two dashed blue pipes~$3$ and~$4$.}
	\label{fig:pipeDreams}
\end{figure}

Two pipe dreams are related by a \defn{flip} if they define the same permutation of the pipes and only differ by the position of a cross and an elbow. An elbow~$e$ in a pipe dream~$P$ is \defn{flippable} if the two pipes passing through elbow~$e$ have a crossing~$c$, and the flip exchanges the elbow~$e$ with the cross~$c$. See \fref{fig:pipeDreams}. A \defn{chute move} (resp.~\defn{ladder move}) is a flip where~$e$ and~$c$ appear in consecutive rows (resp.~columns). We refer to~\cite{BergeronBilley, PilaudStump-ELlabelings} for thorough discussions on flips.

We consider the graded vector space~$\HP = \bigoplus_{n\ge 0} \HP_n$, where~$\HP_n$ is the~$\bk$-span of the pipe dreams in~$\pipeDreams_n$. 


\subsection{Horizontal and vertical packings}
\label{subsec:packings}

Let~$P \in \pipeDreams_n$ and~$k \in \{0,\dots, n\}$.
We color plain \r{red} the pipes entering in the rows $1, 2, \dots, k$ and dashed \b{blue} the pipes entering in the rows $k+1, k+2, \dots, n$.
The \defn{horizontal packing}~$\r{\lrot_k(P)}$ is obtained by removing all blue pipes and contracting the horizontal parts of the red pipes that were contained in a cell~\crossBiColor{red}{blue,dashed} (\ie red horizontal steps that are crossed vertically by a blue pipe).
\fref{fig:horizontalPackingPipeDream} illustrates an example. The following lemma shows that this operation is well defined.
\begin{figure}[ht]
	\capstart
	\centerline{
		\input{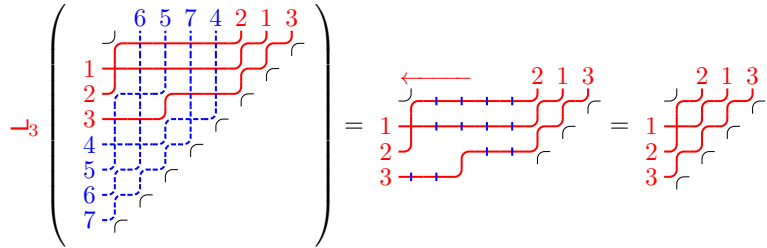}
	}
	\caption{Horizontal packing of a pipe dream at~$k = 3$.}
	\label{fig:horizontalPackingPipeDream}
\end{figure}
\begin{lemma}
\label{lem:horizontalPacking}
The horizontal packing~$\r{\lrot_k(P)}$ is a reduced pipe dream in $\pipeDreams_k$.
\end{lemma}

\begin{proof}
We will show that the resulting pipes in the horizontal packing fill a triangular shape with elbows and crossings. For this we analyze how the cells containing red pipes are moved after packing.  
First, note that blue pipes in $P$ can only cross red pipes vertically, and so, bicolored cells of the form~\crossBiColor{blue,dashed}{red} are forbidden. Furthermore, cells~\crossBiColor{red}{blue,dashed} are horizontally contracted and any other cell~\elbowBiColor{red}{red},~\elbowBiColor{red}{blue,dashed},~\elbowBiColor{blue,dashed}{red},~\crossBiColor{red}{red} is moved to the left as many steps as the number of blue pipes passing to its left. Therefore, ``consecutive" bicolored elbows~$\elbowBiColor{red}{blue,dashed} \; \cdots \elbowBiColor{blue,dashed}{red}$ in the same row (\ie with no red pipes in between them) are moved together:
\[
\elbowBiColor{red}{blue,dashed} \; \cdots \elbowBiColor{blue,dashed}{red} \quad \longrightarrow \quad \elbowBiColor{red}{red},
\]
and full red elbows and crossings are preserved.
As a consequence, the pipes in~$\r{\lrot_k(P)}$ fill a shape with elbows and crossings.
We claim that the final shape of~$\r{\lrot_k(P)}$ is triangular.
To see it, we count the number of elbows and crossings in its $i$th row for~$i \in [k]$.
Note that since the shape of~$P$ is triangular, the $i$th row of~$P$ contains $n-i+1$ elbows and crossings.
As each of the last~$n-k$ pipes contributes to either delete a cross~\crossBiColor{red}{blue,dashed}, a cross~\crossMonoColor{blue,dashed} or to merge~$\elbowBiColor{red}{blue,dashed} \; \cdots \elbowBiColor{blue,dashed}{red}$ into~$\elbowBiColor{red}{red}$, the number of elbows and crossings in the $i$th row of~$\r{\lrot_k(P)}$ is~$k-i+1$.  
Finally, since the crossings between red pipes in $P$ are preserved after packing, $\r{\lrot_k(P)}$ is a reduced pipe dream.
\end{proof}

The \defn{vertical packing}~$\b{\lmir_k(P)}$ is defined symmetrically: it keeps the pipes exiting through the columns~$1, 2, \dots, k$ and contracts the vertical steps that are crossed horizontally by the pipes exiting through the columns~$k+1, k+2, \dots, n$. Note that the remaining pipes need to be relabeled from~$1$ to~$k$. See \fref{fig:verticalPackingPipeDream}. The result is also a reduced pipe dream in $\pipeDreams_k$.
Although they are valid for any gap~$k$, we have illustrated the horizontal and vertical packings only at a global descent of~$\omega_P$ as we will only use this situation later.

\begin{figure}[ht]
	\capstart
	\centerline{
		\input{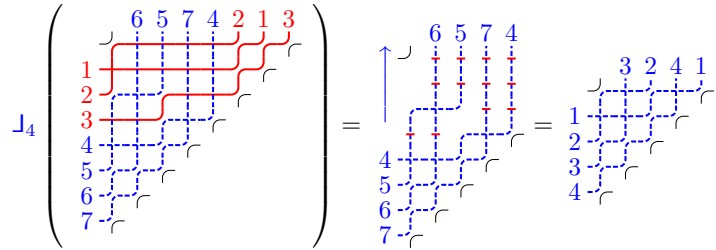}
	}
	\caption{Vertical packing of a pipe dream at~$k = 4$.}
	\label{fig:verticalPackingPipeDream}
\end{figure}


\subsection{Coproduct on pipe dreams}
\label{subsec:coproductPipeDreams}

Consider a reduced pipe dream~$P \in \pipeDreams_n$, and a global descent~$\gamma$ of the permutation~$\omega_P$.
Since~$\gamma$ is a global descent of~$\omega_P$, the relevant pipes of~$P$ are split into two disjoint tangled sets of pipes: those entering in the first~$n-\gamma$ rows (plain \r{red}) and those exiting in the first~$\gamma$ columns (dashed \b{blue}).
The horizontal and vertical packings~$\r{\lrot_{n-\gamma}(P)}$ and~$\b{\lmir_\gamma(P)}$ should thus be regarded as a way to untangle these two disjoint sets of pipes.

We denote by~$\coproduct_{\gamma,n-\gamma}(P)$ the tensor product~$\b{\lmir_\gamma(P)} \otimes \r{\lrot_{n-\gamma}(P)}$ and we say that~$\coproduct_{\gamma,n-\gamma}$ \defn{untangles}~$P$. This operation is illustrated on \fref{fig:untanglePipeDream}.

\begin{figure}[ht]
	\capstart
	\centerline{
		\input{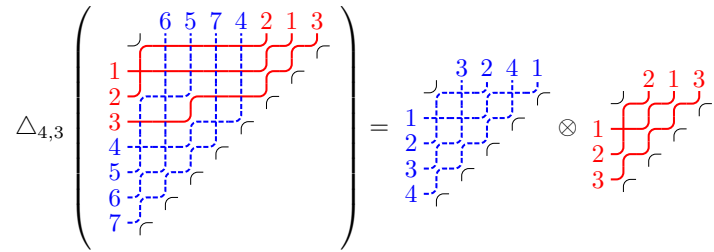}
	}
	\caption{Untangling a pipe dream~$P$ at the global descent~$\gamma = 4$ of~$\omega_P = 6574213$.}
	\label{fig:untanglePipeDream}
\end{figure}

If $\gamma$ is not a global descent of~$\omega_P$, then the pipes of~$P$ are not split by~$\gamma$, and we therefore define $\coproduct_{\gamma,n-\gamma}(P) = 0$.
Finally, we define the coproduct on~$\pipeDreams$ as
\[
\coproduct \eqdef \sum_{m,n \in \N} \coproduct_{m,n}.
\]
See \fref{fig:coproductPipeDreams}.
Note that the non-zero terms of~$\coproduct(P)$ occur when~$\gamma$ ranges over all global descents of~$\omega_P$.
\begin{figure}[t]
	\capstart
	\centerline{
		\input{coproductPipeDreams}
	}
	\caption{Coproduct of a pipe dream.}
	\label{fig:coproductPipeDreams}
\end{figure}
Extended by linearity on~$\HP$, the map~$\coproduct : \HP \to \HP \otimes \HP$ defines a graded comultiplication~on~$\HP$.

\begin{proposition}
\label{prop:coalgebra}
The coproduct~$\coproduct$ defines a coassociative graded coalgebra structure on $\HP$.
\end{proposition}

\begin{proof}
By construction, $\coproduct$ is graded. 
The map~$\epsilon : \HP \to \bk$ determined by~$\epsilon(P) = 1$ if~${P=\elbow \in \pipeDreams_0}$ and ${\epsilon(P) = 0}$ otherwise is the counit for~$\coproduct$. 
Coassociativity follows from the fact that both ${({\bf 1} \otimes \coproduct) \coproduct(P)}$ and ${(\coproduct \otimes {\bf 1})\coproduct(P)}$ are equal to~$P$ untangled at two global descents.
\end{proof}

\begin{proposition}
\label{prop:codescentproduct}
The map $\omega \colon (\HP, \coproduct) \to (\HS, \gddeconcat)$ is a graded morphism of coalgebras.
\end{proposition}
 
\begin{proof} 
Let~$P \in \pipeDreams_n$ and~$\gamma$ be a global descent of $\omega_P$. Let~$P_1 \otimes P_2 = \coproduct_{\gamma,n-\gamma}(P)$.
We have
\[
\omega_P = \omega_{P_1} \gdproduct \omega_{P_2}.
\]
Indeed, when we untangle~$P$ at~$\gamma$ we preserve the crossings of the pipes~$1, 2, \dots, n-\gamma$ and the crossings of the pipes~$n-\gamma+1, n-\gamma+2, \dots, n$. This shows that~$\omega_P(i) = \omega_{P_2}(i)$ for~$1 \le i \le n-\gamma$ and~$\omega_P(j) = \omega_{P_1}(j-n+\gamma) + n - \gamma$ for~$n-\gamma+1 \le j \le n$. Since we sum over all global descents of~$\omega_P$, we obtain the desired result. It is also clear that~$\omega$ preserves the grading.
\end{proof}


\subsection{Product on pipe dreams}
\label{subsec:productPipeDreams}

Let~$\b{P} \in \pipeDreams_\b{m}$ and~$\r{Q} \in \pipeDreams_\r{n}$ be two pipe dreams, and let~$\b{\gamma}$ be a global descent of~$\omega_\b{P}$. We denote by~$\b{P} \gammainsertion_\b{\gamma} \r{Q}$ the pipe dream of~$\pipeDreams_{\b{m}+\r{n}}$ obtained from~$\b{P}$ by
\begin{itemize}
\item inserting~$\r{n}$ columns after column~$\b{\gamma}$ and~$\r{n}$ rows after row~$\b{m}-\b{\gamma}$,
\item filling with~$\r{Q}$ the triangle of boxes located both in one of the rows~$\gamma, \dots, n+\gamma$ and in one of the columns~$m-\gamma, \dots, m+n-\gamma$,
\item filling with crosses~\crossBiColor{blue,dashed}{red} the remaining boxes located in a new column, 
\item filling with crosses~\crossBiColor{red}{blue,dashed} the remaining boxes located in a new row. 
\end{itemize}
We say that~$\gammainsertion_\b{\gamma}$ \defn{inserts}~$\r{Q}$ at gap~$\b{\gamma}$ in~$\b{P}$. This operation is illustrated in \fref{fig:pipeDreamsInsertion}.

\begin{figure}[ht]
	\capstart
	\centerline{
		\input{insertionPipeDreams}
	}
	\caption{Inserting a pipe dream~$\r{Q}$ in a pipe dream~$\b{P}$ at the global descent~$4$ of~$\omega_{\b{P}} = 635412$.}
	\label{fig:pipeDreamsInsertion}
\end{figure}

Note that the result~$\b{P} \gammainsertion_\b{\gamma} \r{Q}$ is a reduced pipe dream in~$\pipeDreams_{\b{m}+\r{n}}$.
Indeed, since~$\b{\gamma}$ is a global descent of~$\omega_{\b{P}}$, the pipes of~$\b{P}$ entering in the last~$\b{\gamma}$ rows and exiting through the first~$\b{\gamma}$ columns all cross the pipes of~$\r{Q}$ only once vertically, while the pipes of~$\b{P}$ entering in the first~$\b{m}-\b{\gamma}$ rows and exiting through the last~$\b{m}-\b{\gamma}$ columns all cross the pipes of~$\r{Q}$ only once horizontally. 
It is crucial here that~$\b{\gamma}$ is a global descent of~$\omega_{\b{P}}$, as otherwise there would exist~$i \in [\b{m}]$ such that~$i > \b{m}-\b{\gamma}$ and~$\omega_{\b{P}}^{-1}(i) > \b{\gamma}$ and pipe~$i+\r{n}$ in~$\b{P} \gammainsertion_\b{\gamma} \r{Q}$ would cross twice each of the new pipes~$\b{m}+1-\gamma, \dots, \b{m}+\r{n}-\gamma$.
The following immediate lemma is left to the reader.

\begin{lemma}
\label{lem:rulesInsertion}
For pipe dreams~$\b{P} \in \pipeDreams_\b{m}$, $\r{Q} \in \pipeDreams_\r{n}$ and~$\g{R} \in \pipeDreams_\g{p}$, and for global descents~$\gamma < \delta$ of~$\omega_\b{P}$ and $\nu$ of~$\omega_\r{Q}$ , we have
\begin{align}
\label{eq:ruleInsertionA}
& (\b{P} \gammainsertion_\gamma \r{Q}) \gammainsertion_{\delta + \r{n}} \g{R} = (\b{P} \gammainsertion_\delta \g{R}) \gammainsertion_\gamma \r{Q}, \\
\label{eq:ruleInsertionB}
\text{and}\qquad
& (\b{P} \gammainsertion_\gamma \r{Q}) \gammainsertion_{\gamma + \nu} \g{R} = \b{P} \gammainsertion_\gamma (\r{Q} \gammainsertion_\nu \g{R}).
\end{align}
\end{lemma}

Finally, we observe in the next statement that the operation~$\gammainsertion$ behaves nicely with the product on permutations defined in Section~\ref{sec:HopfAlgebraPermutations}. The immediate proof is left to the reader.

\begin{lemma}
\label{lem:gbinsertion}
Let~$\b{P} \in \pipeDreams_\b{m}$ and~$\r{Q} \in \pipeDreams_\r{n}$ be two pipe dreams, such that~$\omega_{\b{P}} = \omega_1 \gdproduct \omega_2$ for~$\omega_1 \in \fS_{\b{\gamma}}$ and~$\omega_2 \in \fS_{\b{m}-\b{\gamma}}$. Then~$\omega_{(\b{P} \gammainsertion_\b{\gamma} \r{Q})} = \omega_1 \gdproduct \omega_{\r{Q}} \gdproduct \omega_2$.
\end{lemma}

Consider now a word~$s$ on the alphabet~$\{\b{p}, \r{q}\}$ with~$\b{m}$ letters~$\b{p}$ and~$\r{n}$ letters~$\r{q}$. We call \defn{$\b{p}$-blocks} (resp.~\defn{$\r{q}$-blocks}) the blocks of consecutive letters~$\b{p}$ (resp.~$\r{q}$). We consider that the $\b{p}$-blocks in~$s$ mark gaps of the permutation~$\omega_{\r{Q}}$ while the $\r{q}$-blocks in~$s$ mark gaps of the permutation~$\omega_{\b{P}}$. We say that~$s$ is a $\b{P}/\r{Q}$-\defn{shuffle} if all $\b{p}$-blocks appear at global descents of~$\omega_\r{Q}$ while all $\r{q}$-blocks appear at global descents of~$\omega_\b{P}$. For example, if~$\omega_\b{P} = 53412= 1 \gdproduct 12 \gdproduct 12$ and~$\omega_\r{Q} = 635421 = 1 \gdproduct 132 \gdproduct 1 \gdproduct 1$, then~$\r{q}\b{p}\r{qqq}\b{pp}\r{q}\b{pp}\r{q}$ and~$\r{q}\b{ppp}\r{qqqqq}\b{pp}$ are~$\b{P}/\r{Q}$-shuffles, while $\r{qqq}\b{pp}\r{qqq}\b{ppp}$ is not.

For a $\b{P}/\r{Q}$-shuffle~$s$ with~$\ell$ $\r{q}$-blocks, we denote by~$\b{P} \star_{s} \r{Q}$ the pipe dream obtained by
\begin{itemize}
\item untangling~$\r{Q}$ at all gaps of~$\omega_\r{Q}$ marked by $\b{p}$-blocks in~$s$, resulting to~$\ell$ pipe dreams~$\r{Q_1}, \dots, \r{Q_\ell}$,
\item inserting succesively the pipe dreams~$\r{Q_1}, \dots, \r{Q_\ell}$ in~$\b{P}$ at the positions of the $\r{q}$-blocks in~$s$ (by Lemma~\ref{lem:rulesInsertion}\,\eqref{eq:ruleInsertionA} it does not matter in which order we insert the pipe dreams~$\r{Q_1}, \dots, \r{Q_\ell}$).
\end{itemize}
We say that~$\star_{s}$ \defn{tangles} the pipe dreams~$\b{P}$ and~$\r{Q}$ according to the $\b{P}/\r{Q}$-shuffle~$s$. This operation is illustrated in \fref{fig:pipeDreamsTangle}. 

\begin{figure}[ht]
	\capstart
	\centerline{
		\input{tanglePipeDreams}
	}
	\caption{Tangling two pipe dreams~$\b{P}$ and~$\r{Q}$ according to the $\b{P}/\r{Q}$-shuffle~$\b{p}\r{qqq}\b{pp}\r{q}\b{pp}\r{qq}$.}
	\label{fig:pipeDreamsTangle}
\end{figure}

We define the product of~$\b{P} \in \pipeDreams_\b{m}$ and~$\r{Q} \in \pipeDreams_\r{n}$ by
\[
\b{P} \product \r{Q} = \sum_{s}  \b{P} \star_{s} \r{Q},
\]
where~$s$ ranges over all possible $\b{P}/\r{Q}$-shuffles. See \fref{fig:productPipeDreams}. Note that it always contains~${\b{P} \gammainsertion_0 \r{Q}}$ and~${\b{P} \gammainsertion_\b{m} \r{Q}}$ corresponding to the $\b{P}/\r{Q}$-shuffles~$\r{q^n}\b{p^m}$ and~$\b{p^m}\r{q^n}$.

\begin{figure}[ht]
	\capstart
	\centerline{
		\input{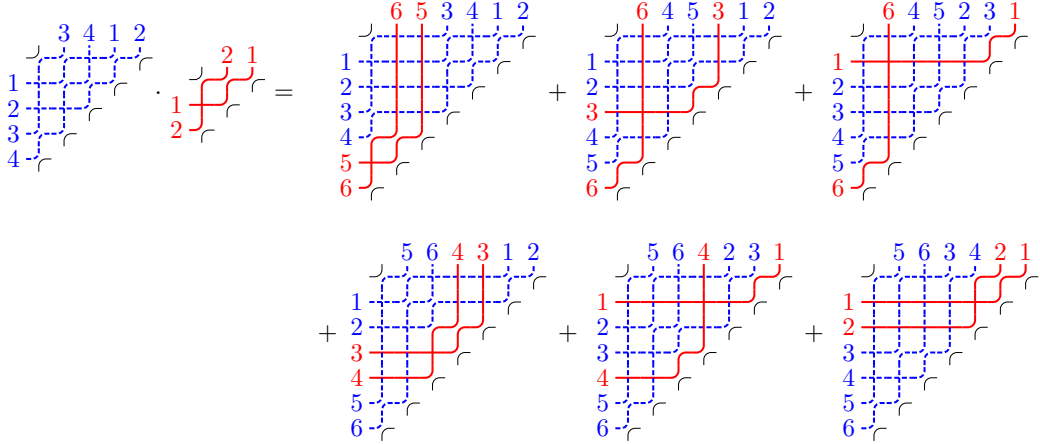}
	}
	\caption{Product of two pipe dreams.}
	\label{fig:productPipeDreams}
\end{figure}

Extended by bilinearity on~$\HP \otimes \HP$, the map~$\product \colon \HP \otimes \HP \to \HP$ defines a graded multiplication~on~$\HP$.

\begin{proposition}
The product~$\product$ defines an associative graded algebra structure on~$\HP$.
\end{proposition}

\begin{proof}
By construction, $\product$ is graded. It is also clear that~$P\product\elbow = \elbow\product P = P$. Associativity follows from the fact that the terms of $(\b{P}\product\r{Q})\product\g{R}$ and $\b{P}\product(\r{Q}\product\g{R})$ both correspond to $\b{P}/\r{Q}/\g{R}$-shuffles and the shuffle operation is associative. We also make use of the fact that it does not matter in which order we insert pipe dreams in $\b{P}$ according to Lemma~\ref{lem:rulesInsertion}\,\eqref{eq:ruleInsertionA}. Thus, in both~$(\b{P}\product\r{Q})\product\g{R}$ and~$\b{P}\product(\r{Q}\product\g{R})$ we will insert the pieces of $\r{Q}$ and $\g{R}$ as dictated by the $\b{P}/\r{Q}/\g{R}$-shuffles, so that the result will be the same by Lemma~\ref{lem:rulesInsertion}\,\eqref{eq:ruleInsertionB}.
\end{proof}

\begin{proposition}
\label{prop:descentproduct}
The map $\omega\colon (\HP, \product) \to (\HS, \gdshuffle)$ is a graded morphism of algebras.
\end{proposition}

\begin{proof}  
Let~$\b{P} \in \pipeDreams_\b{m}$, $\r{Q} \in \pipeDreams_\r{n}$ and~$s$ be a~$\b{P}/\r{Q}$-shuffle. Then~$s$ corresponds to exactly one term in~${\omega_{\b{P}} \gdshuffle \omega_{\r{Q}}}$. 
Moreover, successive use of Lemma~\ref{lem:gbinsertion} shows that~$\omega_{(\b{P} \star_{s} \r{Q})}$ is the term in the shuffle corresponding to~$s$. This shows that~$\omega_{\b{P}} \gdshuffle \omega_{\r{Q}} = \omega_{(\b{P} \product \r{Q})}$. It is also clear that~$\omega$ preserves the grading.
\end{proof}


\subsection{Hopf Algebra Structure}
\label{subsec:HopfAlgebraPipeDreams}

We now show that the product and coproduct on pipe dreams defined in the previous sections are compatible in the sense of the following statement.

\begin{proposition}
\label{prop:HopfAlgebraFixedDepth}
The product~$\product$ and coproduct~$\coproduct$ endow the family~$\pipeDreams$ of all pipe dreams with a graded connected Hopf algebra structure.
\end{proposition}

\begin{proof}
Consider two pipe dreams~$\b{P} \in \pipeDreams_\b{m}$ and~$\r{Q} \in \pipeDreams_\r{n}$. We need to show that
\[
\coproduct(\b{P}\product\r{Q}) = \coproduct(\b{P}) \product \coproduct(\r{Q}),
\]
where the right hand side product has to be understood componentwise. For that, it is enough to show that each term on the left hand side exactly corresponds to one term on the right hand side.

The terms of~$\coproduct(\b{P}\product\r{Q})$ are in correspondence with all choices of a $\b{P}/\r{Q}$-shuffle~$s$ and of a global descent~$\gamma \in \{0, \dots, \b{m} + \r{n}\}$ of~$\omega_{(\b{P} \star_s \r{Q})}$. This global descent uniquely identifies global descents~$\gamma_{\b{P}} \eqdef |\set{i \in [\gamma]}{s_i = \b{p}}|$ of~$\omega_\b{P}$ and~$\gamma_{\r{Q}} \eqdef |\set{i \in [\gamma]}{s_i = \r{q}}|$ of~$\omega_\r{Q}$, and hence defines terms~$\b{P'} \otimes \b{P''}$ of~$\coproduct(\b{P})$ and~$\r{Q'} \otimes \r{Q''}$ of~$\coproduct(\r{Q})$. Moreover~$s' \eqdef s_1 \dots s_\gamma$ is a~$\b{P'}/\r{Q'}$-shuffle and hence corresponds to a unique term of~$\b{P'} \product \r{Q'}$ while~$s'' \eqdef s_{\gamma + 1} \dots s_{\b{m}+\r{n}}$ is a $\b{P''}/\r{Q''}$-shuffle and hence corresponds to a unique term of~$\b{P''} \product \r{Q''}$. This shows that each term of~$\coproduct(\b{P} \product \r{Q})$ appears as a term of~$\coproduct(\b{P}) \product \coproduct(\r{Q})$.

Conversely, the terms of~$\coproduct(\b{P}) \product \coproduct(\r{Q})$ are in correspondence with all choices of a global descent~$\gamma_{\b{P}}$ of~$\omega_\b{P}$ corresponding to a term~$\b{P'} \otimes \b{P''}$ of~$\coproduct(\b{P})$ and a global descent~$\gamma_{\r{Q}}$ of~$\omega_\r{Q}$ corresponding to a term~$\r{Q'} \otimes \r{Q''}$ of~$\coproduct(\r{Q})$, together with a $\b{P'}/\r{Q'}$-shuffle~$s'$ and a $\b{P''}/\r{Q''}$-shuffle~$s''$. Then~$s \eqdef s's''$ is a $\b{P}/\r{Q}$-shuffle which gives us a term of~$\coproduct(\b{P} \product \r{Q})$.
\end{proof}

Combining Propositions~\ref{prop:descentproduct} and~\ref{prop:codescentproduct}, we obtain the following statement.

\begin{proposition}
\label{prop:descentHopf}
The map~$\omega \colon (\HP, \product, \coproduct) \to (\HS, \gdshuffle, \gddeconcat)$ is a morphism of graded Hopf algebras.
\end{proposition}

\begin{remark}
We can refine the grading of both~$\HP$ and~$\HS$ so that it remains preserved by~$\omega$.
For~${P \in \pipeDreams_n}$ and~$\omega_P = \nu_1 \gdproduct \cdots \gdproduct \nu_\ell$ the factorisation into atomic permutations, we have two particularly interesting refinements:
\begin{itemize}
\item the number~$\ell-1$ of global descents,
\item the sum~$\displaystyle \sum_{i=1}^\ell \ell(\nu_i)$ of inversions within atomic permutations.
\end{itemize}
\end{remark}

\begin{remark}
The transposition~$\transpose$ provides a Hopf algebra involution on~$\HP$:
\[
\transpose \circ \product = \product \circ (\transpose \otimes \transpose)
\qquad\text{and}\qquad
\coproduct \circ \transpose = (\transpose \otimes \transpose) \circ \coproduct.
\]
Moreover, given~$\omega_P = \nu_1 \gdproduct \nu_2 \gdproduct \cdots \gdproduct \nu_\ell$ we have
\[
\omega_{\transpose(P)} = \omega_P^{-1} = \nu_\ell^{-1} \gdproduct \cdots \gdproduct \nu_2^{-1} \gdproduct \nu_1^{-1}.
\]
\end{remark}

\begin{remark}
Before we close this section, there is much more that can be done in studying the Hopf algebra  $\HP$. In particular, it would be interesting to study the quasi-symmetric invariants~\cite{AguiarBergeronSottile} associated to 
$\HP$ and its related odd subalgebras. This will be the focus of other projects.
\end{remark}


\section{Freeness of $\HP$}
\label{sec:freeness}

In this section we show that~$(\HP, \product, \coproduct)$ is a free and cofree Hopf algebra and we identify explicitly a set of free generators and cogenerators. This is done by introducing an order on pipe dreams such that the leading term of~${\b{P} \product \r{Q}}$ is given by~${\b{P} \gammainsertion_0 \r{Q}}$ (where~$\gammainsertion_0$ is the insertion in position~$0$, defined in Section~\ref{subsec:productPipeDreams}). More generally, for any~$\b{\Phi}$ and~$\r{\Psi}$ in~$\HP$, our order satisfies that the leading term of the product~${\b{\Phi} \product \r{\Psi}}$ is the product of the leading terms of~$\b{\Phi}$ and~$\r{\Psi}$. The analogue statement is valid for the dual Hopf algebra as well. The (free and cofree) generators are then identified with the indecomposable pipe dreams~$P$ with respect to the operation~$\gammainsertion_0$. We also give in this section a good combinatorial description of the indecomposable pipe dreams.

An alternative approach would be to introduce a bidendriform structure on $\HP$ as in~\cite{Foissy}. This is certainly possible but requires to introduce more algebraic structures. We choose to follow the strategy above as it gives us some combinatorial understanding of the indecomposable pipe dreams.


\subsection{Lexicographic order on pipe dreams and leading term}
\label{subsec:order}

Consider a pipe dream~${P \in \pipeDreams_n}$, and index its rows and columns by~$0, 1, \dots, n$ as before. We denote by~$p_{i,j} \in \{\elbowLetter, \crossLetter\}$ the symbol (elbow or cross) that appears in row~$i$ and column~$j$ of~$P$.
The \defn{column reading word} of~$P$ is the word~$\columnReading(P)$ defined by
\[
\columnReading(P) \eqdef p_{0,0} \, p_{1,0} \cdots p_{n-1,0} p_{n,0} \, | \, p_{0,1} \, p_{1,1} \cdots p_{n-2,1} p_{n-1,1} \, | \, \cdots \, | \, p_{0,n-1} p_{1,n-1} \, | \, p_{0,n}.
\]
Note that~$p_{0,0} = p_{n-i,i} = \elbowLetter$ for all~$0 \le i \le n$.
We order the alphabet~$\{\elbowLetter, \crossLetter\}$ by~$\elbowLetter < \crossLetter$.
Given two pipe dreams $P,Q \in \pipeDreams_n$ we write that~$P \ltlex Q$ if and only if~$\columnReading(P) \ltlex \columnReading(Q)$ in lexicographic order.
For example, the column reading words of the pipe dreams~$P$ (left) and~$Q$ (right) of~\fref{fig:pipeDreams} are
\begin{align*}
\columnReading(P) & = \elbowLetter\elbowLetter\elbowLetter\elbowLetter\elbowLetter\elbowLetter\elbowLetter \, | \, \crossLetter\elbowLetter\elbowLetter\elbowLetter\elbowLetter\elbowLetter \, | \, \crossLetter\crossLetter\crossLetter\elbowLetter\elbowLetter \, | \, \crossLetter\crossLetter\elbowLetter\elbowLetter \, | \, \crossLetter\crossLetter\elbowLetter \, | \, \elbowLetter\elbowLetter \, | \, \elbowLetter \\
\columnReading(Q) & = \elbowLetter\elbowLetter\elbowLetter\crossLetter\elbowLetter\elbowLetter\elbowLetter \, | \, \crossLetter\elbowLetter\elbowLetter\elbowLetter\elbowLetter\elbowLetter \, | \, \crossLetter\crossLetter\crossLetter\elbowLetter\elbowLetter \, | \, \crossLetter\elbowLetter\elbowLetter\elbowLetter \, | \, \crossLetter\crossLetter\elbowLetter \, | \, \elbowLetter\elbowLetter \, | \, \elbowLetter
\end{align*}
and we thus have~$P \ltlex Q$.

\begin{definition}
Given~$\Phi \in \HP_n$, we can order the terms of~$\Phi$ with respect to~$\ltlex$ 
\[
\Phi = \sum_{i=1}^m c_i P_i
\]
where~$c_i \ne 0$ and~$P_i \ltlex P_j$ for any~$1 \le i < j \le m$. We then define the \defn{leading term} of~$\Phi$ to be~$\leadingTerm(\Phi) = c_1 P_1$, the \defn{leading coefficient} of~$\Phi$ to be~$\leadingCoefficient(\Phi) = c_1$ and the \defn{leading pipe dream} of~$\Phi$ to be~$\leadingPipeDream(\Phi) = P_1$.
\end{definition}


\subsection{Leading term of a product}
\label{subsec:leadingTermProduct}

With the lexicographic order on pipe dreams, we are now in position to identify the leading term of a product.

\begin{lemma}
\label{lem:LTproductroduct}
For any pipe dreams~$\b{P} \in \pipeDreams_{\b{m}}$ and~$\r{Q} \in \pipeDreams_{\r{n}}$ we have
\(
\leadingTerm({\b{P} \product \r{Q}})={\b{P} \gammainsertion_0 \r{Q}}.
\)
\end{lemma}
  
\begin{proof}
It is enough to consider the leftmost column of every term in the product~${\b{P} \product \r{Q}}$. Given a $\b{P}/\r{Q}$-shuffle~$s$, let~$\overleftarrow{s}$ be the mirror of~$s$, that is, read from right to left. Observe that the leftmost column of the term~$\b{P} \star_{s} \r{Q}$ is structured as follows:
\begin{itemize}
\item There is an elbow~\elbow{} in position~$(0,0)$.
\item The entries of the leftmost column of~$\b{P}$ in rows~$1, 2, \dots, n$, appear in the positions of the letters~$\b{p}$ in~$\overleftarrow{s}$. The rightmost~$\b{p}$ in~$\overleftarrow{s}$ corresponds to an elbow~\elbow{}.
\item There is a cross~\cross{} in the positions corresponding to each letter~$\r{q}$ in $\overleftarrow{s}$ with at least one letter~$\b{p}$ to its right.
\end{itemize} 
The $\b{P}/\r{Q}$-shuffle corresponding to~${\b{P} \gammainsertion_0 \r{Q}}$ is~$s_0 = \r{q^n} \b{p^m}$. Hence $\overleftarrow{s}_0 = \b{p^m} \r{q^n}$ and we have that
$\columnReading({\b{P} \gammainsertion_0 \r{Q}}) = \elbowLetter \b{p_{1,0} \cdots p_{n-1,0}} \elbowLetter \cdots$ where $\columnReading(\b{P}) = \elbowLetter \b{p_{1,0} \cdots p_{n-1,0}}\elbowLetter \cdots$. Now for any distinct $\b{P}/\r{Q}$-shuffle~$s \ne s_0$, the mirror~$\overleftarrow{s}$ will have at least one letter~$\r{q}$ with a letter~$\b{p}$ to its right. In this case $\columnReading(\b{P} \star_{s} \r{Q}) = \elbowLetter \b{p_{1,0} \cdots} \crossLetter \b{\cdots p_{n-1,0}} \elbowLetter \cdots$ with a $\crossLetter$ for every letter~$\r{q}$ with a letter~$\b{p}$ to its right in~$\overleftarrow{s}$. Thus $\columnReading({\b{P} \gammainsertion_0 \r{Q}}) \ltlex \columnReading(\b{P} \star_{s} \r{Q})$, and therefore~$(\b{P} \gammainsertion_0 \r{Q}) \ltlex (\b{P} \star_{s} \r{Q})$ for any~$s \ne s_0$. This shows that $\leadingTerm({\b{P} \product \r{Q}}) = {\b{P} \gammainsertion_0 \r{Q}}$.
\end{proof}  

\begin{example} 
Consider the six terms~$P_1, P_2, \dots, P_6$ in the expansion of the product in~\fref{fig:productPipeDreams}, in order of appearance. Look at their leftmost column and consider the corresponding $\b{P}/\r{Q}$-shuffles. The pipe dream~$P_1$ is the leading term obtained with~$s_0 = \r{qq}\b{pppp}$. In the table below, we give for each term~$P_i$ the corresponding $\b{P}/\r{Q}$-shuffle~$s$, its reverse~$\overleftarrow{s}$ and the initial subword of the column reading~$\columnReading(P_i)$ up to the position of the last~$\b{p}$ in~$\overleftarrow{s}$. We see that $P_1 \ltlex P_i$ is decided in the leftmost column before the last~$\b{p}$ in~$\overleftarrow{s}$.
\medskip
\[
\begin{array}{c|c|c|l}
	\text{term} & s & \overleftarrow{s} & \qquad \columnReading(P_i) \\
	\hline
	P_1 & \r{qq}\b{pppp} 			& \b{pppp}\r{qq} 			& \elbowLetter\b{\crossLetter\crossLetter\elbowLetter\elbowLetter}\cdots				\\
	P_2 & \r{q}\b{pp}\r{q}\b{pp} 	&\b{pp}\r{q}\b{pp}\r{q} 	& \elbowLetter\b{\crossLetter\crossLetter}\r{\crossLetter}\b{\elbowLetter\elbowLetter}\cdots		\\
	P_3 & \r{q}\b{pppp}\r{q} 		&\r{q}\b{pppp}\r{q} 		& \elbowLetter\r{\crossLetter}\b{\crossLetter\crossLetter}\b{\elbowLetter\elbowLetter}\cdots		\\
	P_4 & \b{pp}\r{qq}\b{pp} 		&\b{pp}\r{qq}\b{pp}			& \elbowLetter\b{\crossLetter\crossLetter}\r{\crossLetter\crossLetter}\b{\elbowLetter\elbowLetter}\cdots		\\
	P_5 & \b{pp}\r{q}\b{pp}\r{q} 	&\r{q}\b{pp}\r{q}\b{pp} 	& \elbowLetter\r{\crossLetter}\b{\crossLetter\crossLetter}\r{\crossLetter}\b{\elbowLetter\elbowLetter}\cdots	\\
	P_5 & \b{pppp}\r{qq} 			&\r{qq}\b{pppp}				& \elbowLetter\r{\crossLetter\crossLetter}\b{\crossLetter\crossLetter\elbowLetter\elbowLetter}\cdots			\\
\end{array}
\]
\smallskip
\end{example}

\begin{lemma}
\label{lem:order_and_zeroinsert}
For~$\b{P},\b{P'}\in\pipeDreams_{\b{m}}$ and~$\r{Q},\r{Q'}\in\pipeDreams_{\r{n}}$, if~$\b{P} \ltlex \b{P'}$ and~$\r{Q} \ltlex \r{Q'}$, then
\[
(\b{P} \gammainsertion_0 \r{Q}) \ltlex (\b{P'} \gammainsertion_0 \r{Q'}).
\]
\end{lemma}
  
\begin{proof}
Let 
\[
\begin{array}{l@{\qquad}l}
\columnReading(\b{P}) = \b{P_0 \cdots P_m}, & \columnReading(\b{P'}) = \b{P'_0 \cdots P'_m}, \\
\columnReading(\r{Q}) = \r{Q_0 \cdots Q_n}, & \columnReading(\r{Q'}) = \r{Q'_0 \cdots Q'_n},
\end{array}
\]
where~$\b{P_j} \eqdef \b{p_{0,j} \, p_{1,j} \cdots p_{m-j,j}}$ denote the top to bottom reading of the $j$th column of~$P$ (and similarly for~$\b{P'_j}$, $\r{Q_j}$ and~$\r{Q'_j}$).
Then we have
\begin{align*}
\columnReading({\b{P} \gammainsertion_0 \r{Q}}) & = \b{P_0} \r{Q_0} \, | \, \crossLetter^{\b{m}} \r{Q_1} \, | \, \crossLetter^{\b{m}} \r{Q_2} \, | \, \cdots \, | \, \crossLetter^{\b{m}} \r{Q_n}  \, | \,\b{P_1} \, | \, \cdots \, | \, \b{P_m}, \\
\columnReading({\b{P'} \gammainsertion_0 \r{Q'}}) & = \b{P'_0} \r{Q'_0} \, | \, \crossLetter^{\b{m}} \r{Q'_1} \, | \, \crossLetter^{\b{m}} \r{Q'_2} \, | \, \cdots \, | \, \crossLetter^{\b{m}} \r{Q'_n} \, | \, \b{P'_1} \, | \, \cdots \, | \, \b{P'_m},
\end{align*}
where~$\crossLetter^{\b{m}} = \crossLetter\crossLetter \cdots \crossLetter$ repeated~$\b{m}$ times and these~$\crossLetter$'s correspond to the rectangle of crosses~\cross{} of~${\b{P} \gammainsertion_0 \r{Q}}$ (or~${\b{P'} \gammainsertion_0 \r{Q'}}$) in rows~$0 \le i \le m-1$ and columns~$1 \le j \le n$. We clearly see that when comparing~$\columnReading({\b{P} \gammainsertion_0 \r{Q}})$ and~$\columnReading({\b{P'} \gammainsertion_0 \r{Q'}})$ in lexicographic order the decision will be done for the leftmost entry that is different in both words. That will be either the leftmost entry that is different in~$\columnReading(\b{P})$ and~$\columnReading(\b{P'})$ or the leftmost entry that is different in~$\columnReading(\r{Q})$ and~$\columnReading(\r{Q'})$. The result follows.
\end{proof}

Combining Lemmas~\ref{lem:LTproductroduct} and~\ref{lem:order_and_zeroinsert} we obtain the following result.

\begin{proposition}
\label{prop:LT}
For~$\b{\Phi}, \r{\Psi} \in \HP$, we have
\[
\leadingPipeDream({\b{\Phi} \product \r{\Psi}}) = \leadingPipeDream(\b{\Phi}) \gammainsertion_0 \leadingPipeDream(\r{\Psi})
\qquad\text{and}\qquad
\leadingCoefficient({\b{\Phi} \product \r{\Psi}}) = \leadingCoefficient(\b{\Phi}) \leadingCoefficient(\r{\Psi}).
\]
\end{proposition}
 

\subsection{(Free) generators of~$\HP$}
\label{subsec:generators}

We are now in position to show that~$\HP$ is free and to give a first characterization of its free generators.

\begin{definition}
A pipe dream~$P \in \pipeDreams_n$ is \defn{$\gammainsertion_0$-decomposable} if it can be written as~$P = Q \gammainsertion_0 R$ for some~${Q \in \pipeDreams_r}$ and~$R \in \pipeDreams_{n-r}$ where~${0 < r < n}$. Otherwise, we say~$P$ is \defn{$\gammainsertion_0$-indecomposable}. We let
\[
\indecomposablePipeDreams_n \eqdef \set{P \in \pipeDreams_n}{\text{$P$ is $\gammainsertion_0$-indecomposable}}
\qquad\text{and}\qquad
\indecomposablePipeDreams \eqdef \bigcup_{n \in \N} \indecomposablePipeDreams_n.
\]
\end{definition}
 
\begin{proposition}
\label{prop:decomposition}
A pipe dream~$P \in\pipeDreams_n$ is $\gammainsertion_0$-decomposable if and only if there is~$0 < r < n$ such that~$P$ has crosses~\cross{} in the rectangle of rows~$0 \le i \le r-1$ and columns~$1 \le j \le n-r$. In that case, $P = Q \gammainsertion_0 R$ where~$Q \otimes R = \coproduct_{r,n-r}(P)$.
\end{proposition}

\begin{proof}
It follows directly from the description of the insertion~$Q \gammainsertion_0 R$.
\end{proof}

\begin{lemma}
\label{lem:decomposition}
For any~$P\in\pipeDreams_n$, there is a unique $\gammainsertion_0$-decomposition~${P = P_1 \gammainsertion_0 P_2 \gammainsertion_0 \cdots \gammainsertion_0 P_\ell}$ of~$P$ into $\gammainsertion_0$-indecomposable pipe dreams.
\end{lemma}

\begin{proof}
To prove existence, we decompose~$P$ repeatedly until we only have $\gammainsertion_0$-indecomposable pipe dreams. To prove uniqueness, we proceed by induction on~$n$. For~$n = 0$ or $1$ the statement is immediate. Assume now that we have two decompositions
\[
P = P_1 \gammainsertion_0 P_2\gammainsertion_0 \cdots \gammainsertion_0 P_\ell  = P'_1 \gammainsertion_0 P'_2 \gammainsertion_0 \cdots \gammainsertion_0 P'_{\ell'}
\]
with~$\ell \le \ell'$. If~$\ell = 1$, then we must have~$\ell' = 1$ for otherwise it would contradict the fact that~$P_1$ is $\gammainsertion_0$-indecomposable. The uniqueness follows in this case. If~$\ell > 1$, then $P = P_1 \gammainsertion_0 Q = P'_1\gammainsertion_0 Q'$ where~$P_1 \in \pipeDreams_r$ and~$P'_1 \in \pipeDreams_{r'}$ for~$0 < r, r' < n$. By Proposition~\ref{prop:decomposition}, $P$ has crosses~\cross{} in the rectangle of rows~${0 \le i \le r-1}$ and columns~${1 \le j \le n-r}$ and also in the rectangle of rows~$0 \le i \le r'-1$ and columns~$1 \le j \le n-r'$. Moreover~${P_1 \otimes Q = \coproduct_{r,n-r}(P)}$ and~${P'_1 \otimes Q' = \coproduct_{r',n-r'}(P)}$. If~$r < r'$, then we would get a rectangle of crosses~\cross{} in rows~${0 \le i \le r'-r-1}$ and columns~${1 \le j \le r}$ of~$P'_1$. Proposition~\ref{prop:decomposition} would then imply that~$P'_1$ is $\gammainsertion_0$-decomposable, a contradiction. So we must have~$r \ge r'$. By symmetry, we also have~$r \le r'$, and hence~$r = r'$. This gives
\[
P_1 \otimes Q = \coproduct_{r,n-r}(P) = \coproduct_{r',n-r'}(P) = P'_1 \otimes Q'.
\]
We can now conclude that~$P_1 = P'_1$ and~$Q = Q'$. The induction hypothesis on~$Q$ gives us that~$\ell = \ell'$ and~$P_i = P'_i$ for all~$1 \le i \le \ell$. 
\end{proof}

\begin{corollary}
\label{cor:indecomp}
For~$P\in\pipeDreams_n$ and its unique decomposition~$P = P_1 \gammainsertion_0 P_2 \gammainsertion_0 \cdots \gammainsertion_0 P_\ell$ into $\gammainsertion_0$-indecomposable pipe dreams, we define
\[
\Psi_P = P_1 \product P_2 \product \dots \product P_\ell.
\]
Then the set~$\bigset{\Psi_P}{P \in \pipeDreams_n}$ is a linear basis of~$\HP_n$.
\end{corollary}

\begin{proof}
From Proposition~\ref{prop:LT} we have
\[
\leadingTerm(\Psi_P) = P_1 \gammainsertion_0 P_2 \gammainsertion_0 \cdots \gammainsertion_0 P_\ell = P.
\]
We get the result by triangularity and Lemma~\ref{lem:decomposition}.
\end{proof}

\begin{theorem}
The algebra~$\HP$ is free with generators~$\indecomposablePipeDreams = \set{P \in \pipeDreams}{\text{$P$ is $\gammainsertion_0$-indecomposable}}$.
\end{theorem}

\begin{proof}
Corollary~\ref{cor:indecomp} implies that the $\gammainsertion_0$-indecomposable pipe dreams are algebraically independent.
\end{proof}


\subsection{Better description of $\gammainsertion_0$-indecomposable pipe dreams}
\label{subsec:betterDescriptionIndecomposables}

In Proposition~\ref{prop:decomposition} we provided a characterization of $\gammainsertion_0$-indecomposable pipe dreams. For~$P \in \pipeDreams_n$, this characterization requires to check for all~$0 < r < n$ if~$P$ has crosses~\cross{} in a large rectangle. In this section we  show that this can be done more efficiently by checking the state of very few positions.

\begin{theorem}
\label{thm:bestindecomposable}
Let~$P \in \pipeDreams_n$ and let~$\omega_P = \omega_1 \gdproduct \omega_2 \gdproduct \cdots \gdproduct \omega_{\ell}$ be the unique factorization of $\omega_P$ into atomic permutations, with $\omega_i \in \fS_{n_i}$. Then $P$ is $\gammainsertion_0$-indecomposable if and only if~$P$ has crosses~\cross{} in positions~$(n_\ell, 0),\ (n_\ell + n_{\ell-1}, 0),\ \ldots,\ (n_\ell + n_{\ell-1} + \dots + n_2, 0)$.
\end{theorem}

To prove this result, we need to recall a result of~\cite{BergeronBilley} about the generation of all pipe dreams of a given permutation as a partial order. See also~\cite{PilaudStump-ELlabelings} for more details about the structure of a related partial order.

Given a permutation~$\omega \in \fS_n$, we let~$I(\omega) = (i_1, i_2, \dots, i_{n-1})$ be the \defn{Lehmer code} of~$\omega$, \ie the sequence counting the number of inversions of each value:
\[
i_k = \big| \set{j \in [n]}{k < j \text{ and } \omega(k) > \omega(j)} \big|.
\]
It is well-known that~$I(\omega)$ characterizes~$\omega$. Given~$I(\omega) = (i_1, i_2, \dots, i_{n-1})$, we define the pipe dream~$\Ptop_\omega$ by putting crosses~\cross{} in the top~$i_k$ rows of the $k$th column. We leave it to the reader to check that~$\Ptop_\omega$ is a reduced pipe dream with~$\omega_{\Ptop_\omega} = \omega$.

A \defn{chute move} is a flip where an elbow is exchanged with a cross in the next row. See \fref{fig:chuteMove}.

\begin{figure}[ht]
	\capstart
	\centerline{
		\input{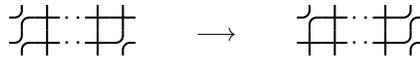}
	}
	\caption{A chute move.}
	\label{fig:chuteMove}
\end{figure}

\noindent
Note that the chute move graph is acyclic since any chute move decreases the sum of the row indices of the elbows of the pipe dream.
We denote by~$(\pipeDreams(\omega), \prec)$ the poset defined by \ie~$Q \prec P$ if we can go from~$P$ to~$Q$ by a sequence of chute moves.

\begin{example}
For~$\omega = 321$, we have~$I(\omega) = (2, 1)$. \fref{fig:PosetPipeDreams} depicts the full poset~$(\pipeDreams(\omega),\prec)$ with the pipe dream~$\Ptop_\omega$ on top.

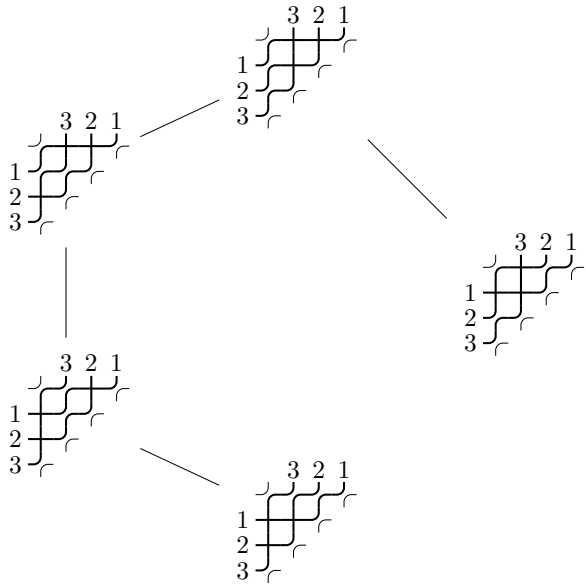
\begin{figure}[h]
	\capstart
	\centerline{
		\begin{tikzpicture}[scale=2.1,baseline=.5cm]
	\node (top) at (0,3) {
		\pipeDreamMonoColor{black}{black}{
			{n,n,3,2,1},
			{n,t,c,c,b},
			{1,e,c,b},
			{2,e,b},
			{3,b},
		}};
	\node (10) at (1.5,1.5) {
		\pipeDreamMonoColor{black}{black}{
			{n,n,3,2,1},
			{n,t,c,e,b},
			{1,c,c,b},
			{2,e,b},
			{3,b},
		}};
	\node (01) at (-1.5,2.3) {
		\pipeDreamMonoColor{black}{black}{
			{n,n,3,2,1},
			{n,t,c,c,b},
			{1,e,e,b},
			{2,c,b},
			{3,b},
		}};
	\node (02) at (-1.5,.7) {
		\pipeDreamMonoColor{black}{black}{
			{n,n,3,2,1},
			{n,t,e,c,b},
			{1,c,e,b},
			{2,c,b},
			{3,b},
		}};
	\node (bot) at (0,0) {
		\pipeDreamMonoColor{black}{black}{
			{n,n,3,2,1},
			{n,t,e,e,b},
			{1,c,c,b},
			{2,c,b},
			{3,b},
		}};
	\draw (top) -- (10); 
	\draw (top) -- (01); 
	\draw (01)  -- (02); 
	\draw (02)  -- (bot); 
\end{tikzpicture}
	}
	\caption{The poset of chute moves on pipe dreams for $\omega=321$.}
	\label{fig:PosetPipeDreams}
\end{figure}
\end{example}
    
\begin{theorem}[\cite{BergeronBilley}]
\label{thm:BB}
The poset~$(\pipeDreams_\omega, \prec)$ is connected with unique maximum~$\Ptop_\omega$.
\end{theorem}

Note that the poset~$\pipeDreams_\omega$ has a nicer structure if one also considers ladder moves (transpose of chute moves) and more generally all other flips. 
For example, the poset of all increasing flips on~$\pipeDreams(\omega)$ is shown to be shellable in~\cite{PilaudStump-ELlabelings}. 
In this section, we restrict to chute moves as it simplifies the proof of our main theorem.

\begin{proof}[Proof of Theorem~\ref{thm:bestindecomposable}]
If~$\ell = 1$, then~$\omega_P$ is atomic and Lemma~\ref{lem:gbinsertion} shows that~$P$ cannot be $\gammainsertion_0$-decomposable. The theorem thus holds: $P$ is indecomposable and there is no restriction on the crosses~\cross{} in the leftmost column.

We now assume that~$\ell > 1$. Let~$\omega = \omega_P = \omega_1 \gdproduct \omega_2 \gdproduct \cdots \gdproduct \omega_{\ell}$ with~$\omega_i \in \fS_{n_i}$. Consider the set~$A_\omega = \set{P \in \pipeDreams_\omega}{P \text{ is $\gammainsertion_0$-decomposable}}$. If~$P \in A_\omega$, then Proposition~\ref{prop:decomposition} ensures that there is~$0 < r < n$ such that~$P$ has crosses~\cross{} in the rectangle of rows~$0 \le i \le r-1$ and columns $1 \le j \le n-r$. Moreover, Lemma~\ref{lem:gbinsertion} implies that~$n - r = n_1 + n_2 + \dots + n_k$ for some~$1 \le k < \ell$. Thus~$r = n_{k+1} + n_{k+2} + \dots + n_\ell$. If we perform any inverse chute moves on~$P$, the resulting pipe dream will still have crosses~\cross{} in the rectangle of rows~$0 \le i \le r-1$ and columns~$1 \le j \le n-r$. This implies that~$Q \in A_\omega$ for all~$Q \succ P$. In other words, $A_\omega$ is an upper ideal in $(\pipeDreams_\omega,\prec)$.
In particular, the set $B_\omega = \pipeDreams_\omega \ssm A_\omega = \set{P \in \pipeDreams_\omega}{P \text{ is $\gammainsertion_0$-indecomposable}}$ is a lower ideal in $(\pipeDreams_\omega,\prec)$.

For any~$P \in B_\omega$ there is a sequence of chute moves from~$\Ptop_\omega$ to $P$. Observe that~$\Ptop_\omega$ has crosses~\cross{} in all the rectangles of rows~$0 \le i \le r-1$ and columns~$1 \le j \le n-r$ where~$n - r = n_1 + n_2 + \dots + n_k$ and~$1 \le k < \ell$. Every one of these rectangles must have at least one less cross~\cross{}. In the sequence of chute moves from~$\Ptop_\omega$ to $P$, consider a move that removes a cross~\cross{} from a rectangle for the first time. We must remove one cross~\cross{} from the rectangle of rows~$0 \le i \le r-1$ and columns~$1 \le j \le n-r$ by that move. Assume that the chute move transports a cross~\cross{} from~$(i,j)$ to~$(i+1,j')$. We have that~$(i,j)$ must be in the rectangle and we must have an elbow~\elbow{} in position~$(i+1,j)$. This can only happen if~$i = r-1$. Moreover,~$j$ cannot be the column of another rectangle, hence this will remove a cross~\cross{} in a single rectangle and not in any other one. Now, we must also have an elbow~\elbow{} in positions~$(r-1,j')$ and~$(r,j')$. This implies that~$j' = 0$. This chute move will move a cross~\cross{} from~$(r-1,j)$ to the position~$(r,0)$.

To summarize, a sequence of chute moves from~$\Ptop_\omega$ to~$P$ must remove at least one cross~\cross{} of each rectangle. The first time it does for each rectangle, it puts a cross~\cross{} in position~$(r,0)$ and does not affect the other rectangles. The cross~\cross{} in position~$(r,0)$ will not move in any subsequent chute move. This shows that~$P \in B_\omega$ if and only if $P$ has crosses~\cross{} in positions~$(n_\ell, 0),\ (n_\ell+n_{\ell-1}, 0),\ \ldots,\ (n_\ell+n_{\ell-1}+\cdots+n_2, 0).$
\end{proof}

 
\subsection{(Free) generators of $\HP^*$}
\label{subsec:generatorsDual}

We now show that~$\HP$ is cofree and show that the cogenerators are also the $\gammainsertion_0$-indecomposable pipe dreams. To this end, we work in the graded dual~$\HP^*$ of~$\HP$. The product in~$\HP^*$ is given by the dual of the coproduct~$\coproduct$ in~$\HP$. Namely, for~$\b{P} \in \pipeDreams_{\b{m}}$ and $\r{Q} \in \pipeDreams_{\r{n}}$ we let~$\b{P^*}$ and~$\r{Q^*}$ denote the dual basis element in~$\HP^*$. The product is now defined by
\[
\b{P^*} \product \r{Q^*} = \sum_{R \in \pipeDreams_{\b{m}+\r{n}} \atop \coproduct_{\b{m},\r{n}}(R) = \b{P} \otimes \r{Q}} R^*.
\]
Using the same order on pipe dreams as in Section~\ref{subsec:order} we obtain the following statement.

\begin{lemma}
For $\b{P}\in\pipeDreams_n$ and  $\r{Q}\in\pipeDreams_m$ we have
\[
\leadingTerm({\b{P^*} \product \r{Q^*}})={(\b{P} \gammainsertion_0 \r{Q})^*}.
\]
\end{lemma}

\begin{proof}
First observe that Proposition~\ref{prop:decomposition} gives that for~$\b{P} \gammainsertion_0 \r{Q}$ we have~$\coproduct_{n,m}(\b{P} \gammainsertion_0 \r{Q}) = \b{P}\otimes \r{Q}$. Hence~$(\b{P} \gammainsertion_0 \r{Q})^*$ is a term of~${\b{P^*} \product \r{Q^*}}$. The crosses~\cross{} in the rectangle in rows~$0 \le i \le \b{m}-1$ and columns~$1 \le j \le \r{n}$ of $\b{P} \gammainsertion_0 \r{Q}$ are exactly the crosses that are ignored in the computation of~$\coproduct_{\b{m},\r{n}}(\b{P} \gammainsertion_0 \r{Q}) = \b{P} \otimes \r{Q}$. They are the crosses between the pipes of~$\b{P}$ and the pipes of~$\r{Q}$. Recall that $\columnReading({\b{P} \gammainsertion_0 \r{Q}}) = \elbowLetter \b{p_{1,0}\cdots p_{n-1,0}} \elbowLetter \cdots$ where~$\columnReading(\b{P} ) = \elbowLetter \b{p_{1,0}\cdots p_{n-1,0}} \elbowLetter\cdots$. Using Lemma~\ref{lem:gbinsertion}, let $\omega = \omega_{(\b{P} \gammainsertion_0 \r{Q})} = \omega_{\b{P}}\gdproduct\omega_{\r{Q}}$.
  
Now for any other term~$R^*$ in~${\b{P^*} \product \r{Q^*}}$ we have that~$\coproduct_{\b{m},\r{n}}(R) = \b{P}\otimes \r{Q}$. By Proposition~\ref{prop:codescentproduct}, we have that~$\omega_R = \omega_P \gdproduct \omega_Q = \omega$. Hence, all~$R^*$ appearing in the product~${\b{P^*} \product \r{Q^*}}$ are such that~$R \in \pipeDreams_\omega$. If~$R$ has crosses~\cross{} in the rectangle in rows~$0 \le i \le \b{m}-1$ and columns~$1 \le j \le \r{n}$, then it must be~$R = \b{P} \gammainsertion_0 \r{Q}$. If~$R \ne \b{P} \gammainsertion_0 \r{Q}$, then some of the crosses~\cross{} in the rectangle must be moved out. As in the proof of Theorem~\ref{thm:bestindecomposable}, this implies that~$R$ has a cross in position~$(\b{m},0)$. At position $(\b{m},0)$ the pipes $\b{m}<\r{n}$ cross. Since $\coproduct_{\b{m},\r{n}}(R) = \b{P}\otimes \r{Q}$ and looking back at how~$P$ is obtained from~$R$, we must have~$\columnReading(R) = \elbowLetter \b{p_{1,0}\cdots p_{i,0}} \crossLetter \crossLetter \cdots \crossLetter \cdots$ for some~$i < \b{m}$. The crosses~$\crossLetter$ of~$\columnReading(R)$ in positions~$i+1,\ldots,n$ correspond to the crosses~\cross{} of~$\b{P} \gammainsertion_0 \r{Q}$ in column~$0$ that involves the pipe~$m$. It is clear now that~$\columnReading({\b{P} \gammainsertion_0 \r{Q}}) \ltlex \columnReading(R)$ and the lemma follows.
\end{proof}

Following exactly the same analysis as before we conclude with the following theorem. Its proof is left to the reader.

\begin{theorem}
The dual Hopf algebra~$\HP^*$ is free with generators 
\[
\set{P^*}{P\in\pipeDreams \text{ and $P$ is $\gammainsertion_0$-indecomposable}}.
\]
\end{theorem}


\part{Some relevant Hopf subalgebras}

In this part, we study some  interesting Hopf subalgebras of~$(\HP, \product, \coproduct)$ arising when we restrict either the atom sets of the permutations (Sections~\ref{sec:subalgebrasAtoms} and~\ref{sec:matomic}), or the pipe dreams to be acyclic (Section~\ref{sec:acyclicPipeDreams}).


\section{Some Hopf subalgebras of~$\HP$ from restricted atom sets}
\label{sec:subalgebrasAtoms}

Recall from Section~\ref{sec:HopfAlgebraPermutations} that a permutation~$\omega \in \fS$ has a unique factorization~${\omega = \nu_1 \gdproduct \nu_2 \gdproduct \cdots \gdproduct \nu_\ell}$ into atomic permutations and that we denote by~$\omega^\gdproduct \eqdef \{\nu_1, \nu_2, \dots, \nu_\ell\}$ the set of atomic permutations that appear in its factorization.

Given a subset~$S$ of atomic permutations, we define
\[
\fS_n \langle S \rangle \eqdef \set{\omega \in \fS_n}{\omega^{\gdproduct} \subseteq S}
\qquad\text{and}\qquad
\fS \langle S \rangle \eqdef \bigsqcup_{n \in \N} \fS_n \langle S \rangle,
\]
from which we derive
\[
\pipeDreams_n \langle S \rangle \eqdef \set{P \in \pipeDreams_n}{\omega_P \in \fS_n \langle S \rangle}
\qquad\text{and}\qquad
\pipeDreams \langle S \rangle \eqdef \bigsqcup_{n \in \N} \pipeDreams_n \langle S \rangle.
\]
The following statement is immediate from Proposition~\ref{prop:descentproduct}.

\begin{theorem}
For any set~$S$ of atomic permutations,
\begin{itemize}
\item the subspace~$\HS \langle S \rangle$ defines a Hopf subalgebra of~$(\HS, \gdshuffle, \gddeconcat)$,
\item the subspace~$\HP \langle S \rangle$ defines a Hopf subalgebra of~$(\HP, \product, \coproduct)$.
\end{itemize}
\end{theorem}

\begin{proof}
By definition, $\HS \langle S \rangle$ is preserved by the product~$\gdshuffle$ and coproduct~$\gddeconcat$. Now~$\HP \langle S \rangle$ is the inverse image of $\HS \langle S\rangle$ via the Hopf morphism $\omega \colon \HP \to \HS$ (see Proposition~\ref{prop:descentHopf}).
\end{proof}

All results of Section~\ref{sec:freeness} restrict to the Hopf subalgebra $\HP \langle S \rangle$ without any difficulties. In particular we have the following statement for any subset~$S$ of atomic permutations.

\begin{theorem}
\label{thm:freenessofOmega}
The Hopf subalgebra~$\HP \langle S \rangle$ is free and cofree. The generators and cogenerators of~$\HP \langle S \rangle$ are exactly the $\gammainsertion_0$-indecomposible pipe dreams in~$\pipeDreams \langle S \rangle$.
\end{theorem}

In the following we let
\[
\indecomposablePipeDreams_n \langle S \rangle \eqdef \set{P \in \pipeDreams_n \langle S \rangle}{\text{$P$ is $\gammainsertion_0$-indecomposable}}.
\]


\subsection{The Loday--Ronco Hopf algebra on complete binary trees $\HP \langle 1 \rangle$}
\label{subsec:LodayRonco}

We say that a pipe dream~$P \in \pipeDreams_n$ is \defn{reversing} if it completely reverses the order of its relevant pipes, \ie if~${\omega_P = [n, n-1, \dots, 1]= 1 \gdproduct 1 \gdproduct \cdots \gdproduct 1}$.
In other words, the reversing pipe dreams are those of~$\HP \langle 1 \rangle$.
As observed by different authors~\cite{Woo, PilaudPocchiola, Pilaud-these, Stump}, the reversing pipe dreams are enumerated by the Catalan numbers and are in bijection with various Catalan objects. 
\fref{fig:bijection} illustrates explicit bijections between the reversing pipe dreams with $n$ relevant pipes, the complete binary trees with $n$ internal nodes, and the triangulations of a convex $(n+2)$-gon.

\begin{figure}[h]
	\capstart
	\centerline{\includegraphics[scale=1.3]{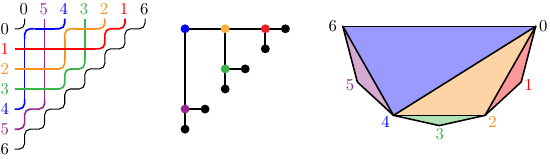}}
	\caption{The bijection between reversing pipe dreams (left), complete binary trees (middle) and triangulations (right).}
	\label{fig:bijection}
\end{figure}

More precisely, the map which sends an elbow~\elbow{} in row~$i$ and column~$j$ of the triangular shape to the diagonal~$[i,n+1-j]$ of the~$(n+2)$-gon provides the following correspondence:
\centerline{
\begin{tabular}{r@{$\quad\longleftrightarrow\quad$}l}
pipe dream~$P \in \HP_n \langle 1 \rangle$ & triangulation~$P\duality$ of the $(n+2)$-gon, \\
$i$th pipe of~$P$ & $i$th triangle of~$P\duality$ (with central vertex~$i$), \\
elbows of~$P$ & diagonals of~$P\duality$ (including boundary edges of the polygon), \\
crosses of~$P$ & common bisectors between triangles of~$P\duality$, \\
elbow flips in~$P$ & diagonal flips in~$P\duality$.
\end{tabular}
}

The complete binary tree dual to the triangulation~$P\duality$ can also be easily described from the pipe dream $P$.
Each elbow in $P$ is replaced by a node in the tree, a node is connected with the next node below it (if any) and with the next node to its right (if any). 
The induced map, which we denote by $\Psi$, provides the following correspondence:

\begin{center}
\begin{tabular}{r@{$\quad\longleftrightarrow\quad$}l}
pipe dream~$P \in \HP_n \langle 1 \rangle$ & complete binary tree $\Psi(P)$ with $n$ internal nodes, \\
elbow flips in~$P$ & tree rotations in~$\Psi(P)$.
\end{tabular}
\end{center}

In~\cite{LodayRonco}, J.~L.~Loday and M.~Ronco introduced a Hopf algebra structure on complete binary trees 
which has been widely studied in the litarature, see for instance~\cite{AguiarSottile-LodayRonco, HivertNovelliThibon-algebraBinarySearchTrees}. 
Interestingly, this Hopf algebra structure is equivalent to the Hopf subalgebra of reversing pipe dreams. 

\begin{proposition}
\label{prop:reversingLodayRonco}
The map $\Psi$ is a Hopf algebra isomorphism between the Hopf subalgebra~$\HP \langle 1 \rangle$ of reversing pipe dreams
and the Loday--Ronco Hopf algebra on complete binary trees.
\end{proposition}

This proposition is a particular case\footnote{The map $\Psi$ in Theorem~\ref{thm:dominantPipeDreamsVSNuTrees} is a bit more general as it is defined for dominant pipe dreams. Its restriction to reversing pipe dreams is what we use as the map $\Psi$ in Proposition~\ref{prop:reversingLodayRonco}.} of a stronger result (Theorem~\ref{thm:dominantPipeDreamsVSNuTrees}) which will be discussed in Section~\ref{sec:nuTamari}.

\begin{remark}
The dimension of~$\HP_n \langle 1 \rangle$ is the number of binary trees with~$n$ internal nodes, that is the $n$th Catalan number~$C_n \eqdef \frac{1}{n+1} \binom{2n}{n}$. From Theorem~\ref{thm:freenessofOmega} we obtain that the generators of degree~$n$ in~$\HP \langle 1 \rangle$ are the elements~$P \in \indecomposablePipeDreams_n \langle 1 \rangle$. According to Theorem~\ref{thm:bestindecomposable}, $P \in \indecomposablePipeDreams_n \langle 1 \rangle$ if and only if~$P$ has crosses~\cross{} in the positions~$(1,0), (2,0), \dots, (n-1,0)$. This allows us to give a bijection $\indecomposablePipeDreams_n \langle 1 \rangle \to \pipeDreams_{n-1} \langle 1 \rangle$. The bijection is simply $P \mapsto P'$ where we remove the leftmost column of~$P$ to get~$P'$. Hence the number of free generators of degree~$n$ for~$\HP \langle 1 \rangle$ is $C_{n-1}$ the $(n-1)$st Catalan number. Remark that the generating function of Catalan numbers is
\[
c(t) = \sum_{n\ge 0} C_n t^n = \frac{1-\sqrt{1-4t}}{2t} = \frac{2}{1+\sqrt{1-4t}} = \frac{1}{1-tc(t)}\,.
\]
The last expression shows that indeed we need~$C_{n-1}$ free generators for this free algebra.
Note that~$P \in \indecomposablePipeDreams_n \langle 1 \rangle$ if and only if its binary tree~$\Psi(P)$ is right-tilting, meaning that its left child is empty.
\end{remark}


\subsection{The Hopf algebra $\HP \langle 12 \rangle$}
\label{sec:LR12}

It is interesting to look at other subalgebras generated by single atomic permutations. As a warm up, we now look at~$\HP \langle 12 \rangle$. Using Theorem~\ref{thm:freenessofOmega}, the generetors of degree~$n$ of~$\HP \langle 12 \rangle$ are exactly the pipe dreams~$P \in \indecomposablePipeDreams_n \langle 12 \rangle$. Observe that~$\pipeDreams_n \langle 12 \rangle$ is empty unless~$n = 2k$ for~$k \ge 0$. From Theorem~\ref{thm:bestindecomposable}, $P \in \indecomposablePipeDreams_{2k} \langle 12 \rangle$ if and only if~$P$ has crosses~\cross{} in  the leftmost column at the coordinates~$(2,0),\ (4,0),\ \ldots,\ (2k-2,0)$.
We now claim that there is a bijection~$\indecomposablePipeDreams_{2k} \langle 12 \rangle \to \pipeDreams_{2k-1} \langle 1 \rangle$. This will be done in Section~\ref{sec:matomic} for a more general case. This will give us that the number of generators of degree~$n$ is~$C_{2k-1}$ if~$n = 2k$, and~$0$ otherwise. The Hilbert series~$h_{12}(t)$ of the subalgebra~$\HP \langle 12 \rangle$ is then
\[
h_{12}(t) = \frac{1}{1 - t \Big( \frac{c(t)-c(-t)}{2} \Big)}\,.
\]
This is exactly the generating series of the number of planar trees on~$2k$ edges with every subtree at the root having an even number of edges~\href{https://oeis.org/A066357}{\cite[A066357]{OEIS}}. One could be interested in constructing an explicit bijection between these trees and the set~$\HP \langle 12 \rangle$. We leave this problem open to the interested reader.
 

\subsection{The Hopf algebra $\HP \langle (m-1) \cdots 2 1 m \rangle$}
\label{sec:matomic}

We now look at a generalization of Sections~\ref{subsec:LodayRonco} and~\ref{sec:LR12}. Consider the atomic permutation~$(m-1)\cdots 2 1 m$ and the pipe dreams~${\pipeDreams \langle (m-1)\cdots 2 1 m \rangle}$. To count the number of generators we need the following proposition.

\begin{proposition}
There is a bijection~$\indecomposablePipeDreams_{km} \langle (m-1) \cdots 2 1 m \rangle \to \pipeDreams_{km-1} \langle 1 \rangle$.
Thus, the number of generators of degree $n$ in $\HP \langle (m-1) \cdots 2 1 m \rangle$ is the Catalan number~$C_{km-1}$ if~$n = km$, and~$0$ otherwise.
\end{proposition}
 
\begin{proof}[Proof (sketch)] 
For the sake of space, we only sketch a proof of the bijection.
Fix~$m \ge 1$ and~$k \ge 1$.
Consider the sequence of permutations~$(\omega_i)_{i \in [k]}$ defined by
\[
\omega_1 \eqdef (km-1) (km-2) \cdots 3 2 1 (km)
\qquad\text{and}\qquad
\omega_{i+1} \eqdef (im \;\, km) \, \omega_i,
\]
and the sets of pipe dreams
\[
\Gamma_i \eqdef \set{P \in \pipeDreams(\omega_i)}{P \text{ has crosses at positions } ((i-1)m,0), \dots, (m,0)}.
\]
Observe that
\begin{itemize}
\item $\pipeDreams_{km-1} \langle 1 \rangle$  can be identified with~$\Gamma_1$ by adding or deleting a diagonal of elbows~$\elbow$,
\item $\indecomposablePipeDreams_{km} \langle (m-1) \cdots 2 1 m \rangle$ is precisely~$\Gamma_k$ by application of Theorem~\ref{thm:bestindecomposable} to the permutation
\[
(m-1) \cdots 2 1 m \gdproduct \dots \gdproduct (m-1) \cdots 2 1 m = (m \;\, 2m \;\, \dots \;\, km) \, \omega_1 = \omega_k.
\]
\end{itemize}
We claim that there is a simple bijection from~$\Gamma_{i+1}$ to~$\Gamma_i$.
Consider a pipe dream~$P \in \Gamma_{i+1}$.
Since~$\omega_i = (im \;\, km) \, \omega_{i+1}$, we just need to uncross the pipes~$im$ and~$km$ of~$P$ to transform it into a pipe dream of~$\Gamma_i$.
To achieve this, the naive approach is to just replace the unique crossing between the pipes~$im$ and~$km$ of~$P$ by an elbow.
However, we need to keep track of the position of the replacement so that the map we define is invertible, and this will also make sure that resulting paths do not cross twice.
For that, we consider that the pipe~$im$ has extra crossing~$(x,y)$ with a pipe~$\ell$ (initially, $\ell = km$), and will successively \emph{bump out} this extra cross at~$(x,y)$ until we reach a reduced pipe dream.
Note that this process is very similar to the insertion algorithm used to prove Monk's rule in~\cite{BergeronBilley}.
More precisely, there are three possible situations:

\medskip
\noindent{\bf (A)} Suppose there is no elbow $\elbow$ to the left of~$(x,y)$ in~$P$. In this case, we must have~$x = im$ so that changing the crossing at~$(x,y)$ in~$P$ for an elbow results in a reduced pipe dream:

\medskip
\centerline{
	\pipeDreamBiColor{black}{red}{black}{
	{n,n,n,n,n,n,n,y},
	{},
	{x = im,n,n,c/l/l,c/l/n,\;.\,.\;,c/l/n,c/r/r,\;.\,.\;},
	{n,n,n,n,n,n,n,\g{\ell}},
}
$\qquad\longleftrightarrow\qquad$
\pipeDreamBiColor{black}{red}{black}{
	{n,n,n,n,n,n,n,y},
	{},
	{x = im,n,n,c/l/l,c/l/n,\;.\,.\;,c/l/n,e/r/r,\;.\,.\;},
	{n,n,n,n,n,n,n,\g{\ell}},
}
}
\medskip

In the other two cases, we assume there are elbows in row~$x$ to the left of position~$(x,y)$. We locate the largest~$z<y$ such that there is an~$\elbow$ in position~$(x,z)$. This elbow involves the pipe~$im$ and a pipe~$j$ as pictured below 

\medskip
\centerline{
	\pipeDreamBiColor{black}{red}{black}{
	{n,n,n,z,n,n,n,y},
	{},
	{x,n,\g{\ell'},e/l/l,c/l/n,\;.\,.\;,c/l/n,c/r/r,\;.\,.\;},
	{n,n,n,\g{im},n,n,n,\g{\ell}},
}
}
\medskip

There are two cases to consider. 

\medskip
\noindent{\bf (B1)} If~$j < im$, then the pipe~$j$ must cross the pipe~$\ell$ at some position~$(x',y')$ for~$x' < x$ and~$y' \ge y$. This is guarantied since~$j$ and~$\ell$ are inverted in~$\omega_{i+1}$ and this must happen strictly above row~$x$ and weakly to the right of column~$y$.

\medskip
\centerline{
	\pipeDreamBiColor{black}{red}{black}{
	{n,n,n,z,n,n,n,y,n,y'},
	{},
	{x',n,n,n,n,n,e/n/l,\;.\,.\;,\;.\,.\;,c/l/l},
	{n,n,n,e/n/l,c/l/n,\;.\,.\;,e/l/n,n,n,:},
	{n,n,n,:,n,n,n,e/n/l,\;.\,.\;,e/l/n},
	{x,n,\g{\ell'},e/l/l,c/l/n,\;.\,.\;,c/l/n,c/r/r,\;.\,.\;},
	{n,n,n,\g{im},n,n,n,\g{\ell}},
}
$\qquad\longleftrightarrow\qquad$
\pipeDreamBiColor{black}{red}{black}{
	{n,n,n,z,n,n,n,y,n,y'},
	{},
	{x',n,n,n,n,n,e/n/l,\;.\,.\;,\;.\,.\;,c/r/r},
	{n,n,n,e/n/l,c/l/n,\;.\,.\;,e/l/n,n,n,:},
	{n,n,n,:,n,n,n,e/n/l,\;.\,.\;,e/l/n},
	{x,n,\g{\ell'},c/l/l,c/l/n,\;.\,.\;,c/l/n,e/r/r,\;.\,.\;},
	{n,n,n,\g{im},n,n,n,\g{\ell}},
}
}
\medskip

We can now continue the process with~$(x',y')$ with~$\ell'$.
Remark that a portion of the pipe~$im$ is replaced by a portion strictly left.

\medskip
\noindent{\bf (B2)} If~$j > im$, then the pipes~$j$ and~$im$ must have crossed at a position~$(x',y')$ where~$x' > x$ and~$y' < z < y$.

\medskip
\centerline{
	\pipeDreamBiColor{black}{red}{black}{
	{n,n,n,y',n,n,n,z,n,n,n,y},
	{},
	{x,n,n,n,n,e/n/l,\;.\,.\;,e/l/l,c/l/n,\;.\,.\;,c/l/n,c/r/r,\;.\,.\;},
	{n,n,n,e/n/l,\;.\,.\;,e/l/n,n,:,n,n,n,\g{\ell}},
	{n,n,n,:,n,e/n/l,\;.\,.\;,e/l/n,n,n,n},
	{x',n,\g{im\ },c/l/l,c/l/n,e/l/n,n,n,n,n,n},
	{n,n,n,\g{\ell'}},
}
$\qquad\longleftrightarrow\qquad$
\pipeDreamBiColor{black}{red}{black}{
	{n,n,n,y',n,n,n,z,n,n,n,y},
	{},
	{x,n,n,n,n,e/n/l,\;.\,.\;,c/l/l,c/l/n,\;.\,.\;,c/l/n,e/r/r,\;.\,.\;},
	{n,n,n,e/n/l,\;.\,.\;,e/l/n,n,:,n,n,n,\g{\ell}},
	{n,n,n,:,n,e/n/l,\;.\,.\;,e/l/n,n,n,n},
	{x',n,\g{im\ },c/r/r,c/l/n,e/l/n,n,n,n,n,n},
	{n,n,n,\g{\ell'}},
}
}
\medskip

If there are no elbows to the left of~$(x',y')$ we are in case~(A) and we stop. If not, we continue the process with~$(x',y')$ and~$\ell'$. 
Again we remark that a portion of the pipe~$im$ is replaced by a portion strictly left. The process in~(B1) and~(B2) will eventually stop as we replace the pipe~$im$ with a pipe strictly to the left.  The pipe~$im$ will eventually interact with a pipe that has a cross in row~$x = im$ and step~(A) will be used and it stops. 
We finally get a reduced pipe dream in~$\Gamma_{i+1}$.
The process is invertible, hence the map is a bijection from~$\Gamma_{i+1}$ to~$\Gamma_i$.
Composing all these bijections, we get a bijection from~$\Gamma_k = \indecomposablePipeDreams_{km} \langle (m-1) \cdots 2 1 m \rangle$ to~$\Gamma_i \simeq \pipeDreams_{km-1} \langle 1 \rangle$ as desired.
\end{proof}

\begin{remark}
We can also in this case find an explicit formula for the Hilbert series of the generators and for the Hopf algebra.
Recall that we denote by~$c(t)$ the generating series of the Catalan numbers~$C_n$. The generating series of generators here is
\[
h_{(m-1) \cdots 2 1 m}^0(t) = \frac{t}{m} \sum_{k=0}^m e^{\frac{2k\pi i}{m}} c(e^{\frac{2k\pi i}{m}}t)
\]
and the Hilbert series of the Hopf algebra is
\[
h_{(m-1)\cdots 2 1 m}(t) = \frac{1}{1-h_{(m-1)\cdots 2 1 m}^0(t)}\,.
\]
\end{remark}


\section{The Hopf subalgebra $\HP \langle 1, 12, 123, \dots \rangle$ and lattice walks on the quarter plane}
\label{sec:matomic}


\subsection{Conjectured bijection between pipe dreams and walks}
\label{subsec:conjecturedBijection}

Let us consider one last example of a Hopf subalgebra obtained from a set of atomic permutations.
We consider the infinite set consisting of all identity permutations of any size
\[
\Sid \eqdef \{1, 12, 123, \ldots\}.
\] 
Experimental computations show that the dimensions of the graded components of this Hopf subalgebra are given by the sequence 
\[1, 1, 3, 12, 57, 301, 1707, 10191, 63244, 404503, 2650293, .\dots\]
which coincides with the sequence determined by the number of walks in a special family of walks in the quarter plane~\href{https://oeis.org/A151498}{\cite[A151498]{OEIS}} considered by M.~Bousquet-M\'elou and M.~Mishna in~\cite{BousquetMelouMishna}, see \fref{fig:walks}. We propose the following conjecture.

\begin{conjecture}
\label{conj:walks1}
The dimension of $\HP_n \langle 1, 12, 123, \dots \rangle$ is equal the number of walks in the quarter plane (within~$\mathbb{N}^2\subset \mathbb{Z}^2$) starting at $(0,0)$, ending on the horizontal axis, and consisting of $2n$ steps taken from $\{(-1, 1), (1, -1), (0, 1)\}$.
\end{conjecture}

\begin{remark}
\label{rem:latticewalksEnumeration}
In~\cite{BousquetMelouMishna}, Bousquet-M\'elou and Mishna considered 79 different models of lattice walks with small steps in the quarter plane. 
The special family we consider is one among them, and was proven to be non $D$-finite by M.~Mishna and A.~Rechnitzer in~\cite{MishnaRechnitzer}, see also~\cite{MelczerMishna}. 
This means that the generating function for these walks does not satisfy any linear differential equation with polynomial coefficients. 
In~\cite{MishnaRechnitzer, MelczerMishna}, the authors give a complicated expression for  the generating function using a variant of the kernel method, 
and show that the number of such walks behaves asymptotically as $\alpha \frac{3^n}{\sqrt{n}}(1+o(1))$ 
for a constant $0\leq \alpha \leq \sqrt{\frac{3}{\pi}}$~\cite[Prop.~16]{MishnaRechnitzer}.
\end{remark}

The dimension of $\HP_n \langle 1, 12, 123, \dots \rangle$ is counted by the number of pipe dreams $P$ such that $\omega_P\in \fS_n$ is a permutation whose factorization into atomics consists of identity permutations of arbitrary size. The refined counting of such pipe dreams considering the number of atomic parts gives rise to the following stronger conjecture which implies Conjecture~\ref{conj:walks1}.

\begin{conjecture}
\label{conj:walks2}
The following two families have the same cardinality:
\begin{enumerate}
\item Pipe dreams $P$ such that $\omega_P\in \fS_n$ is a permutation whose factorization into atomics consists of $k$ identity permutations.
\item Walks in the quarter plane starting at $(0,0)$ and ending on the horizontal axis, consisting of $2n$ steps taken from $\{(-1, 1), (1, -1), (0, 1)\}$ from which $k$ of them are $(0,1)$. 
\end{enumerate}
\end{conjecture}
 
\begin{example}
For $n=3$, there are four permutations whose factorization into atomics consists of identity permutations:
\[
321=1\bullet 1\bullet 1, \quad 312=1\bullet 12, \quad 231=12\bullet 1, \quad 123 = 123. 
\]
Their corresponding $12 = 5+3+3+1$ pipe dreams are illustrated in the four columns of the left hand side of \fref{fig:pipe_walks}, respectively.
 
The number of desired lattice walks with $2n = 6$ steps is also given by $12 = 5+3+3+1$, where the refinement is determined by the number of north steps~$(0,1)$. 
These are illustrated on the right hand side of \fref{fig:pipe_walks}, using an alternative model which gives some insight of how the bijection could look like in general. 
The objects in consideration are colored Dyck paths of size~$n = 3$.
\end{example}

\begin{figure}[ht]
	\capstart
	\centerline{\input{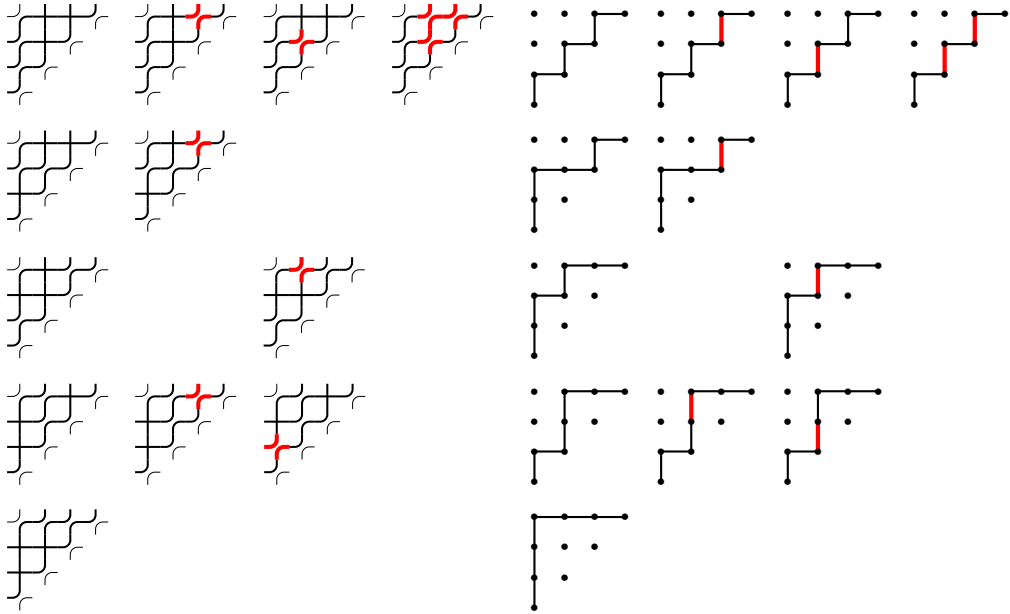} \quad $
\begin{array}{cccc}
    \begin{tikzpicture}[scale=.4, baseline=.5cm]
    	\draw [thick] (0,0)--(0,1); 
    	\draw [thick] (0,1)--(1,1); 
    	\draw [thick] (1,1)--(1,2); 
    	\draw [thick] (1,2)--(2,2); 
    	\draw [thick] (2,2)--(2,3); 
    	\draw [thick] (2,3)--(3,3); 
 		\foreach \x in {0, ..., 3}
    		\foreach \y in {\x, ..., 3}
    			\node (\x,\y) at (\x,\y) [inner sep=0pt] {$\scriptscriptstyle \bullet$};
    \end{tikzpicture} 
    &
    \begin{tikzpicture}[scale=.4, baseline=.5cm]
    	\draw [thick] (0,0)--(0,1); 
    	\draw [thick] (0,1)--(1,1); 
    	\draw [thick] (1,1)--(1,2); 
    	\draw [thick] (1,2)--(2,2); 
    	\draw [ultra thick, color=red] (2,2)--(2,3); 
    	\draw [thick] (2,3)--(3,3); 
 		\foreach \x in {0, ..., 3}
    		\foreach \y in {\x, ..., 3}
    			\node (\x,\y) at (\x,\y) [inner sep=0pt] {$\scriptscriptstyle \bullet$};
    \end{tikzpicture} 
    &
    \begin{tikzpicture}[scale=.4, baseline=.5cm]
    	\draw [thick] (0,0)--(0,1); 
    	\draw [thick] (0,1)--(1,1); 
    	\draw [ultra thick, color=red] (1,1)--(1,2); 
    	\draw [thick] (1,2)--(2,2); 
    	\draw [thick] (2,2)--(2,3); 
    	\draw [thick] (2,3)--(3,3); 
 		\foreach \x in {0, ..., 3}
    		\foreach \y in {\x, ..., 3}
    			\node (\x,\y) at (\x,\y) [inner sep=0pt] {$\scriptscriptstyle \bullet$};
    \end{tikzpicture} 
    &
    \begin{tikzpicture}[scale=.4, baseline=.5cm]
    	\draw [thick] (0,0)--(0,1); 
    	\draw [thick] (0,1)--(1,1); 
    	\draw [ultra thick, color=red] (1,1)--(1,2); 
    	\draw [thick] (1,2)--(2,2); 
    	\draw [ultra thick, color=red] (2,2)--(2,3); 
    	\draw [thick] (2,3)--(3,3); 
 		\foreach \x in {0, ..., 3}
    		\foreach \y in {\x, ..., 3}
    			\node (\x,\y) at (\x,\y) [inner sep=0pt] {$\scriptscriptstyle \bullet$};
    \end{tikzpicture} 
    \\[.8cm]
    \begin{tikzpicture}[scale=.4, baseline=.5cm]
    	\draw [thick] (0,0)--(0,1); 
    	\draw [thick] (0,1)--(0,2); 
    	\draw [thick] (0,2)--(1,2); 
    	\draw [thick] (1,2)--(2,2); 
    	\draw [thick] (2,2)--(2,3); 
    	\draw [thick] (2,3)--(3,3); 
 		\foreach \x in {0, ..., 3}
    		\foreach \y in {\x, ..., 3}
    			\node (\x,\y) at (\x,\y) [inner sep=0pt] {$\scriptscriptstyle \bullet$};
    \end{tikzpicture} 
    &
    \begin{tikzpicture}[scale=.4, baseline=.5cm]
    	\draw [thick] (0,0)--(0,1); 
    	\draw [thick] (0,1)--(0,2); 
    	\draw [thick] (0,2)--(1,2); 
    	\draw [thick] (1,2)--(2,2); 
    	\draw [ultra thick, color=red] (2,2)--(2,3); 
    	\draw [thick] (2,3)--(3,3); 
 		\foreach \x in {0, ..., 3}
    		\foreach \y in {\x, ..., 3}
    			\node (\x,\y) at (\x,\y) [inner sep=0pt] {$\scriptscriptstyle \bullet$};
    \end{tikzpicture} 
    &
    &
    \\[.8cm]
    \begin{tikzpicture}[scale=.4, baseline=.5cm]
    	\draw [thick] (0,0)--(0,1); 
    	\draw [thick] (0,1)--(0,2); 
    	\draw [thick] (0,2)--(1,2); 
    	\draw [thick] (1,2)--(1,3); 
    	\draw [thick] (1,3)--(2,3); 
    	\draw [thick] (2,3)--(3,3); 
 		\foreach \x in {0, ..., 3}
    		\foreach \y in {\x, ..., 3}
    			\node (\x,\y) at (\x,\y) [inner sep=0pt] {$\scriptscriptstyle \bullet$};
    \end{tikzpicture} 
    &
    &
    \begin{tikzpicture}[scale=.4, baseline=.5cm]
    	\draw [thick] (0,0)--(0,1); 
    	\draw [thick] (0,1)--(0,2); 
    	\draw [thick] (0,2)--(1,2); 
    	\draw [ultra thick, color=red] (1,2)--(1,3); 
    	\draw [thick] (1,3)--(2,3); 
    	\draw [thick] (2,3)--(3,3); 
 		\foreach \x in {0, ..., 3}
    		\foreach \y in {\x, ..., 3}
    			\node (\x,\y) at (\x,\y) [inner sep=0pt] {$\scriptscriptstyle \bullet$};
    \end{tikzpicture} 
    &
    \\[.8cm]
    \begin{tikzpicture}[scale=.4, baseline=.5cm]
    	\draw [thick] (0,0)--(0,1); 
    	\draw [thick] (0,1)--(1,1); 
    	\draw [thick] (1,1)--(1,2); 
    	\draw [thick] (1,2)--(1,3); 
    	\draw [thick] (1,3)--(2,3); 
    	\draw [thick] (2,3)--(3,3); 
 		\foreach \x in {0, ..., 3}
    		\foreach \y in {\x, ..., 3}
    			\node (\x,\y) at (\x,\y) [inner sep=0pt] {$\scriptscriptstyle \bullet$};
    \end{tikzpicture} 
    &
    \begin{tikzpicture}[scale=.4, baseline=.5cm]
    	\draw [thick] (0,0)--(0,1); 
    	\draw [thick] (0,1)--(1,1); 
    	\draw [thick] (1,1)--(1,2); 
    	\draw [ultra thick, color=red] (1,2)--(1,3); 
    	\draw [thick] (1,3)--(2,3); 
    	\draw [thick] (2,3)--(3,3); 
 		\foreach \x in {0, ..., 3}
    		\foreach \y in {\x, ..., 3}
    			\node (\x,\y) at (\x,\y) [inner sep=0pt] {$\scriptscriptstyle \bullet$};
    \end{tikzpicture} 
    &
    \begin{tikzpicture}[scale=.4, baseline=.5cm]
    	\draw [thick] (0,0)--(0,1); 
    	\draw [thick] (0,1)--(1,1); 
    	\draw [ultra thick, color=red] (1,1)--(1,2); 
    	\draw [thick] (1,2)--(1,3); 
    	\draw [thick] (1,3)--(2,3); 
    	\draw [thick] (2,3)--(3,3); 
 		\foreach \x in {0, ..., 3}
    		\foreach \y in {\x, ..., 3}
    			\node (\x,\y) at (\x,\y) [inner sep=0pt] {$\scriptscriptstyle \bullet$};
    \end{tikzpicture} 
    &
    \\[.8cm]
    \begin{tikzpicture}[scale=.4, baseline=.5cm]
    	\draw [thick] (0,0)--(0,1); 
    	\draw [thick] (0,1)--(0,2); 
    	\draw [thick] (0,2)--(0,3); 
    	\draw [thick] (0,3)--(1,3); 
    	\draw [thick] (1,3)--(2,3); 
    	\draw [thick] (2,3)--(3,3); 
 		\foreach \x in {0, ..., 3}
    		\foreach \y in {\x, ..., 3}
    			\node (\x,\y) at (\x,\y) [inner sep=0pt] {$\scriptscriptstyle \bullet$};
    \end{tikzpicture} 
\end{array}
$}
	\caption{For $n=3$, the 12 pipe dreams whose permutations are factored into identity atomic permutations, and the corresponding 12 colored Dyck paths.}
	\label{fig:pipe_walks}
\end{figure}

In the following two sections we describe alternative combinatorial models to the considered families of pipe dreams and lattice walks, which will be used to prove Conjecture~\ref{conj:walks2} for the special values $k=1,2,n-1$, and $n$.


\subsection{Pipe dreams and bounce Dyck paths}
\label{subsec:bouncepairs}

The family of pipe dreams related to lattice walks on the quarter plane is also related to a special family of pairs of Dyck paths.
We say that a Dyck path $\pi$ is a \defn{bounce path} if it is of the form~$N^{i_1}E^{i_1} N^{i_2}E^{i_2} \dots N^{i_k}E^{i_k}$ for arbitrary positive integers $i_1, \dots, i_k$.
The number of parts of a bounce path is the number~$k$ of smaller paths~$N^{i_j}E^{i_j}$ that are being concatenated. 
A pair of Dyck paths $(\pi_1,\pi_2)$ is said to be \defn{nested} if $\pi_1$ is weakly below $\pi_2$.

\begin{lemma}
\label{lem:bouncepairs}
The following two families are in bijective correspondence:
\begin{enumerate}
\item Pipe dreams $P$ such that $\omega_P\in \fS_n$ is a permutation whose factorization into atomics consists of $k$ identity permutations.
\item Nested pairs of Dyck paths $(\pi_1,\pi_2)$ of size $n$ such that $\pi_1$ is a bounce path with $k$ parts.
\end{enumerate}
\end{lemma}

This lemma is a special case of a stronger result (Corollary~\ref{cor:dominantPairsDyckPaths}) relating a bigger family of pipe dreams with nested pairs of Dyck paths which is presented in Section~\ref{sec:nuTamari}.
This lemma enables to translate Conjecture~\ref{conj:walks2} into the following refined enumeration of lattice paths.

\begin{corollary}
\label{cor:walksEnumerationDeterminants}
If Conjecture~\ref{conj:walks2} holds, the number of walks in the quarter plane starting at $(0,0)$ and ending on the horizontal axis, consisting of $2n$ steps taken from $\{(-1, 1), (1, -1), (0, 1)\}$, is a refined sum of $2^{n-1}$ determinants. 
\end{corollary}

\begin{proof}
By Conjecture~\ref{conj:walks2} and Lemma~\ref{lem:bouncepairs} the desired number of walks is equal to the number of nested pairs of Dyck paths $(\pi_1,\pi_2)$ of size $n$ such that $\pi_1$ is a bounce path. 
For a fixed $\pi_1$, a result of Kreweras~\cite{Kreweras} implies that number of such pairs is a determinant whose entries depend on the partition bounded above $\pi_1$ (see for instance~\cite[Sect.~5]{CeballosGonzalezDLeon} for more details).
As there are $2^{n-1}$ possible bounce paths $\pi_1$, we get a sum of $2^{n-1}$ determinants.
\end{proof}


\subsection{Lattice walks, colored Dyck paths, and steep Dyck paths}
\label{sec:coloredDyckpaths}

A \defn{colored Dyck path} of size $n$ is a Dyck path of size $n$ with some red colored north steps, such that at each step of the path the number of red north steps is less than or equal to the number of horizontal steps. A Dyck path is \defn{steep} if it contains no consecutive east steps $EE$, except at the end on top of the grid. 
The following statement is illustrated on \fref{fig:walks}.

\begin{lemma}
\label{lem:coloredDyckpaths}
The following three families are in bijective correspondence:
\begin{enumerate}
\item Walks in the quarter plane starting at $(0,0)$ and ending on the horizontal axis, consisting of $2n$ steps taken from $\{(-1, 1), (1, -1), (0, 1)\}$ from which $k$ of them are $(0,1)$. 
\item Colored Dyck paths of size $n$ with $k$ black north steps.
\item Nested pairs of Dyck paths $(\pi_1,\pi_2)$ of size $n$ such that $\pi_2$ is a steep path with $n-k$ isolated east steps strictly below the top.
\end{enumerate}
\end{lemma}

\begin{proof}
The bijection between (1) and (2) is easily obtained by mapping steps of the walk to colored steps of a Dyck path as follows: $(0,1)$ to uncolored $N$, $(1,-1)$ to $E$, and $(-1,1)$ to red $\r{N}$.
The condition on the number of red north steps being less than or equal to the number of east steps, and that the resulting path is a Dyck path, come from the fact that the walk lies inside the quarter plane and finishes on the $x$-axis. 

The bijection between (2) and (3) is obtained by mapping a colored Dyck path $\pi$ to the pair $(\pi_1,\pi_2)$ where $\pi_1$ is the Dyck path $\pi$ forgetting the colors, and $\pi_2$ is the steep Dyck path with single east steps in the rows given by 
the red  $\r{N}$ steps of $\pi$. The condition on the coloring of $\pi$ guaranties that $\pi_1$ and $\pi_2$ are nested. This procedure is clearly invertible (see Figure~\ref{fig:walks})
\begin{figure}[t]
	\capstart
	\centerline{\includegraphics[scale=.4]{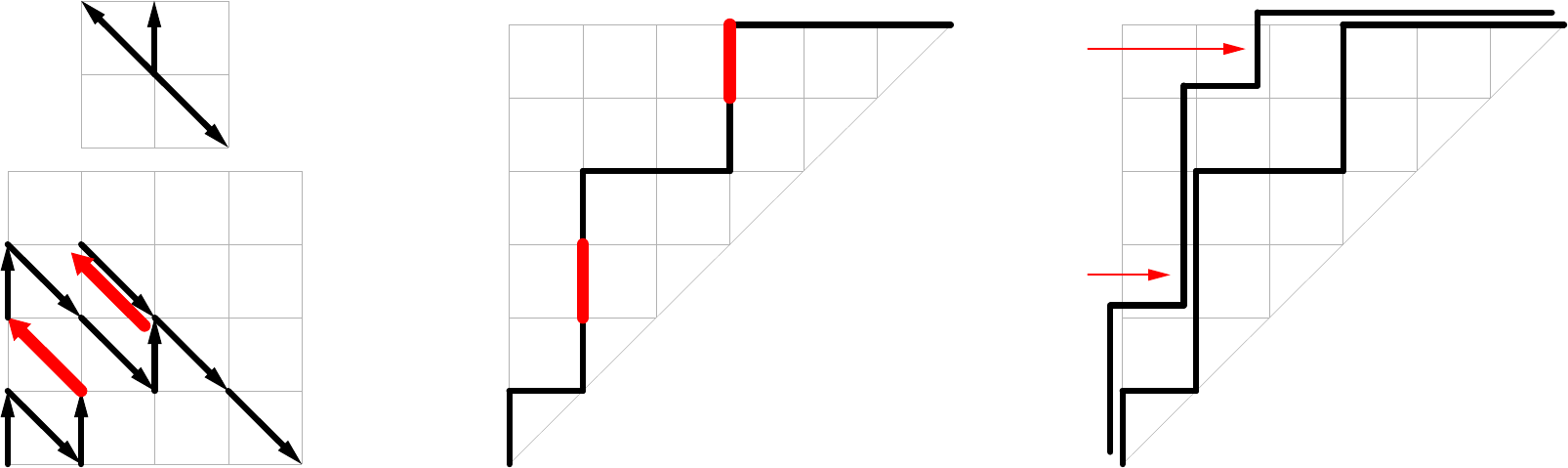}}
	\caption{(Left) A walk in the quarter plane starting at $(0,0)$, ending on the horizontal axis, and consisting of steps taken from $\{(-1, 1), (1, -1), (0, 1)\}$. (Middle) A colored Dyck path such that at each step, the number of red north steps is at most the number of horizontal steps. (Right) A nested pair of Dyck paths~$(\pi_1,\pi_2)$ where~$\pi_2$ is steep. The three objects are connected by Lemma~\ref{lem:coloredDyckpaths}: the $(-1,1)$ steps in the walk become red north steps in the colored Dyck path, which correspond to east steps in~$\pi_2$.}
	\label{fig:walks}
\end{figure}
\end{proof}


\subsection{The Steep-Bounce Conjecture}
\label{sec_steepbounce_conjecture}

\enlargethispage{.5cm}
Using Lemmas~\ref{lem:bouncepairs} and~\ref{lem:coloredDyckpaths}, Conjecture~\ref{conj:walks2} can now be reformulated as follow\footnote{Since the completion of this paper, Conjecture~\ref{conj:SteepBounce} has been solved in~\cite{CeballosFangMuhle}.}.

\begin{conjecture}
\label{conj:SteepBounce}
For any $k\le n$, there is a bijection between the following two sets
\begin{enumerate}
\item Nested pairs of Dyck paths $(\pi_1,\pi_2)$ of size $n$ such that $\pi_1$ is a bounce path with $k$ parts.
\item Nested pairs of Dyck paths $(\pi'_1,\pi'_2)$ of size $n$ such that $\pi'_2$ is a steep path with~$k$ east steeps on top of the grid.
\end{enumerate}
\end{conjecture}

Before giving some evidence supporting this conjecture let us recall the famous zeta map in $q,t$-Catalan combinatorics.
The $q,t$-Catalan polynomials $C_n(q,t)$ are certain polynomials in two variables whose evaluations at $(q,t)=(1,1)$ recover the Catalan numbers.
They appeared as the bivariate Hilbert series of diagonal harmonic alternants in the theory of diagonal harmonics~\cite{Haglund-qt-catalan} 
(see more details in Section~\ref{subsec:definitionsMultivariateDiagonalHarmonics}).
Finding an explicit description of these polynomials has been a particularly difficult problem.
An explicit rational expression for $C_n(q,t)$ conjectured by A.~Garsia and M.~Haiman~\cite{GarsiaHaiman-remarkableCatalanSequence} 
follows from M.~Haiman's work~\cite{Haiman-vanishingTheorems}. A combinatorial interpretation of this rational expression was conjectured by 
Haglund~\cite{haglund_conjectured_2003} and proved by A.~Garsia and J.~Haglund~\cite{garsia_proof_2002} using plethystic  machinery developed by F.~Bergeron~et~al.~\cite{bergeron_identities_1999}. 
J.~Haglund's combinatorial interpretation for the $q,t$-Catalan polynomial, which uses a pair of statistics on Dyck paths known as $\area$ and $\bounce$, was discovered after prolonged investigations of tables of $C_n(q,t)$. Shortly after, M.~Haiman announced another combinatorial interpretation in terms of two statistics $\area$ and $\dinv$. Unexpectedly, these two pairs of statistics were different but gave rise to the same expression:
\[
C_n(q,t) = \sum_{\text{Dyck paths }\pi} q^{\area(\pi)} t^{\bounce(\pi)}
= \sum_{\text{Dyck paths }\pi} q^{\dinv(\pi)} t^{\area(\pi)}.
\]
The zeta map $\zeta$ is a bijection on Dyck paths that explains this phenomenon by sending the pair of statistics~$(\area,\dinv)$ to the pair of statistics~$(\bounce,\area)$.

The inverse of the zeta map appeared first in the work of G.~Andrews, C.~Krattenthaler, L.~Orsina and P.~Papi~\cite{AndrewsKrattenthalerOrsinaPapi} in connection with nilpotent ideals in certain Borel subalgebra of $sl(n)$, and was rediscovered by J.~Haglund and M.~Haiman in their study of diagonal harmonics and $q,t$-Catalan numbers~\cite{Haglund-qt-catalan}.
We refer to~\cite[Proof of Theorem~3.15]{Haglund-qt-catalan} for a precise description of the zeta map and its inverse, and to~\cite[Chapt.~3]{Haglund-qt-catalan} for more details about the $\bounce$ and $\dinv$ statistics.

Now we connect the zeta map to the Steep-Bounce Conjecture~\ref{conj:SteepBounce}.
A bounce path $\pi$ with~$k$ parts is determined by the sequence $0=b_0<b_1<\dots<b_{k-1}<b_k=n$ where $(b_i,b_i)$ are the diagonal points in $\pi$. 
We denote by $B(\pi) \eqdef \{b_1,b_2,\dots,b_{k-1}\} \subseteq [n-1]$ the set of $b_i$'s excluding the values $b_0=0$ and $b_k=n$. 
Similarly, a steep path $\pi'$ with $n-k$ isolated east steps strictly below the top is determined by the sequence $1\leq b_1' < b_2' < \dots < b_{n-k}' \leq n-1$ 
of the $y$-coordinates of the isolated east steps, and we denote
$B'(\pi') \eqdef \{b_1',b_2',\dots,b_{n-k}'\} \subseteq [n-1]$. 

\begin{proposition}
\label{prop:bijectionBounceSteep}
The inverse of the zeta map is a bijection between:
\begin{enumerate}
\item Bounce paths $\pi$ with $k$ parts.
\item Steep paths $\pi'$ with $n-k$ isolated east steps strictly below the top.
\end{enumerate}
Moreover, $\pi'=\zeta^{-1}(\pi)$ if and only if $B'(\pi')=[n-1]\setminus B(\pi)$.
\end{proposition}

\begin{proof}
Let $\pi$ be a bounce path with $k$ parts such that $B(\pi) \eqdef \{b_1,b_2,\dots,b_{k-1}\}$. As above, we take $b_0=0$ and $b_k=n$. 
The area sequence of a Dyck path $\pi'$ of size $n$ is defined as the $n$-tuple whose $i$th entry counts the number of boxes in row $i$ that are between $\pi'$ and the main diagonal of the grid.
By the description of the inverse zeta map in~\cite[Proof of Theorem~3.15]{Haglund-qt-catalan}, the path $\pi'=\zeta^{-1}(\pi)$ is the path whose area sequence consists of $b_1-b_0$ zeros, followed by $b_2-b_1$ ones, followed by $b_3-b_2$ twos, and so on until $b_k-b_{k-1}$ values $k-1$. 
Therefore, the path $\pi'$ can be obtained from the diagonal path $(NE)^n$ by contracting the east steps with $y$-coordinates in $B(\pi)$. 
As a consequence, $\pi'$ is a steep path satisfying $B'(\pi')= [n-1]\setminus B(\pi)$. Since bounce paths and steep paths area completely characterized by their associated sets $B$ and $B'$, the result follows. 
\end{proof}

Given a Dyck path $\pi$ there is a unique nested pair $(\pi_{\bounce},\pi)$ such that $\pi_{\bounce}$ is a bounce path and 
the set $B(\pi_{\bounce}) = \{b_1,b_2,\dots,b_{k-1}\}$ is such that $b_1$ is maximal, and among those $b_2$ is maximal, and so
on until $b_{k-1}$. The bounce statistic of $\pi$ is then defined to be ${\bounce(\pi) \eqdef \sum(n-b_i)}$.
Similarly, given a Dyck path $\pi'$ there is a unique nested pair  $(\pi',\pi'_{\steep})$ such that 
$\pi'_{\steep}$ is a steep path and the set $B'(\pi'_{\steep})=\{b_1',b_2',\dots,b_{n-k}'\}$ 
is such that $b_1'$ is minimal, and among those $b_2'$ is minimal, and so on until $b_{n-k}'$.
The following proposition is a stronger evidence for Conjecture~\ref{conj:SteepBounce}.

\begin{proposition}
The inverse of the zeta map induces a bijection (determined by $\pi'=\zeta^{-1}(\pi)$) between:
\begin{enumerate}
\item Pairs $(\pi_{\bounce},\pi)$ such that $\pi_{\bounce}$ has $k$ parts.
\item Pairs $(\pi',\pi'_{\steep})$ such that $\pi'_{\steep}$ has $n-k$ isolated east steps strictly below the top.
\end{enumerate}
\end{proposition}

\begin{proof}
By the description of the inverse zeta map in~\cite[Proof of Theorem~3.15]{Haglund-qt-catalan}, 
$\pi_{\bounce}$ has $k$ parts if and only if the area sequence of $\pi'=\zeta^{-1}(\pi)$ contains all values $0,1,2,\dots, k-1$ (with repetitions in some order). 
It is not difficult to see that the second statement holds if and only if $\pi'_{\steep}$ has has $n-k$ isolated east steps strictly below the top.
\end{proof}

\begin{remark}
\label{rem_bouncek_heightk}
As pointed out in our previous proof, the inverse zeta map gives a bijection between paths $\pi$ whose bounce path has $k$ parts and paths $\pi'$ whose area sequence contains all values $0,1,2,\dots, k-1$. Such paths are commonly referred to as Dyck paths of height $k$.
This observation was noticed in current work in progress by M. Kallipoliti, R. Sulzgruber and E. Tzanaki in the context of pattern avoidance in Shi tableaux.
\end{remark}

\begin{remark}
It would be interesting to define a steep path statistic of a Dyck path $\pi'$ in terms of the pair $(\pi',\pi'_{\steep})$, whose marginal distribution is equal to the marginal distribution of the area statistic on Dyck paths. This could be useful towards a combinatorial proof of the $q,t$-symmetry of the Catalan numbers~\cite[Open~Problem~3.11]{Haglund-qt-catalan}.
\end{remark}

\begin{remark}
We want to emphasize that the following two naive approaches to Conjecture~\ref{conj:SteepBounce} already fail for~$n = 4$.
\begin{itemize}
\item For a bounce path~$\pi$ and its corresponding steep path~$\pi' = \zeta^{-1}(\pi)$ (see Proposition~\ref{prop:bijectionBounceSteep}), the number of Dyck paths above~$\pi$ does not coincide with the number of Dyck paths below~$\pi'$. In fact, the distributions of the number of Dyck paths above a bounce path and the number of Dyck paths below a steep path do not coincide. Therefore, we cannot use a bijection between bounce paths and steep paths to construct a bijection for Conjecture~\ref{conj:SteepBounce}.
\item For a Dyck path~$\pi$, and its corresponding Dyck path~$\pi' = \zeta^{-1}(\pi)$, the number of bounce paths below~$\pi$ does not coincide with the number of steep paths above~$\pi'$. In fact, the distributions of the number of bounce paths below a Dyck path and the number of steep paths above a Dyck path do not coincide. Therefore, we cannot use a bijection between Dyck paths to construct a bijection for Conjecture~\ref{conj:SteepBounce}.
\end{itemize}
\end{remark}

\begin{remark}
\label{rem_step_zeta_map}
Our intuition for a proof of Conjecture~\ref{conj:SteepBounce} is that there should be a ``steep-zeta map'' $Z$ on the pairs $(\pi_1',\pi_2')$ satisfying 
\[
Z(\pi',\pi'_{\steep}) = (\pi_{\bounce},\pi)
\]
with $\pi=\zeta(\pi')$. The map $Z$ can therefore be thought as a generalization of the zeta map on pairs of Dyck paths that acts on $\pi_1'$ depending on its relative position with respect to the steep path $\pi_2'$.
\end{remark}


\subsection{The cases $k=1,2,n-1,n$}

We now prove Conjecture~\ref{conj:walks2} in four special cases.

\begin{proposition}
Conjecture~\ref{conj:walks2} holds in the special cases $k=1,2, n-1,n$.
\end{proposition}

\begin{proof}
Our proof is rather enumerative than bijective. 
Let $A_{n,k}$ denote the number of pipe dreams~$P$ such that $\omega_P\in \mathfrak{S}_n$ factorizes into $k$ identity permutations.
We also denote by $B_{n,k}$ the number of colored Dyck paths of size $n$ with $k$ black north steps (and $n-k$ red north steps). 
We will show that:

\begin{enumerate}
\item $A_{n,1}=B_{n,1}=1$,
\item $A_{n,2}=B_{n,2}=2^n-2$,
\item $A_{n,n-1}=B_{n,n-1}=(n+1)C_n-C_{n+1}$,
\item $A_{n,n}=B_{n,n}=C_n$.
\end{enumerate}
Where $C_n \eqdef \frac{1}{n+1}{2n \choose n}$ denotes the $n$th Catalan number.

The first equality $(1)$ follows from the fact that there is only one pipe dream $P$ such that $\omega_P=12\dots n$ ($P$ consists only of elbows), and only one colored Dyck path $D$ with one black north step and $n-1$ red north steps ($D$ is the diagonal path with all but the first north steps colored red).

Equality $(4)$ follows from the correspondence between pipe dreams with permutation ${w = n \dots 2 1}$ and complete binary trees with $n$ internal nodes, as mentioned in Section~\ref{subsec:LodayRonco}. 
These are counted by the Catalan number $C_n$ and are in bijection with Dyck paths of size $n$, which are equivalently colored Dyck paths with all $n$ north steps colored black.

For equalities $(2)$ and $(3)$ we use the alternative models in Sections~\ref{subsec:bouncepairs} and~\ref{sec:coloredDyckpaths}.  
Let us start considering the case $k=2$. By Proposition~\ref{lem:bouncepairs}, $A_{n,2}$ is equal to the number of pairs of nested Dyck paths $(\pi_1,\pi_2)$ of size $n$ such that $\pi_1$ is a bounce path with two parts. If we denote by $A_{n,2}^i$ the number of such pairs such that $\pi_1=(N^iE^i)(N^{n-i}E^{n-i})$ for $1\leq i \leq n-1$, then $A_{n,2}^i={n \choose i}$. Therefore 
\[
A_{n,2}= \sum_{i=1}^{n-1} A_{n,2}^i =  \sum_{i=1}^{n-1} {n \choose i} = 2^n-2.
\]
On the other hand, $B_{n,2}$ is equal to the number of colored Dyck paths of size $n$ with $2$ black north steps and $n-2$ red north steps. 
The first north step of any colored Dyck path is always forced to be black. If we denote by $B_{n,2}^i$ the number of such colored Dyck paths where the second black north step is in row $i$ for $2\leq i \leq n$, then it is not hard to check that $B_{n,2}^i=2^{n-i+1}$ (because every north step in a row $j\geq i$ has two possible placements whereas all others only have one). Therefore, 
\[
B_{n,2}= \sum_{i=2}^{n} B_{n,2}^i =  \sum_{i=2}^{n} 2^{n-i+1} = 2^n-2.
\]

Finally, let us consider the case $k=n-1$. By Proposition~\ref{lem:bouncepairs}, $A_{n,n-1}$ is equal to the number of pairs of nested Dyck paths $(\pi_1,\pi_2)$ of size $n$ such that $\pi_1$ is a bounce path with $n-1$ parts. Exactly one of these parts is of the form $N^2E^2$ and all others are just $NE$.
Let $A_{n,n-1}^i$ be the number of such nested pairs $(\pi_1,\pi_2)$ such that $\pi_1$ is a bounce path whose $i$th part is $N^2E^2$ for some $1\leq i \leq n-1$.
We also denote $A_{n,n-1}^0=A_{n,n-1}^n=0$ for convenience. The only paths $\pi_2$ that do not contribute to $A_{n,n-1}^i$ are the ones touching the diagonal of the grid at the point $(i,i)$. Therefore, we have
\[
A_{n,n-1}^i = C_n - \tilde A_{n,n-1}^i,
\]
where $\tilde A_{n,n-1}^i=C_i C_{n-i}$ is the number of Dyck paths of size $n$ containing the diagonal point $(i,i)$. 
Note that this equality also holds for $i=0$ and $i=n$.
Summing over $i$ we obtain
\[
A_{n,n-1} = \sum_{i=0}^{n} A_{n,n-1}^i = 
\sum_{i=0}^{n} (C_n - \tilde A_{n,n-1}^i) = 
(n+1)C_n - C_{n+1}. 
\]
On the other hand, let $B_{n,n-1}$ the number of colored Dyck paths of size $n$ with $n-1$ black north steps and $1$ red north step. 
Denote by $B_{n,n-1}^i$ the number of such colored Dyck paths such the the red north step is in row $i$ for $2\leq i\leq n$. 
We also denote $B_{n,n-1}^0=B_{n,n-1}^1=0$ for convenience.
The only paths whose north step at row $i$ can not be colored red are exactly those containing the point $(0,i)$. Therefore,
\[
B_{n,n-1}^i = C_n - \tilde B_{n,n-1}^i,
\]
where $\tilde B_{n,n-1}^i$ is the number of Dyck paths of size $n$ containing the point $(0,i)$. 
Note that this equality also holds for $i=0$ and $i=1$.
Summing over $i$ we obtain
\[
B_{n,n-1} = \sum_{i=0}^{n} B_{n,n-1}^i = 
\sum_{i=0}^{n} (C_n - \tilde B_{n,n-1}^i) = 
(n+1)C_n - C_{n+1}. 
\]
The last equality follows from $C_{n+1}=\sum_{i=0}^{n} \tilde B_{n,n-1}^i$, which is obtained as the refined counting of Dyck paths of size $n+1$ according 
to when the path leaves the $y$-axis. 
\end{proof}
 

\section{Hopf subalgebra of acyclic pipe dreams}
\label{sec:acyclicPipeDreams}

We conclude this part with a last family of Hopf subalgebras of the Hopf algebra on pipe dreams.
The \defn{contact graph} of a pipe dream~$P$ is the directed graph~$P\contact$ with one node for each pipe of~$P$ 
and one arc for each elbow of~$P$ connecting the south-east pipe to the north-west pipe of the elbow. 
A pipe dream~$P$ is called \defn{acyclic} when its contact graph~$P\contact$ is acyclic (no oriented cycle). 
For~$\omega \in \fS$, we denote by~$\acyclicPipeDreams(\omega)$ the set of acyclic pipe dreams of~$\pipeDreams(\omega)$, and we let~$\acyclicPipeDreams_n \eqdef \bigsqcup_{\omega \in \fS_n} \acyclicPipeDreams(\omega)$ and~$\acyclicPipeDreams \eqdef \bigsqcup_{n \in \N} \acyclicPipeDreams_n$.

\begin{remark}
A special class of acyclic pipe dreams is the collection of reversing pipe dreams considered in Section~\ref{subsec:LodayRonco}. 
The contact graph~$P\contact$ of a reversing pipe dream~$P$ is equal to its corresponding complete binary tree~$\Psi(P)$ (see Figure~\ref{fig:bijection}) after removing all leaves and orienting all edges towards the root.
Such oriented binary trees are clearly acyclic.
\end{remark}

\begin{lemma}
\label{lem:singleSource}
The contact graph~$P\contact$ of an acyclic pipe dream~$P \in \acyclicPipeDreams_n$ has a unique sink.
\end{lemma}

\begin{proof}
Observe first that any pipe of~$P$ which is different than~$0$ must be the south-east pipe of at least one elbow.
Otherwise, such a pipe would bound a rectangle~$R$.
Since~$P$ is reduced, all pipes entering~$R$ from south should leave $R$ towards north, which is impossible since pipe~$0$ uses the first way out.
This proves that~$P\contact$ has a unique sink corresponding to the pipe~$0$ which passes through the top left elbow.
\end{proof}

\begin{lemma}
\label{lem:packingAcyclic}
For any acyclic pipe dream~$P \in \acyclicPipeDreams_n$ and any~$k \in \{0, \dots, n\}$, the horizontal packing~$\r{\lrot_k(P)}$ and the vertical packing~$\b{\lmir_k(P)}$ are both acyclic.
\end{lemma}

\begin{proof}
As explained in the proof of Lemma~\ref{lem:horizontalPacking}, each elbow~$\elbowBiColor{red}{red}$ in the horizontal packing~$\r{\lrot_k(P)}$ either corresponds to two consecutive bicolored elbows~$\elbowBiColor{red}{blue,dashed} \cdots \elbowBiColor{blue,dashed}{red}$ in the pipe dream~$P$ that where merged together, or to a red contact~$\elbowBiColor{red}{red}$ in~$P$. In other words, each arc of the contact graph~$\r{\lrot_k(P)}\contact$ corresponds to either two arcs or just one arc in the contact graph~$P\contact$. Therefore, the acyclicity of~$P\contact$ implies the acyclicity of~$\r{\lrot_k(P)}\contact$. The proof for the vertical packing follows by symmetry.
\end{proof}

\begin{lemma}
\label{lem:starAcyclic}
For any acyclic pipe dreams~$\b{P}, \r{Q}$ and any $\b{P}/\r{Q}$-shuffle~$s$, the pipe dream~$\b{P} \star_{s} \r{Q}$ is acyclic.
\end{lemma}

\begin{proof}
Following the description of~$\b{P} \star_{s} \r{Q}$, let~$\r{Q_1}, \dots, \r{Q_\ell}$ denote the pipe dreams obtained by untangling~$\r{Q}$ at the gaps of~$\omega_\r{Q}$ marked by $\b{p}$-blocks in~$s$. Then the contact graph~$(\b{P} \star_{s} \r{Q})\contact$ is obtained by relabeling the contact graphs~$\b{P}\contact$, $\r{Q_1}\contact, \dots, \r{Q_\ell}\contact$ and connecting the sink vertex of each~$\r{Q_i}\contact$ to a vertex of~$\b{P}\contact$. Since~$\b{P}\contact$ and $\r{Q_1}\contact, \dots, \r{Q_\ell}\contact$ are acyclic (by Lemma~\ref{lem:packingAcyclic}), we obtain that~$(\b{P} \star_{s} \r{Q})\contact$ is acyclic. 
\end{proof}

\begin{proposition}
The subspace~$\HA$ of acyclic pipe dreams is a Hopf subalgebra of~$\HP$.
\end{proposition}

\begin{proof}
Lemmas~\ref{lem:packingAcyclic} and~\ref{lem:starAcyclic} immediately imply that acyclic pipe dreams are stable by product and coproduct.
\end{proof}

Merging the results of Sections~\ref{sec:subalgebrasAtoms} and~\ref{sec:acyclicPipeDreams}, we obtain the following general statement.

\begin{corollary}
For any set~$S$ of atomic permutations, the subspace~$\HA \langle S \rangle$ of acyclic pipe dreams with restricted atom sets contained in $S$ is a Hopf subalgebra of~$(\HP, \product, \coproduct)$.
\end{corollary}


\part{The dominant pipe dream algebra and its connections to multivariate diagonal harmonics}
\label{part:multivariateDiagonalHarmonics}

In the final part of this paper, we study a Hopf subalgebra obtained from dominant permutations. 
This Hopf algebra contains all Hopf algebras arising from atom sets described above. 
In Section~\ref{sec:nuTamari}, we will show that it is isomorphic to a generalization of the Loday--Ronco Hopf algebra on a special 
family of binary trees called $\nu$-trees, which are related to the $\nu$-Tamari lattices recently introduced 
by L.-F.~Pr\'eville-Ratelle and X.~Viennot in~\cite{PrevilleRatelleViennot}. In Section~\ref{sec:multivariateDiagonalHarmonics}, 
we will present applications of this Hopf algebra to the theory of multivariate diagonal harmonics~\cite{Bergeron-multivariateDiagonalCoinvariantSpaces}.


\section{Hopf subalgebra from dominant permutations and $\nu$-Tamari lattices}
\label{sec:nuTamari}


\subsection{The Hopf algebras of dominant permutations and dominant pipe dreams}
\label{subsec:dominantPermutationsAndDominantPipeDreams}

Recall that the \defn{Rothe diagram} of a permutation~$\omega \in \fS_n$ is the set
\[
R_\omega \eqdef \set{\big( \omega(i), j \big)}{i > j \text{ and } \omega(i) < \omega(j)}.
\]
If we represent this diagram in matrix notation (\ie the box~$(i,j)$ appears in row~$i$ and column~$j$), then the Rothe diagram of~$\omega$ is 
the set of boxes which are not weakly below or weakly to the right of a box~$\big( \omega(i), i \big)$ for all~$i \in [n]$. See \fref{fig:dominantPermutation}.

A permutation~$\omega \in \fS_n$ is \defn{dominant} if its Rothe diagram~$R_\omega$ is a partition containing the top-left corner. Such a permutation is uniquely determined by its Rothe diagram~$R_\omega$, or equivalently by the Dyck path~$\pi_\omega$ delimiting the boundary of its Rothe diagram. See \fref{fig:dominantPermutation}.

\begin{figure}[ht]
	\capstart
	\centerline{\includegraphics[scale=.8]{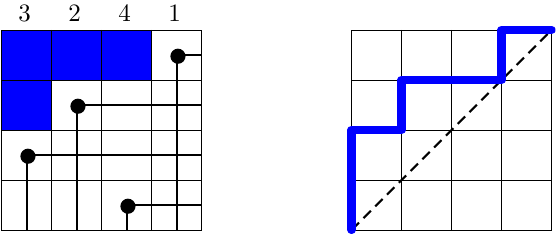}}
	\caption{The Rothe diagram~$R_\omega$ and the Dyck path~$\pi_\omega$ of the dominant permutation $\omega = 3241$.}
	\label{fig:dominantPermutation}
\end{figure}

We denote the set of dominant permutations by
\[
\Sdom_n \eqdef \set{\omega \in \fS_n}{\omega \text{ is dominant}}
\qquad\text{and}\qquad
\Sdom \eqdef \bigsqcup_{n \in \N} \Sdom_n.
\]
The next statement describes the product~$\gdproduct$ of Section~\ref{subsec:globalDescents} on dominant permutations.

\begin{proposition}
\label{prop:gdproductDominant}
The product~$\gdproduct$ on dominant permutations corresponds to the concatenation on the corresponding Dyck paths. More precisely,
\begin{enumerate}[(i)]
\item $\omega \in \Sdom$ has a global descent at~$\gamma$ if an only if~$\pi_\omega$ contains the diagonal point~$(\gamma, \gamma)$,
\item for~$\mu, \nu \in \Sdom$, the Dyck path~$\pi_{\mu \gdproduct \nu}$ is the concatenation~$\pi_\mu \gdproduct \pi_\nu$ of the Dyck paths~$\pi_\mu$ and~$\pi_\nu$,
\item for~$\omega \in \Sdom$, the unique atomic factorization~$\omega = \omega_1 \gdproduct \omega_2 \gdproduct \cdots \gdproduct \omega_\ell$ corresponds to the unique decomposition of~$\pi_\omega = \pi_{\omega_1} \gdproduct \pi_{\omega_2} \gdproduct \cdots \gdproduct \pi_{\omega_\ell}$ into Dyck paths not returning to the diagonal.
\end{enumerate}
\end{proposition}

\begin{proof}
For~(i), observe that, by definition of dominant permutations, the box at row~$x$ and column~$y$ belongs to the Rothe diagram~$R_\omega$ if and only if for all~$i \in [n]$, either~$\omega(i) > x$ or~$i > y$. In particular, $(\gamma, \gamma) \in R_\omega$ if and only if for all~$i \in [n]$, either~$\omega(i) > \gamma$ or~$i > \gamma$, \ie $\gamma$ is a global descent of~$\omega$.

For~(ii), consider~$\omega = \mu \gdproduct \nu$ for some~$\mu \in \fS_m$ and~$\nu \in \fS_n$. Then the entries~$\big( \omega(i), i \big)$ appear in two diagonal blocks, 
namely~$\big( \mu(i) + n, i \big)$ and~$\big( \nu(i), i + m \big)$. 
Therefore, the Rothe diagram~$R_\omega$ can be obtained by gluing the rectangle~$[n] \times [m]$ with~$R_\mu$ below and~$R_\nu$ on the right. It follows that the Dyck path~$\pi_\omega$ is the concatenation of~$\pi_\mu$ and~$\pi_\nu$.

Finally, consider~$\omega \in \Sdom$ with factorization~$\omega = \omega_1 \gdproduct \cdots \gdproduct \omega_\ell$. By~(i), all global descents of~$\omega$ correspond to diagonal points of~$\pi_\omega$. Therefore, they decompose~$\pi_\omega$ into~$\pi_\omega = \pi_1 \gdproduct \cdots \gdproduct \pi_\ell$. Now these Dyck paths define atomic dominant permutations~$\omega'_1, \dots, \omega'_\ell$. By~(ii), the product~$\omega' = \omega'_1 \gdproduct \cdots \gdproduct \omega'_\ell$ is dominant and has Dyck path~$\pi_{\omega'} = \pi_1 \gdproduct \cdots \gdproduct \pi_\ell = \pi_\omega$. Therefore, $\omega = \omega'$ and by uniqueness of the decomposition, $\omega_i = \omega'_i$ for all~$i \in [\ell]$.
\end{proof}

\begin{corollary}
The subspace~$\HSdom$ defines a Hopf subalgebra of~$(\HS, \gdshuffle, \gddeconcat)$. The dimension of the homogeneous component~$\HSdom_n$ is the Catalan number~$C_n \eqdef \frac{1}{n+1} \binom{2n}{n}$.
\end{corollary}

\begin{proof}
It immediately follows from Proposition~\ref{prop:gdproductDominant} as the definitions of the product~$\gdshuffle$ and the coproduct~$\gddeconcat$ only involve the product~$\gdproduct$ and the factorization into atomic permutations.
\end{proof}

\begin{remark}
It is natural to transport the product~$\gdshuffle$ and the coproduct~$\gddeconcat$ on dominant permutations to the corresponding Dyck paths:
\begin{itemize}
\item The coproduct~$\gddeconcat$ of a Dyck path~$\pi$ is given by
\[
\gddeconcat(\pi) \eqdef \sum_{i=0}^\ell (\pi_1 \gdproduct \cdots \gdproduct \pi_i) \otimes (\pi_{i+1} \gdproduct \cdots \gdproduct \pi_\ell),
\]
where~$\pi$ factorizes into~$\pi = \pi_1 \gdproduct \pi_2 \gdproduct \cdots \gdproduct \pi_\ell$ with~$\pi_i$ not returning to the diagonal.
See \fref{fig:coproductDyckPaths}.

\begin{figure}[ht]
	\capstart
	\centerline{
		$\gddeconcat\left(\;\raisebox{-0.45\height}{\includegraphics[scale=.5]{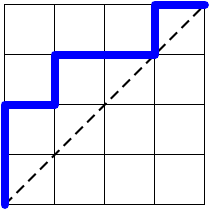}}\;\right) =
		\raisebox{-0.45\height}{\includegraphics[scale=.5]{dyckPath0}} \otimes \epsilon +
		\raisebox{-0.45\height}{\includegraphics[scale=.5]{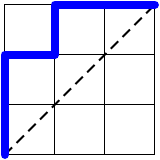}} \otimes \raisebox{-0.45\height}{\includegraphics[scale=.5]{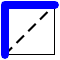}} +
		\epsilon \otimes \raisebox{-0.45\height}{\includegraphics[scale=.5]{dyckPath0}}$
	}
	\caption{The coproduct~$\gddeconcat$ on Dyck paths.}
	\label{fig:coproductDyckPaths}
\end{figure}

\item The product~$\gdshuffle$ of two Dyck paths~$\b{\rho}, \r{\tau}$ is given by~$\b{\rho} \gdshuffle \epsilon \eqdef \b{\rho}$,~$\epsilon \gdshuffle \r{\tau} \eqdef \r{\tau}$ and 
\[
\b{\rho} \gdshuffle \r{\tau} \eqdef \b{\rho_1} \gdproduct (\b{\rho_2} \gdshuffle \r{\tau}) + \r{\tau_1} \gdproduct (\b{\rho} \gdshuffle \r{\tau_2}).
\] 
if~$\b{\rho} = \b{\rho_1} \gdproduct \b{\rho_2}$ and~$\r{\tau} = \r{\tau_1} \gdproduct \r{\tau_2}$ where~$\b{\rho_1}$ and~$\r{\tau_1}$ are non-trivial Dyck paths not returning to the diagonal.
See \fref{fig:productDyckPaths}.

\begin{figure}[ht]
	\capstart
	\centerline{
		$\raisebox{-0.45\height}{\includegraphics[scale=.5]{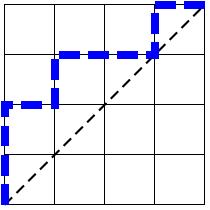}} \gdshuffle \raisebox{-0.45\height}{\includegraphics[scale=.5]{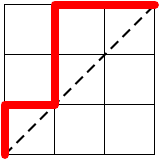}} =
		\raisebox{-0.45\height}{\includegraphics[scale=.3]{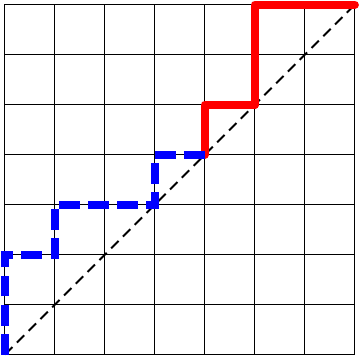}} + \raisebox{-0.45\height}{\includegraphics[scale=.3]{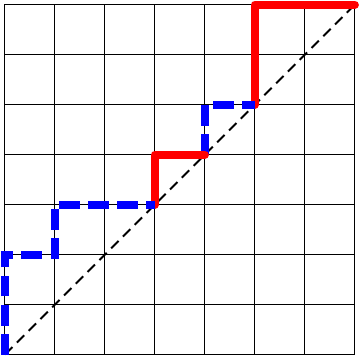}} + \dots + \raisebox{-0.45\height}{\includegraphics[scale=.3]{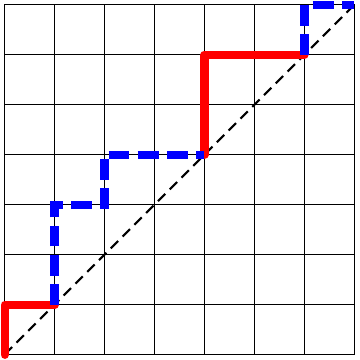}} + \dots + \raisebox{-0.45\height}{\includegraphics[scale=.3]{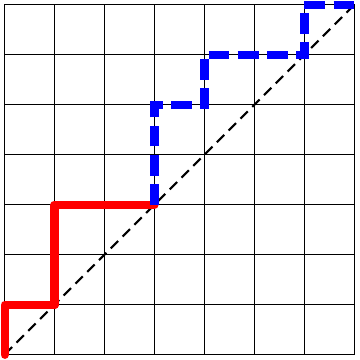}}$
	}
	\caption{The product~$\gdshuffle$ on Dyck paths.}
	\label{fig:productDyckPaths}
\end{figure}
\end{itemize}
\end{remark}

Finally, we pull back this Hopf subalgebra of dominant permutations of~$\HS$ to a Hopf subalgebra of~$\HP$ via the Hopf morphism $\omega \colon \HP \to \HS$ (see Proposition~\ref{prop:descentHopf}). 

A \defn{dominant pipe dream} is a pipe dream~$P$ whose permutation~$\omega_P$ is dominant. We denote the set of dominant pipe dreams by
\[
\pipeDreamsDom_n  \eqdef \set{P \in \pipeDreams_n}{\omega_P \in \Sdom_n}
\qquad\text{and}\qquad
\pipeDreamsDom \eqdef \bigsqcup_{n \in \N} \pipeDreamsDom_n.
\]

\begin{corollary}
The subspace~$\HPdom$ defines a Hopf subalgebra of~$(\HP, \product, \coproduct)$.
\end{corollary}

\begin{remark}
It follows from Proposition~\ref{prop:gdproductDominant} that~$\HSdom$ and~$\HPdom$ can also be viewed as Hopf subalgebras of~$\HS$ and~$\HP$ arising from restricted atom sets. Namely,~$\HSdom = \HS \langle S^{\operatorname{dom}} \rangle$ and~$\HPdom = \HP \langle S^{\operatorname{dom}} \rangle$ where~$S^{\operatorname{dom}}$ is the set of dominant atomic permutations.
\end{remark}


\subsection{The Hopf algebra of $\nu$-trees}
\label{subsec:nuTrees}

We now consider the following family of combinatorial objects defined by C.~Ceballos, A.~Padrol and C.~Sarmiento in~\cite{CeballosPadrolSarmiento-nuTamariSubwordComplexes}.
In the following, we consider a Dyck path drawn on the semi-integer lattice~$(1/2, 1/2) + \Z^2$ and points on the lattice~$\Z^2$.

\begin{definition}[\cite{CeballosPadrolSarmiento-nuTamariSubwordComplexes}]
\label{def:nuTrees}
Let~$\nu$ be a Dyck path of size~$n$ drawn on the semi-integer lattice.
\begin{enumerate}
\item Two lattice points~$p,q$ inside the $n \times n$ grid and weakly above $\nu$ are said \defn{$\nu$-incompatible} if~$p$ is located strictly southwest or northeast to~$q$, and the smallest rectangle containing~$p$ and~$q$ lies above~$\nu$. Otherwise, $p$ and $q$ are called \defn{$\nu$-compatible}.
\item A \defn{$\nu$-tree} is a maximal collection of pairwise $\nu$-compatible lattice points, called \defn{nodes}.
\item Two $\nu$-trees are related by a \defn{rotation} if they differ by only two nodes.
\end{enumerate}
We let~$\nuTrees(\nu)$ the set of $\nu$-trees and we let~$\nuTrees_n \eqdef \bigsqcup_{\nu} \nuTrees(\nu)$ and~$\nuTrees \eqdef \bigsqcup_{n \in \N} \nuTrees_n$.
\end{definition}

A $\nu$-tree $T$ can be viewed as a tree in the graph-theoretical sense by connecting each node~$p$ of~$T$ with the next node of~$T$ below it (if any), and with the next node of~$T$ to its right (if any). Since the nodes of~$T$ are pairwise $\nu$-compatible, the resulting graph is a rooted binary tree with no cycles and no crossings, with its root located at the top-left corner of the grid. For example, we obtain classical complete binary trees when~$\nu = (NE)^n$ is the staircase Dyck path. See Section~\ref{subsec:LodayRonco}. The rotation operation on $\nu$-trees is then similar to the classical rotation on binary trees. \fref{fig:nuTrees} presents some examples. The connection with dominant pipe dreams will be explained in the next section.

\begin{figure}[ht]
	\capstart
	\centerline{$\raisebox{-.5\height}{\includegraphics[scale=0.3]{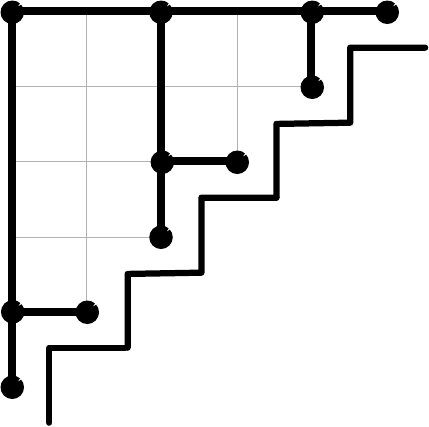}} \qquad \raisebox{-.5\height}{\includegraphics[scale=0.3]{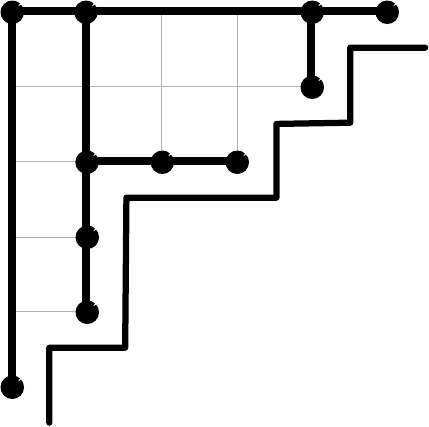}} \qquad \raisebox{-.5\height}{\includegraphics[scale=0.3]{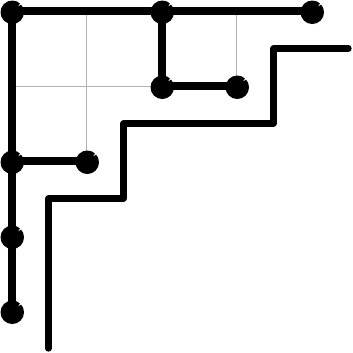}} \qquad \raisebox{-.5\height}{\includegraphics[scale=0.3]{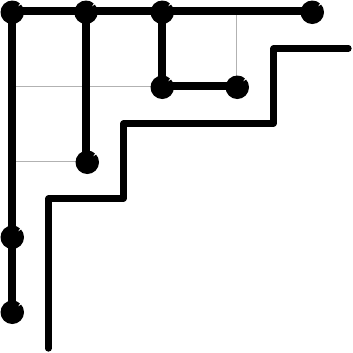}}$}
	\caption{Some $\nu$-trees for different Dyck paths~$\nu$: the first is just the binary tree of \fref{fig:bijection}, the last two are related by a rotation.}
	\label{fig:nuTrees}
\end{figure}

We consider the graded space~$\HT = \bigoplus_{n\ge 0} \HT_n$, where~$\HT_n$ is the~$\bk$-span of all $\nu$-trees with~$\nu$ varying over all Dyck paths of size~$n$.
We now define a product and coproduct on $\HT$ that endow it with a Hopf algebra structure. 
These operations are very similar to that of the Loday--Ronco Hopf algebra~\cite{LodayRonco, AguiarSottile-LodayRonco}, and we call the resulting Hopf algebra the generalized Loday--Ronco Hopf algebra on $\nu$-trees.
The connection to the Hopf algebra of dominant pipe dreams will appear in the next section. 

\subsubsection{Packings on $\nu$-trees}
\label{subsubsec:packingsNuTrees}

A leaf of a $\nu$-tree~$T$ is called a \defn{diagonal leaf} if it belongs to the main diagonal of the $n\times n$ grid. A diagonal leaf~$b$ in~$T$ divides the path~$\nu$ into two paths~$\b{\nu_\ell}$ (on the left) and~$\r{\nu_r}$ (on the right). Cutting the tree~$T$ along the path from~$b$ to its root gives rise to two trees~$\b{\tilde T_\ell}$ (on the left) and $\r{\tilde T_r}$ (on the right). We define the \defn{vertical packing}~$\b{\lmir_b(T)}$ as the $\b{\nu_\ell}$-tree obtained by contracting all vertical segments of $\b{\tilde T_\ell}$ that are above~$b$. Similarly, the \defn{horizontal packing}~$\r{\lrot_b(T)}$ is the $\r{\nu_r}$-tree obtained by contracting all horizontal segments of~$\r{\tilde T_r}$ that are on the left of~$b$. These operations are illustrated on Figures~\ref{fig:horizontalPackingNuTree} and~\ref{fig:verticalPackingNuTree}, and will correspond to the packing operations on dominant pipe dreams as described in Section~\ref{subsec:packings}

\begin{figure}[ht]
	\capstart
	\centerline{$\r{\lrot_b}\left( \raisebox{-.5\height}{\includegraphics[scale=.3]{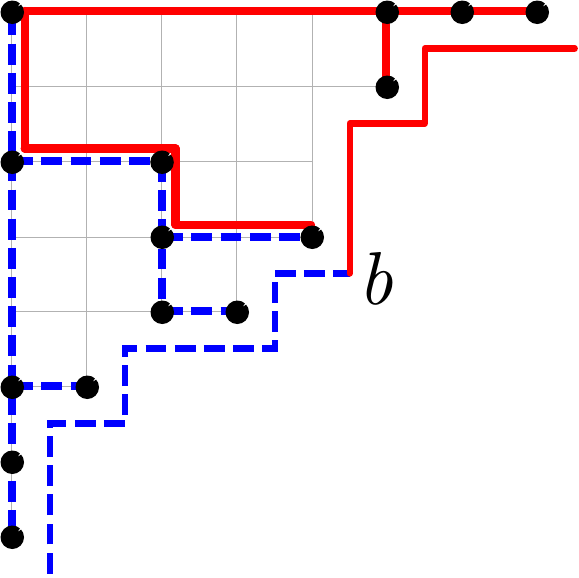}} \right) = \raisebox{-.5\height}{\includegraphics[scale=.3]{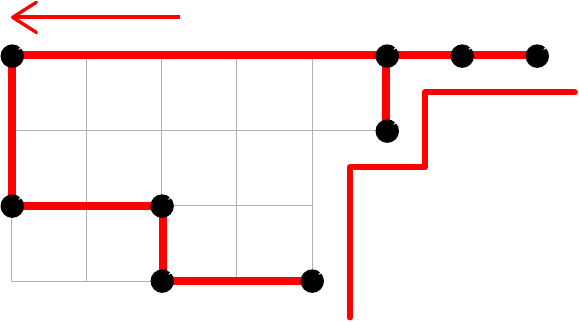}} = \raisebox{-.5\height}{\includegraphics[scale=.3]{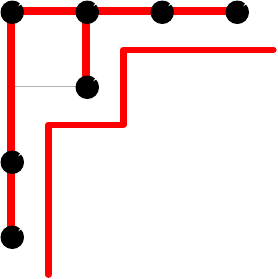}}$}
	\caption{The horizontal packing of a $\nu$-tree at a diagonal leaf~$b$.}
	\label{fig:horizontalPackingNuTree}
\end{figure}

\begin{figure}[ht]
	\capstart
	\centerline{$\b{\lmir_b}\left( \raisebox{-.5\height}{\includegraphics[scale=.3]{nuTreePacking1}} \right) = \raisebox{-.5\height}{\includegraphics[scale=.3]{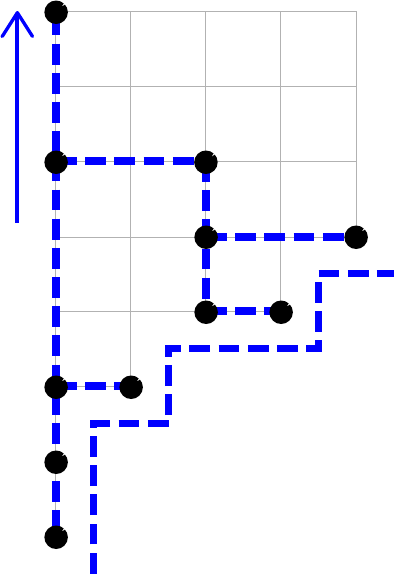}} = \raisebox{-.5\height}{\includegraphics[scale=.3]{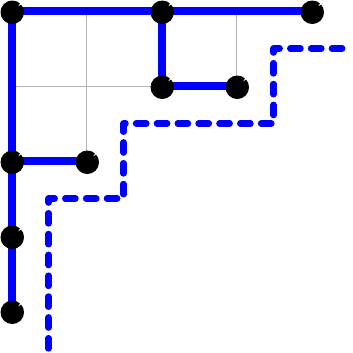}}$}
	\caption{The vertical packing of a $\nu$-tree at a diagonal leaf~$b$.}
	\label{fig:verticalPackingNuTree}
\end{figure}

\subsubsection{Coproduct on $\nu$-trees}
\label{subsubsec:coproductNuTrees}

We define the coproduct of a $\nu$-tree~$T$ as
\[
\coproduct(T)= \sum \b{\lmir_b(T)} \otimes \r{\lrot_b(T)},
\]
where the sum runs over all diagonal leaves~$b$ of~$T$. See \fref{fig:coproductNuTrees}. This operation is extended linearly to~$\HT$.

\begin{figure}[ht]
	\capstart
	\centerline{$\coproduct \left( \raisebox{-.5\height}{\includegraphics[scale=.3]{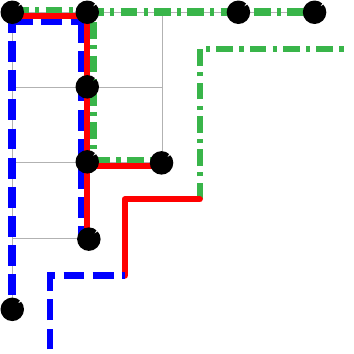}} \right) = \raisebox{-.5\height}{\includegraphics[scale=.3]{nuTreeCoproduct1}} \otimes \epsilon + \raisebox{-.5\height}{\includegraphics[scale=.3]{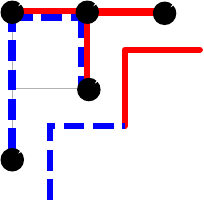}} \otimes \raisebox{-.5\height}{\includegraphics[scale=.3]{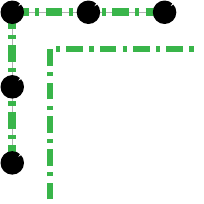}} + \raisebox{-.5\height}{\includegraphics[scale=.3]{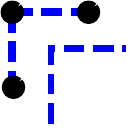}} \otimes \raisebox{-.5\height}{\includegraphics[scale=.3]{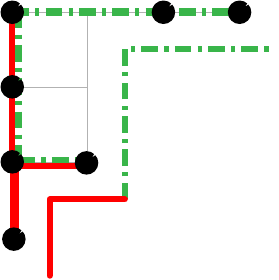}} + \epsilon \otimes \raisebox{-.5\height}{\includegraphics[scale=.3]{nuTreeCoproduct1}}$}
	\caption{The coproduct of a $\nu$-tree.}
	\label{fig:coproductNuTrees}
\end{figure}

\subsubsection{Product on $\nu$-trees}
\label{subsubsec:productNuTrees}

Let~${\bf b} = (b_1, \dots, b_{\ell-1})$ be a tuple of $\ell-1$ diagonal leaves of a $\nu$-tree~$T$ which are located in order along the main diagonal with possible repetitions. They partition~$\nu$ into~$\ell$ Dyck paths $\nu_1, \dots, \nu_\ell$. The tree~$T$ is subdivided into~$\ell$ trees~$\tilde T_1, \dots, \tilde T_\ell$ by cutting along the paths from the leaves~$b_i$ to the root. Define~$T_i$ to be the $\nu_i$-tree obtained by contracting segments of~$\tilde T_i$ that are either horizontal on the left of~$b_{i-1}$ or vertical above~$b_i$. By convention, $b_0$~and~$b_\ell$ denote two extra leaves at coordinates~$(0,0)$ and~$(n,n)$ respectively. 

Given a $\mu$-tree~$S$ and a $\nu$-tree~$T$ we will define the product~$S\cdot T$ as follows.
If~$S$ has~$\ell$ diagonal leaves, we choose a tuple~${\bf b} = (b_1, \dots, b_{\ell-1})$ of $\ell-1$ leaves of~$T$ and we ``cut''~$T$ along~$\bf b$ to produce~$\ell$ trees~$T_1, \dots, T_\ell$ as described above. See \fref{fig:cuttingNuTrees}. We then ``glue" these trees~$T_1, \dots, T_\ell$ on the~$\ell$ diagonal leaves of~$S$. The resulting tree~$S \star_{\bf b} T$ is a $\lambda$-tree for some Dyck path~$\lambda$ obtained as a shuffle of~$\mu$ and~$\nu$ with cuts at diagonal leaves, see \fref{fig:cutGlue}. 

\begin{figure}[ht]
	\capstart
	\centerline{$\raisebox{-.5\height}{\includegraphics[scale=0.3]{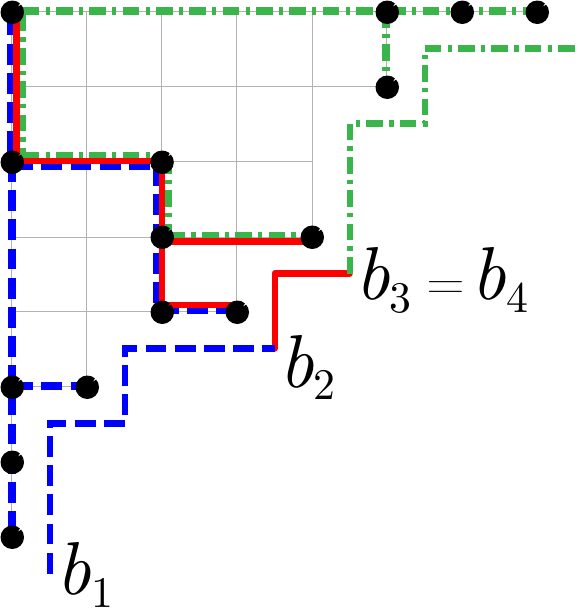}} \qquad \longrightarrow \qquad \raisebox{-.5\height}{\includegraphics[scale=0.3]{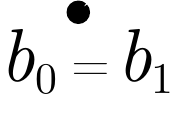}} \qquad \raisebox{-.5\height}{\includegraphics[scale=0.3]{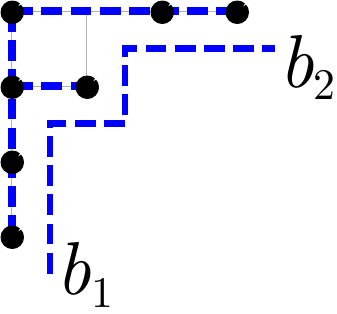}} \qquad \raisebox{-.5\height}{\includegraphics[scale=0.3]{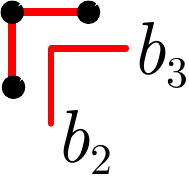}} \qquad \raisebox{-.5\height}{\includegraphics[scale=0.3]{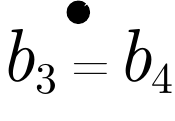}} \qquad \raisebox{-.5\height}{\includegraphics[scale=0.3]{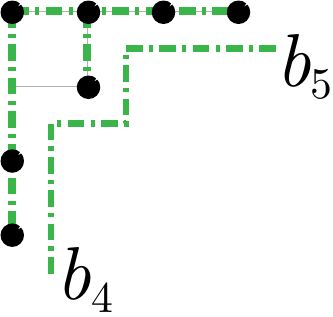}}$ }
	\caption{Cutting a $\nu$-tree into $5$ smaller trees along diagonal leaves $(b_1, b_2, b_3, b_4)$.}
	\label{fig:cuttingNuTrees}
\end{figure}

\begin{figure}[ht]
	\capstart
	\centerline{$\raisebox{-.5\height}{\includegraphics[scale=0.3]{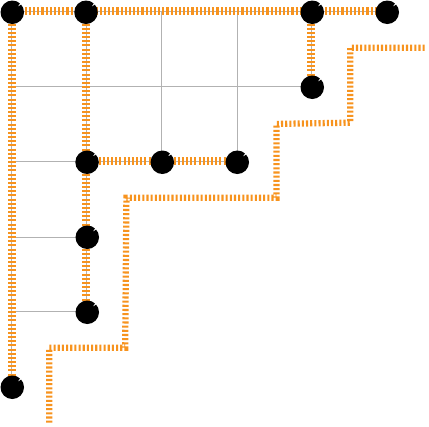}} \star_{(b_1,b_2,b_3,b_4)} \quad \raisebox{-.5\height}{\includegraphics[scale=0.3]{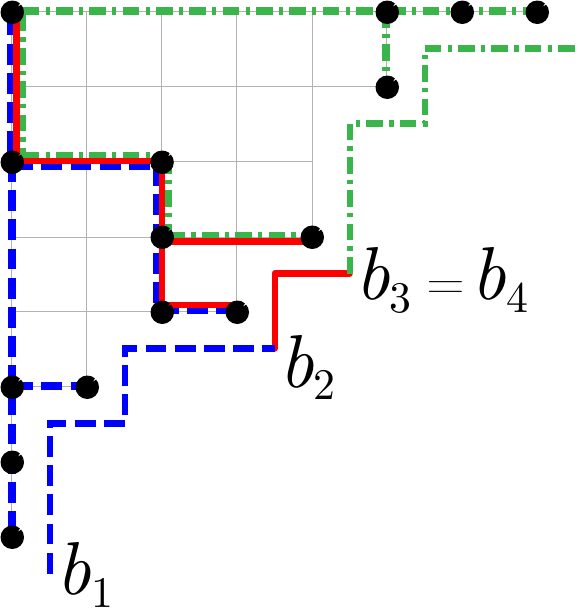}} \quad = \quad \raisebox{-.5\height}{\includegraphics[scale=0.3]{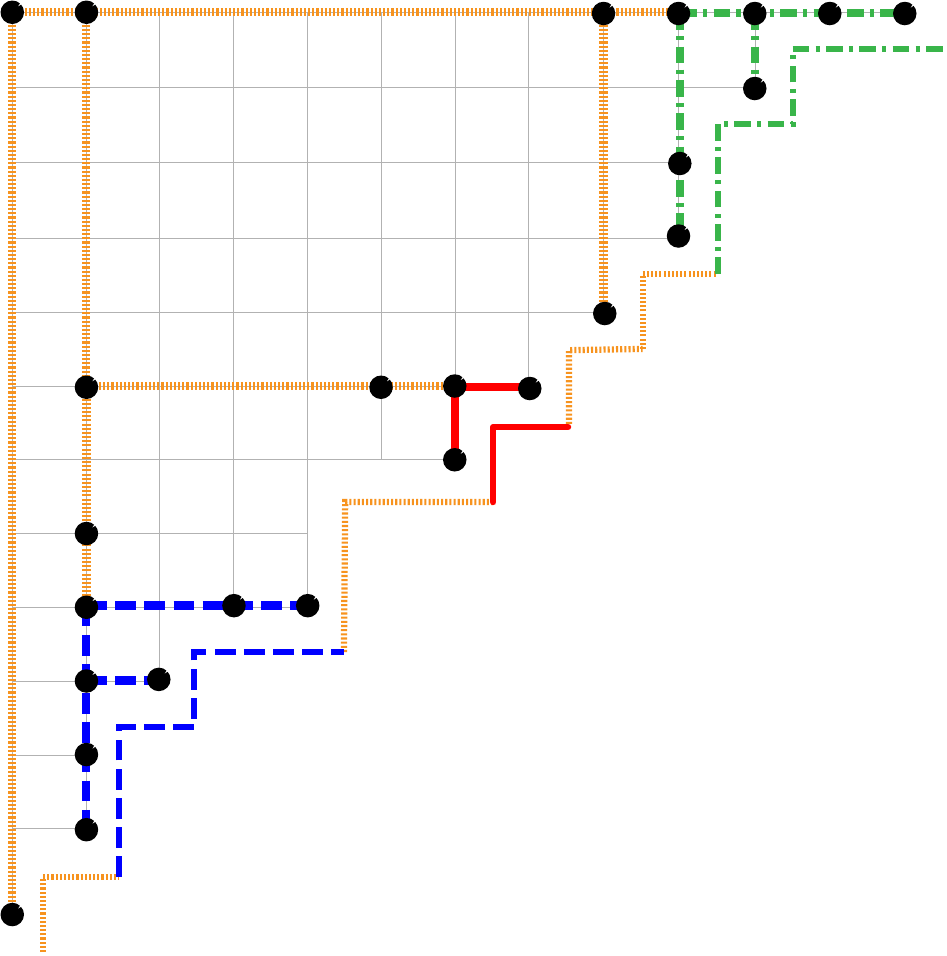}}$}
	\caption{The cut-and-glue operation~$S\star_{\bf b} T$ of a $\mu$-tree~$S$ with a $\nu$-tree~$T$ according to the diagonal leaves ${\bf b}=(b_1, b_2, b_3, b_4)$.}
	\label{fig:cutGlue}
\end{figure}

If~$S$ has~$\ell$ diagonal leaves, we define the product of~$S$ and~$T$ by
\[
S\cdot T = \sum_{\bf b} S \star_{\bf b} T,
\]
where the sum ranges over all ordered tuples~${\bf b}=(b_1,\dots ,b_{\ell-1})$ of~$\ell-1$ diagonal leaves in~$T$ with possible repetitions. An example of this product is illustrated in \fref{fig:productNuTrees}.

\begin{figure}[ht]
	\capstart
	\centerline{$
		\raisebox{-.5\height}{\includegraphics[scale=0.3]{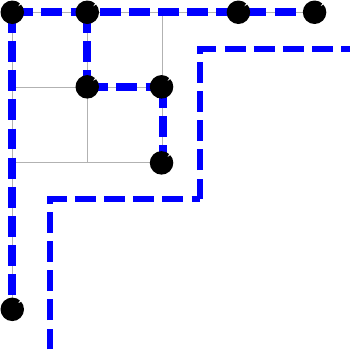}} \product \raisebox{-.5\height}{\includegraphics[scale=0.3]{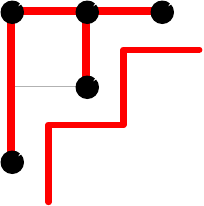}} =
		\begin{array}[t]{c}
			\; \raisebox{-.5\height}{\includegraphics[scale=0.3]{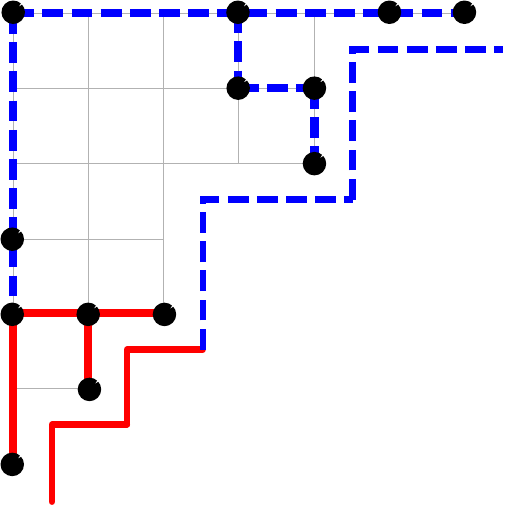}} \,+\, \raisebox{-.5\height}{\includegraphics[scale=0.3]{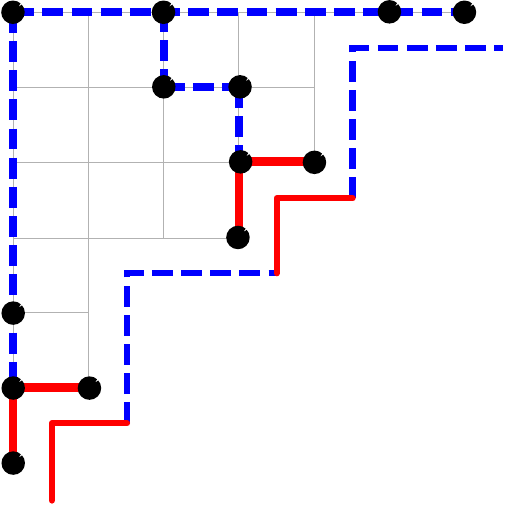}} \,+\, \raisebox{-.5\height}{\includegraphics[scale=0.3]{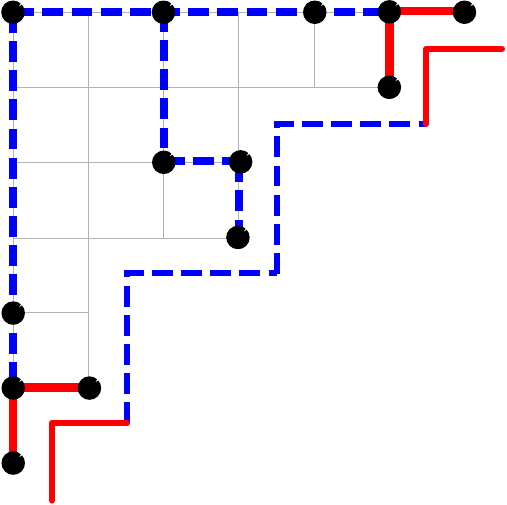}} \\ \\
			\!\!+\, \raisebox{-.5\height}{\includegraphics[scale=0.3]{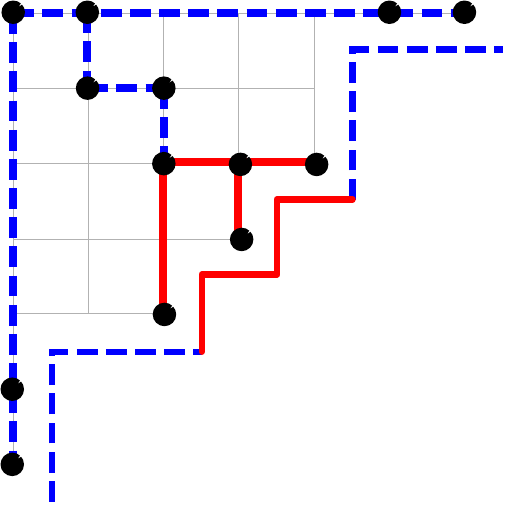}} \,+\, \raisebox{-.5\height}{\includegraphics[scale=0.3]{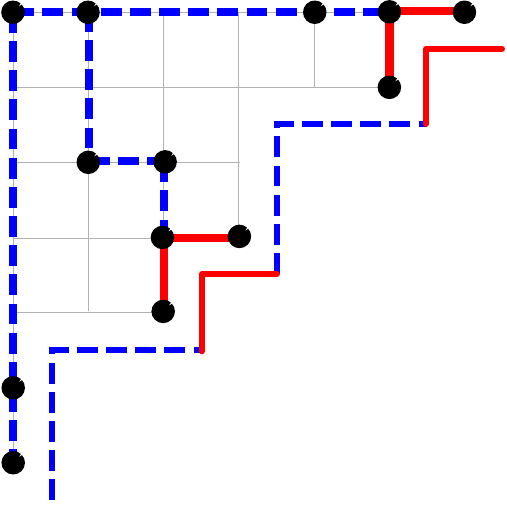}} \,+\, \raisebox{-.5\height}{\includegraphics[scale=0.3]{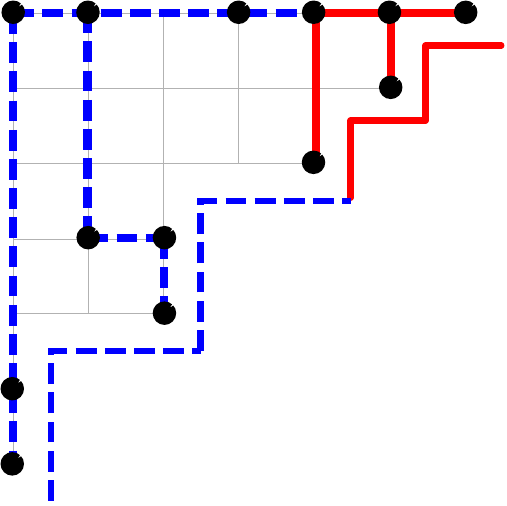}}
		\end{array}
	$}
	\caption{The product of a $\mu$-tree and a $\nu$-tree.}
	\label{fig:productNuTrees}
\end{figure}

The following statement will follow from Theorem~\ref{thm:dominantPipeDreamsVSNuTrees}.

\begin{proposition}
The product~$\product$ and coproduct~$\coproduct$ endow the family of~$\nuTrees$ of $\nu$-trees for all Dyck paths~$\nu$ with a graded connected Hopf algebra structure.
\end{proposition}


\subsection{Dominant pipe dreams versus $\nu$-trees}
\label{subsec:dominantPipeDreamsVSNuTrees}

We now connect the dominant pipe dreams with the $\nu$-trees and show that the Hopf algebras considered in the previous two sections are isomorphic. 
For this, we consider the map~$\Psi$, illustrated on \fref{fig:bijectionNuTreeDominantPipeDream}, that sends a pipe dream~$P \in \pipeDreams(\omega)$ 
with dominant permutation $\omega$ to a $\nu$-tree $T$ where $\nu=\pi_\omega$ is the Dyck path associated to $\omega$.
This $\nu$-tree is defined as the set of lattice points~$\Psi(P)$ given by the elbows of~$P$ located in the topmost row or leftmost column, or inside the Rothe diagram of~$\omega_P$.
\begin{figure}[t]
	\capstart
	\centerline{\includegraphics[scale=0.4]{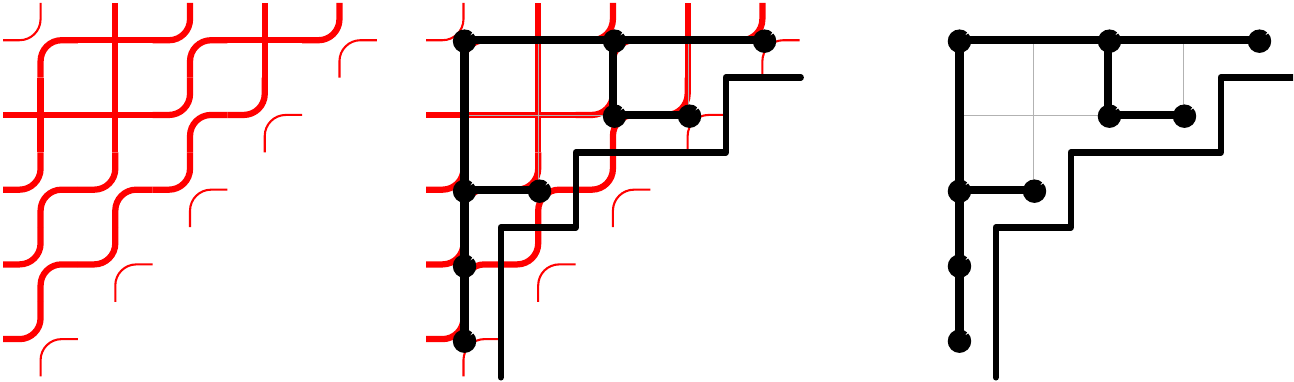}}
	\caption{The bijection between dominant pipe dreams and $\nu$-trees: replace each elbow of $P$ that are in the Rothe diagram of~$\omega$ by a node.}
	\label{fig:bijectionNuTreeDominantPipeDream}
\end{figure}

\begin{proposition}[\cite{SerranoStump, CeballosPadrolSarmiento-nuTamariSubwordComplexes}]
\label{prop:bijectionDominantPipeDreamsNuTrees}
For any dominant permutation~$\omega$ with corresponding Dyck path~$\nu=\pi_\omega$, the map~$\Psi$ is a bijection between the dominant pipe dreams in~$\pipeDreams(\omega)$ and the $\nu$-trees of~$\nuTrees(\nu)$, which sends flips in dominant pipe dreams to rotations in~$\nu$-trees.
\end{proposition}

Using the bijection of Proposition~\ref{prop:bijectionDominantPipeDreamsNuTrees}, we derive the following statement.

\begin{theorem}
\label{thm:dominantPipeDreamsVSNuTrees}
The map~$\Psi$ is a Hopf algebra isomorphism between the Hopf algebra~$(\HPdom, \product, \coproduct)$ of dominant pipe dreams and the Hopf algebra~$(\HT, \product, \coproduct)$ of~$\nu$-trees.
\end{theorem}

\begin{proof}
Observe first that contracting in a dominant pipe dream~$P$ the horizontal segments of the red pipes crossed by blue pipes as in Section~\ref{subsec:packings} 
corresponds to contracting in~$\Psi(P)$ the horizontal edges to the right of some diagonal leaf~$b$ as in Section~\ref{subsubsec:packingsNuTrees}. 
Therefore, $\Psi$ commutes with the horizontal packing (and similarly with the vertical packing). Compare Figures~\ref{fig:horizontalPackingPipeDream} and~\ref{fig:horizontalPackingNuTree} and Figures~\ref{fig:verticalPackingPipeDream} and~\ref{fig:verticalPackingNuTree}.

The result then directly follows since the definitions of the product and coproduct on pipe dreams (Sections~\ref{subsec:coproductPipeDreams} and~\ref{subsec:productPipeDreams}) are parallel to the definitions of the product and coproduct on $\nu$-trees (Sections~\ref{subsubsec:coproductNuTrees} and~\ref{subsubsec:productNuTrees}).
For example, compare the coproducts in Figures~\ref{fig:coproductPipeDreams} and~\fref{fig:coproductNuTrees} and the products in Figures~\ref{fig:productPipeDreams} and~\ref{fig:productNuTrees}.
\end{proof}


\subsection{Connection to $\nu$-Tamari lattices} 
\label{subsec:nuTamari}

To conclude, we consider the $\nu$-Tamari lattice introduced by L.-F.~Pr\'eville-Ratelle and X.~Viennot in~\cite{PrevilleRatelleViennot}.

\begin{definition}
Let~$\nu$ be a Dyck path.
\begin{enumerate}
\item A \defn{$\nu$-path} is a Dyck path lying weakly above $\nu$ with the same starting and ending points.
\item The \defn{horizontal distance} from a horizontal step~$x$ to~$\nu$ is the length of the segment between the rightmost point of~$x$ and the rightmost point of~$\nu$ on the horizontal line supporting~$x$.
\item Two $\nu$-path~$\mu, \mu'$ are related by a \defn{$\nu$-Tamari flip} if~$\mu = \mu_1 EN \mu_2 \mu_3$ while~$\mu' = \mu_1 N \mu_2 E \mu_3$ are so that the horizontal distance from the distinguished east step~$E$ to~$\nu$ coincides in~$\mu$ and~$\mu'$, and $\mu_2$ is the shortest path satisfying this condition. See \fref{fig:bijectionNuPathDominantPipeDream}\,(right).
\item the \defn{$\nu$-Tamari lattice} is the transitive closure of the (oriented) graph of $\nu$-Tamari flips.
\end{enumerate}
\end{definition}

The following statement relates the dominant pipe dreams with the $\nu$-paths. We have already seen that the collection of pipe dreams~$\pipeDreams(\omega)$ of a permutation $\omega$ can be endowed with various natural poset structures, for instance using increasing flips or chute moves. In~\cite{Rubey}, M.~Rubey considered the poset of pipe dreams induced by \defn{general chute moves}, defined as flips where the interior of the rectangle connecting the exchanged elbow and cross only contains crosses. In other words, the difference with the chute moves illustrated in \fref{fig:chuteMove} is that the rectangle may be wider than just two rows. He moreover conjectured that this poset has the structure of a lattice~\cite[Conj.~2.8]{Rubey}. For the special case of dominant permutations, it was shown in~\cite{CeballosPadrolSarmiento-nuTamariSubwordComplexes} that this poset is isomorphic to the $\nu$-Tamari lattice of~\cite{PrevilleRatelleViennot}.

Let~$\omega \in \fS_n$ be a dominant permutation and~$\nu$ denote its corresponding Dyck path. Consider the map~$\phi_\omega$ that sends a dominant pipe dream~$P \in \pipeDreams(\omega)$ to the unique lattice path~$\phi_\omega(P)$ that shares the endpoints of~$\nu$ and contains as many points at level~$i$ as there are elbows in~$P$ located at level~$i$ and either in the topmost row, or in the leftmost column or inside the Rothe diagram of~$\omega$.

\begin{proposition}[{\cite{SerranoStump, CeballosPadrolSarmiento-nuTamariSubwordComplexes}}]
\label{prop:bijectionNuTreesDyckIntervals}
For any dominant permutation~$\omega$ with corresponding Dyck path~$\nu=\pi_\omega$, the map~$\phi_\omega$ is a bijection between the dominant pipe dreams of~$\pipeDreams(\omega)$ and the $\nu$-paths, which sends general chute moves in dominant pipe dreams to $\nu$-Tamari flips in $\nu$-paths. In particular, the poset of dominant pipe dreams in $\pipeDreams(\omega)$ induced by general chute moves is isomorphic to the $\nu$-Tamari lattice.  
\end{proposition}

\begin{figure}[t]
	\capstart
	\centerline{$
		\input{dominantPipeDream1} \quad\longleftrightarrow\quad \input{dominantPipeDream2}
		\qquad\qquad
		\raisebox{-.5\height}{\includegraphics[scale=0.3]{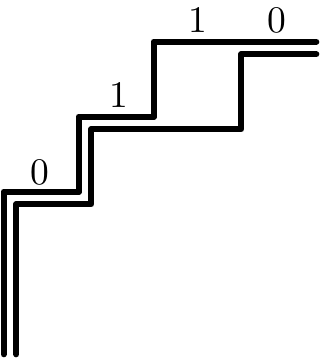}} \quad\longleftrightarrow\quad \raisebox{-.5\height}{\includegraphics[scale=0.3]{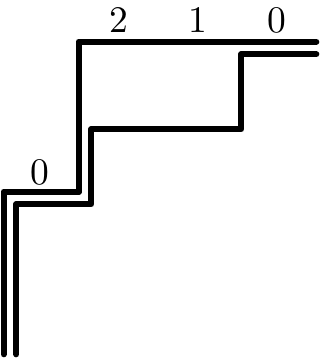}}
	$} 
	\caption{Two dominant pipe dreams connected by a general chute move (left) are sent to two $\nu$-paths connected by a $\nu$-Tamari flip (right). We have labeled the horizontal steps of the $\nu$-paths with their horizontal distance to~$\nu$.}
	\label{fig:bijectionNuPathDominantPipeDream}
\end{figure}

\begin{corollary}[{\cite{SerranoStump, CeballosPadrolSarmiento-nuTamariSubwordComplexes}}]
\label{cor:dominantPairsDyckPaths}
The map~$\Phi : P \mapsto \big( \pi_{\omega_P}, \, \phi_{\omega_P}(P) \big)$ is a bijection between dominant pipe dreams of~$\pipeDreamsDom_n$ and pairs of nested Dyck paths of size~$n$.
\end{corollary}

As a consequence, the Hopf algebra of dominant dreams, or equivalently, the Hopf algebra of~$\nu$-trees, can be regarded as a Hopf algebra on $\nu$-Tamari lattices.

\begin{corollary}
The dimension of the homogeneous component~$\HPdom_n$ of the dominant pipe dream algebra is the Hankel determinant of Catalan numbers
\[
\dim \big( \HPdom_n \big) = \det 
\begin{vmatrix} 
C_n & C_{n+1} \\
C_{n+1} & C_{n+2}
\end{vmatrix}.
\]
\end{corollary}

\begin{proof}
The dimension of~$\HPdom_n$ is equal to the number of dominant pipe dreams~$P\in \pipeDreamsDom_n$, which is equal to the number of pairs of nested Dyck paths of size~$n$. Such nested tuples of Dyck paths are know to be counted by the desired determinant of Catalan numbers~\cite{GesselViennot}. 
\end{proof}


\section{Hopf chains and multivariate diagonal harmonics}
\label{sec:multivariateDiagonalHarmonics}

In this section we develop a connection between the theory of multivariate diagonal harmonics and certain chains of the Tamari lattice inspired by our work on the Hopf structure on pipe dreams.


\subsection{Multivariate diagonal harmonics}
\label{subsec:definitionsMultivariateDiagonalHarmonics}

We start with a brief introduction to the multivariate diagonal harmonic spaces, see~\cite{Bergeron-multivariateDiagonalCoinvariantSpaces} for more details. 
Let~$X=[x_{ij}]$ be a set of $nr$ variables for~$i \in [r]$ and~$j \in [n]$. 
We refer to $r$ as the number of sets of variables and to $n$ as the number of variables in each of the sets.
The symmetric group~$\fS_n$ acts on the polynomial ring~$\C[X]$ by permuting the $n$ columns of the matrix $X$. 
We consider the subring~$\Sym{n}{r}$ of polynomials~$f(X)$ that are invariant under the action of~$\fS_n$.
The \defn{multivariate diagonal harmonic space}~$\DiagHarm{n}{r}$ is the space
\[
\DiagHarm{n}{r} = \bigset{p\in \C[X]}{f(\partial X)p = 0 \text{ for all } f\in \Sym{n}{r} \text{ such that } f({\bf 0}) = 0},
\]
where $f(\partial X)$ is the partial differential operator obtained by replacing the variables $x_{ij}$ by $\frac{\partial}{\partial x_{ij}}$ in~$f(X)$.
This space is closed under the action of the symmetric group $\fS_n$, and therefore defines a representation of~$\fS_n$. 

The modules $\DiagHarm{n}{r}$ have intrigued several mathematicians for over 25 years and constitute an active area of current research. 
For small fixed values of $r$, the dimension of $\DiagHarm{n}{r}$ and 
the multiplicity of the sign representation~$\Alt(\DiagHarm{n}{r})$ satisfy beautiful formulas in terms of $n$: 

\[
\begin{array}{|c|c|c|c}
\cline{1-3}
r & \dim \big( \DiagHarm{n}{r} \big) & \dim \big( \Alt(\DiagHarm{n}{r}) \big)\\
\cline{1-3}
r = 1 & n! & 1 \\[4pt]
r = 2 & (n+1)^{n-1} & \frac{1}{n+1}{2n\choose n} \\[5pt]
r = 3 & 2^n (n+1)^{n-2} & \frac{2}{n(n+1)} {4n+1\choose n-1} & \text{(conjectured)} \\[3pt]
\cline{1-3}
\end{array}
\]

\bigskip
For~$r=2$, the Catalan number~$\frac{1}{n+1}{2n\choose n}$ is the number of Dyck paths of size $n$ (\ie elements in the $n$-Tamari lattice) and~$(n+1)^{n-1}$ is the number of parking functions. 
The last can be interpreted as the number labeled Dyck paths, which are Dyck paths whose north steps are labeled with integers~$1, 2, \dots, n$ such that the labels are increasing along each column. 
For~$r=3$, the number~$\frac{2}{n(n+1)} {4n+1\choose n-1}$ was conjectured by M.~Haiman in~\cite{Haiman-conjectures}, and F.~Bergeron noticed that this expression counts the number of intervals in the $n$-Tamari lattice~\cite{Chapoton-intervalsTamari}. 
Analogously, the number~$2^n(n+1)^{n-2}$ is the number of labeled intervals in the $n$-Tamari lattice.
Here, a labeled interval refers to a pair of Dyck paths forming an interval in the Tamari lattice where the north steps of the top path are labeled with integers~$1, 2, \dots, n$ such that the labels are increasing along each column. 
In addition, there is no known conjectured formula for the dimensions of $\DiagHarm{n}{r}$ and $\Alt(\DiagHarm{n}{r})$ for~$r > 3$.

The spaces $\DiagHarm{n}{r}$ and $\Alt(\DiagHarm{n}{r})$ can be further decomposed into homogeneous components that are invariant under the action of the symmetric group.
For a monomial~$X^A=\prod x_{ij}^{a_{ij}}$, its multi-degree is defined by
\[
\deg_{r}(X^A) = \Big( \sum_{j=1}^n a_{1j}, \sum_{j=1}^n a_{2j}, \dots, \sum_{j=1}^n a_{rj} \Big).
\]
The subspaces of $\DiagHarm{n}{r}$ and $\Alt(\DiagHarm{n}{r})$ of fixed degree are invariant under the action of~$\fS_n$.
For $r=2$, the bigraded Frobenius characteristic of $\DiagHarm{n}{2}$ in terms of Macdonald polynomials was conjectured by A.~Garsia and M.~Haiman~\cite{GarsiaHaiman-remarkableCatalanSequence}, and proved by M.~Haiman in~\cite{Haiman-vanishingTheorems}.
A combinatorial description, involving a pair of statistics $\area$ and $\dinv$ on parking functions, is described by the former shuffle conjecture of J.~Haglund et al.~\cite{HaglundHaimanLoehrRemmelUlyanov}, which was recently proved by E.~Carlsson and A.~Mellit in~\cite{CarlssonMellit} (see Formula~\eqref{eq:shuff_formula}). 
The bigraded Hilbert series of~$\Alt(\DiagHarm{n}{2})$ is the famous $q,t$-Catalan polynomial~\cite{Haglund-qt-catalan}, which is the $q,t$ counting of Dyck paths with respect to the $\area$ and $\dinv$ statistics (see Formula~\ref{eq:qt_CatalanFormula}).
For $r=3$, there is no known triple of statistics on labeled and unlabeled intervals in the Tamari lattice that would match the tri-degree of $\DiagHarm{n}{3}$ and $\Alt(\DiagHarm{n}{3})$.
F.~Bergeron and L.-F.~Pr\'eville-Ratelle conjectured in~\cite[Conj.~1]{BergeronPrevilleRatelle} that length of a longest chain in the interval and $\dinv$ are two of the statistics.
Remark that the area of a Dyck path~$\nu$ in the case~$r=2$ could also be interpreted as the length of a longest chain in the interval from the diagonal path~$(NE)^n$ to the Dyck path~$\nu$.


\subsection{Hopf chains and collar statistic}
\label{subsec:HopfChains}

Computer experimentations provide strong evidences that the dimensions of the spaces of diagonal harmonics for more sets of variables should be related to certain chains in the Tamari lattice. It is therefore very natural to study chains of paths arising from intervals in the graded dimensions of the Hopf algebra of dominant pipe dreams, where the order is induced by general chute moves as before. This motivates the following definition.

\begin{definition}
\label{def:HopfChains}
A \defn{Hopf chain} of length~$r$ and size~$n$ is a nested tuple $(\pi_1,\pi_2,\dots ,\pi_r)$ of Dyck paths of size~$n$ such that:
\begin{enumerate}
\item $\pi_1$ is the bottom diagonal path~$(NE)^n$,
\item for every~$1 \le i < j < k \le r$, the pair $(\pi_j,\pi_k)$ is an interval in the $\pi_i$-Tamari lattice.
\end{enumerate}
We denote by~$\HC{n}{r}$ the set of Hopf chains of length~$r$ and size~$n$.
\end{definition}

\begin{remark}
In Definition~\ref{def:HopfChains}, observe that
\begin{itemize}
\item Condition (2)  coincides with the following equivalent condition by Proposition~\ref{prop:bijectionNuTreesDyckIntervals}: every subtriple~$(\pi_i,\pi_j,\pi_k)$ comes from an interval of dominant pipe dreams, meaning that the two pipe dreams corresponding to~$\pi_j$ and~$\pi_k$ in~$\pipeDreams(\pi_i)$ form an interval in the poset induced by general chute moves.
\item Condition~(2) applied to the diagonal path~$\pi_i = \pi_1$, implies that every Hopf chain is a chain in the classical Tamari lattice. We denote by~$\TC{n}{r}$ the set of Tamari chains of length~$r$ and size~$n$, starting with the bottom diagonal path~$(NE)^n$.
\end{itemize}
\end{remark}

\begin{example}
For~$n \le 3$, all Tamari chains in~$\TC{n}{r}$ are Hopf chains in~$\HC{n}{r}$ and the number matches the dimension~$\dim(\Alt(\DiagHarm{n}{r}))$ computed in~\cite{Bergeron-multivariateDiagonalCoinvariantSpaces}.
In contrast, already for~$n=4$, there are Tamari chains that are not Hopf chains. All these Tamari but non-Hopf chains contain the one in \fref{fig:killer} as a subchain.
\begin{figure}[t]
	\capstart
	\centerline{\includegraphics[scale=.5]{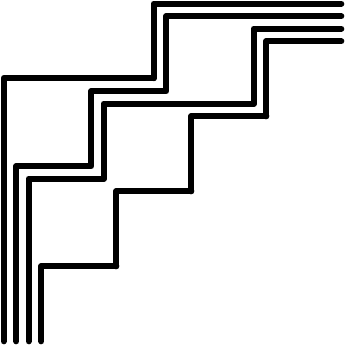}} 
	\caption{The Tamari chain that is not a Hopf chain when~$n=4$. Indeed, the paths~$\pi_4$ is not greater than the path~$\pi_3$ in the $\pi_2$-Tamari lattice.}
	\label{fig:killer}
\end{figure}
The numbers~$|\TC{4}{r}|$ of Tamari chains and~$|\HC{4}{r}|$ of Hopf chains of size~$n=4$ and different~$r$ are given by
\[
\begin{array}{c|cccccccc}
r & 1 & 2 & 3 & 4 & 5 & 6 & 7 & \dots \\
\hline
|\TC{4}{r}| & 1 & 14 & 68 & 218 & 556 & 1224 & 2429 & \dots \\
\hline
|\HC{4}{r}| & 1 & 14 & 68 & 217 & 549 & 1196 & 2345 & \dots
\end{array}
\]
Interestingly, the numbers~$|\HC{4}{r}|$ coincide with the dimensions~$\dim(\Alt(\DiagHarm{4}{r}))$ computed in~\cite{Bergeron-multivariateDiagonalCoinvariantSpaces}. Let us insist here: it is the number of Hopf chains, not the number of Tamari chains, that coincides with the dimension of the alternating component of the space of diagonal harmonics~$\DiagHarm{n}{r}$ for fixed~$n=4$ and~$r$ sets of variables.
\end{example}

We now define a new statistic associated to Hopf chains that generalizes the area for~$r = 2$ and the length of a longest chain in an interval for~$r = 3$.

\begin{definition}
Given a Hopf chain $\bpi=(\pi_1,\pi_2,\dots ,\pi_r)$, the \defn{collar} of~$\bpi$ is one plus the maximal number of distinct Dyck paths that can be inserted in~$\bpi$ strictly between~$\pi_{r-1}$ and~$\pi_r$ such that the result is a Hopf chain. That is
\[
\col(\bpi) = \max \set{\ell}{{(\pi_{r-1},\nu_1,\cdots,\nu_{\ell-1},\pi_r) \text{ {\bf strict} Tamari chain,} \hfill \atop (\pi_1,\pi_2,\dots ,\pi_{r-1},\nu_1,\cdots,\nu_{\ell-1},\pi_r) \text{ Hopf chain}}}
\]
Note that~$\col(\bpi) = 0$ if~$\pi_{r-1} = \pi_r$ by definition.
\end{definition} 


\subsection{$\dinv$-statistic and LLT-polynomials}

For a step~$x$ in a Dyck path~$\pi$, we denote by~$N(x)$ the number of north steps before~$x$ and by~$E(x)$ the number of east steps before~$x$.
Note that with these notations, the diagonal level of~$x$ is~$N(x)-E(x)$.
The \defn{diagonal order} on the steps of a Dyck path~$\pi$ is the lexicographic order on~$\big( N(x) - E(x), N(x) + E(x) \big)$.
A \defn{parking function} on a Dyck path~$\pi$ is a labeling~$\psi$ of its north steps by~$[n]$ such that the labels are increasing along columns.
The \defn{diagonal reading} of~$\psi$ is the permutation~$\delta(\psi)$ obtained by reading the labels of the north steps of~$\pi$ in reverse diagonal order.
We denote by~$\mu(\psi)$ the recoil composition of~$\delta(\psi)$ (\ie the descent composition of the inverse of~$\delta(\psi)$).
The \defn{diagonal inversions} of~$\psi$ are the pairs of values~$u < v$ such that~$u$ appears after~$v$ in~$\delta(\psi)$ and 
\begin{itemize}
\item either~$u$ and~$v$ appear along the same diagonal of~$\pi$,
\item or the diagonal of~$u$ is just below the diagonal of~$v$ and~$u$ is to the right of~$v$.
\end{itemize}
The two cases are depicted as follows:
\begin{equation}\label{eq:diag_inv}
   \begin{tikzpicture}[scale=.5,baseline=.6cm]
      \node at (2,2)  {$v$};
      \node at (0,0)  {$u$};
      \draw[thick] (1.7,1.7) -- (1.7,2.3);
      \draw[thick] (-.3,-.3) -- (-.3,.3);
      \draw[dashed,<-] (.3,.3) -- (1.7,1.7);
      \draw[dotted,-] (-.6,-.6) -- (-.3,-.3);
      \draw[dotted,-] (2.3,2.3) -- (2.6,2.6);
    \end{tikzpicture}
\qquad\qquad,\qquad\qquad
   \begin{tikzpicture}[scale=.5,baseline=.6cm]
      \node at (.2,.2)  {$v$};
      \node at (2.4,1.9)  {$u$};
      \draw[dotted,-] (.5,.5) -- (2.6,2.6);
      \draw[thick] (-.1,-.1) -- (-.1,.5);
      \draw[thick] (2.2,1.6) -- (2.2,2.2);
      \draw[dashed,<-] (-.6,-.6) -- (-.1,-.1);
      \draw[dashed,<-] (2.7,2.1) -- (3.2,2.6);
      \draw[dotted,-] (0,-.6) -- (2.2,1.6);
    \end{tikzpicture}
\end{equation}
The number of diagonal inversions of~$\psi$ is denoted by~$\dinv(\psi)$.

\begin{example}
The parking function $\psi$ given in Figure~\ref{fig:dinv} has permutation $\delta(\psi)=5643712$.
Its recoil composition is $\mu(\psi)=2113$. 
The set of diagonal inversions of the first kind is $\{(4,3)\}$, and of the second kind is $\{(6,4),(7,1)\}$.
Hence, $\dinv(\psi)=3$.

\begin{figure}[t]
	\capstart
	\centerline{
	  \begin{tikzpicture}[scale=.6,baseline=.6cm]
    	  \node at (.2,.2)  {$\scriptstyle 2$};
    	  \node at (.2,.7)  {$\scriptstyle 7$};
    	  \node at (1.2,1.2)  {$\scriptstyle 1$};
    	  \node at (1.2,1.7)  {$\scriptstyle 3$};
    	  \node at (1.2,2.2)  {$\scriptstyle 6$};
    	  \node at (2.2,2.7)  {$\scriptstyle 4$};
    	  \node at (2.2,3.2)  {$\scriptstyle 5$};
    	  \draw[thick] (0,0) -- (0,1);
    	  \draw[thick] (0,1) -- (1,1);
    	  \draw[thick] (1,1) -- (1,2.5);
    	  \draw[thick] (1,2.5) -- (2,2.5);
    	  \draw[thick] (2,2.5) -- (2,3.5);
    	  \draw[thick] (2,3.5) -- (3.5,3.5);
    	  \draw[dashed,<-] (.3,.3) -- (1,1);
     	  \draw[dotted,-] (1.4,1.4) -- (3.5,3.5);
     	  \draw[dashed,<-] (.3,.8) -- (1,1.5);
     	  \draw[dotted,-] (2.4,2.9) -- (3,3.5);
     	  \draw[dashed,<-] (1.3,1.8) -- (2.,2.5);
     	  \draw[dashed,<-] (1.3,2.3) -- (2.,3.);
	  \draw[dashed,<-] (2.3,3.3) -- (2.9,3.9);
   	 \end{tikzpicture}
 } 
	\caption{The north steps of the Dyck path above are labeled by the numbers $\{1,2,3,4,5,6,7\}$ and the labels are increasing up along each column. This labeling~$\psi$ is an example of parking function. The permutation $\delta(\psi)=5643712$.}
	\label{fig:dinv}
\end{figure}
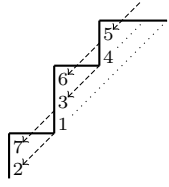
\end{example}

The \defn{LLT-polynomial} of a Dyck path~$\pi$ is the quasi-symmetric function
\[
\LLT{\pi}{t} \eqdef \sum_{\psi} t^{\dinv(\psi)} F_{\mu(\psi)},
\]
where the sum runs over the parking functions~$\psi$ of~$\pi$, and~$F_{\mu}$ denotes the fundamental quasi-symmetric function associated to a composition~$\mu$.
One can also express $\LLT{\pi}{t}$ in terms of monomials as in~\cite[Definition 3.2]{HHL05}.
With the monomial expression it is clear that
$\LLT{\pi}{1}=e_{\type(\pi)}$ as we ignore the $\dinv$ statistic.

These functions were originally due to A.~Lascoux, B.~Leclerc and J.-Y.~Thibon~\cite{LascouxLeclercThibon}, and were revisited in~\cite{HHL05,HaglundHaimanLoehrRemmelUlyanov}.
The former shuffle conjecture of~\cite{HaglundHaimanLoehrRemmelUlyanov}, which was proven by E.~Carlsson and A.~Mellit in~\cite{CarlssonMellit}, states that the Frobenius characteristic of~$\DiagHarm{n}{2}$, in the two variate case, is given by the following expression:  
\begin{equation}
\label{eq:shuff_formula}
\sum_{\pi} q^{\area(\pi)} {\mathbb L}_{\pi}(t),
\end{equation}
where the sum is over all Dyck paths $\pi$ of size $n$.

The $\dinv$ statistic of a Dyck path $\pi$ is a specialization of the $\dinv$ statistic on parking functions. Namely, consider $\psi_0$ the unique parking function of $\pi$ such that $\delta(\psi_0)=n\cdots 321$. We define $\dinv(\pi) \eqdef \dinv(\psi_0)$. Observe that $\mu(\psi_0)=1^n$, and so 
$t^{\dinv(\pi)}$ is the coefficient of~$s_{1^n}$ in~$\LLT{\pi}{t}$. 
The bigraded Hilbert series of~$\Alt(\DiagHarm{n}{2})$ is then the $q,t$-Catalan polynomial:
\begin{equation}
\label{eq:qt_CatalanFormula}
\sum_{\pi} q^{\area(\pi)}  t^{\dinv(\pi)},
\end{equation}
where the sum is over all Dyck paths $\pi$ of size $n$.


\subsection{The $n\le 4$ case}

We now present the main result of this section stating a strong connection between diagonal harmonics and Hopf chains.
In~\cite{Bergeron-multivariateDiagonalCoinvariantSpaces}, F.~Bergeron showed that the Schur expansion formulas for the (multi-graded) 
Frobenius characteristic of~$\DiagHarm{n}{r}$ for a fixed~$n$ stabilizes in~$r$ and can be computed as a formula in~$r$ where the coefficients are symmetric functions evaluated in~$(q_1, \dots, q_r, 0, 0, \dots)$. This expression was explicitly computed for~$n \le 5$ by F.~Bergeron and later verified and expanded for $n \le 6$ by N.~Thi\'ery~\cite{Thiery}. 
We denote by~$\Phi_{n,r}(q,t)$ the \defn{$q,t$-Frobenius characteristic} of~$\DiagHarm{n}{r}$, which is obtained by substituting the~$(q_1, \dots, q_r)$ parameters in the Frobenius characteristic of~$\DiagHarm{n}{r}$ by the $r$-tuple~$(q, t, 1, \dots, 1)$.
Here, we use Hopf chains, their collar statistics, and the LLT-polynomials to interpret this $q,t$-Frobenius characteristic.

\begin{theorem}
\label{thm:LLT}
The following two symmetric functions coincide for degree~$n \le 4$ and \textbf{any number $r$ of sets of variables}:
\begin{itemize}
\item the $q,t$-Frobenius characteristic~$\Phi_{n,r}(q,t)$ of~$\DiagHarm{n}{r}$, and
\item the sum over Hopf chains of~$\HC{n}{r}$ given by
\[
\Psi_{n,r}(q,t) \eqdef \sum_{\substack{\bpi=(\pi_1,\pi_2,\dots ,\pi_r) \\ \text{Hopf chain of } \HC{n}{r}}}  q^{\col(\bpi)} \LLT{\pi_r}{t}.
\]
\end{itemize}
\end{theorem}

\begin{remark}
The formula of Theorem~\ref{thm:LLT} specializes to the already known interpretations when we ignore the bottom diagonal path~$\pi_1$ in the Hopf chains:
\begin{itemize}
\item when~$r = 2$, the sum ranges over Dyck paths~$\pi_2$ and the collar statistic of~$(\pi_1, \pi_2)$ is just the area of~$\pi_2$. The statement is thus the shuffle formula~\eqref{eq:shuff_formula} (for $n \le 4$).
\item when~$r = 3$, the sum ranges over Tamari intervals~$[\pi_2, \pi_3]$ and the collar statistic of~$(\pi_1, \pi_2, \pi_3)$ is the length of a longest chain in the interval~$[\pi_2, \pi_3]$. The statement is thus~\cite[Conj.~1]{BergeronPrevilleRatelle} (for $n \le 4$).
\end{itemize}
\end{remark}

\enlargethispage{-.5cm}
\begin{proposition}
\label{prop:FrobeniousCharacter_n_1_4}
For~$n \le 4$, the symmetric functions~$\Psi_{n,r}(q,t)$ of Theorem~\ref{thm:LLT} are given by
\allowdisplaybreaks
\begin{align*}
\Psi_{1,r}(q,t) = & s_1 \\[.2cm]
\Psi_{2,r}(q,t) = & \big[ \textstyle ({\scriptstyle q + t}) + {r - 2 \choose 1} \big] s_{11} + s_2 \\[.2cm]
\Psi_{3,r}(q,t) = & \big[ \textstyle ({\scriptstyle q^{3} + q^{2} t + q t^{2} + t^{3} + q t}) + ({\scriptstyle q^{2} + q t + t^{2} + 2q + 2t + 1}) {r - 2 \choose 1} + ({\scriptstyle q + t + 3}) {r - 2 \choose 2} + {r - 2 \choose 3} \big] s_{111} \\[.1cm]
& + \big[ \textstyle ({\scriptstyle q^{2} + q t + t^{2} + q + t}) + ({\scriptstyle q + t + 2}) {r - 2 \choose 1} + {r - 2 \choose 2} \big] s_{21} \\[.1cm]
& + s_{3} \\[.2cm]
\Psi_{4,r}(q,t) = 
& \big[ \textstyle ({\scriptstyle q^{6} + q^{5} t + q^{4} t^{2} + q^{3} t^{3} + q^{2} t^{4} + q t^{5} + t^{6} + q^{4} t + q^{3} t^{2} + q^{2} t^{3} + q t^{4} + q^{3} t + q^{2} t^{2} + q t^{3}}) \\[-.1cm]
& \textstyle \quad + ({\scriptstyle q^{5} + q^{4} t + q^{3} t^{2} + q^{2} t^{3} + q t^{4} + t^{5} + 2 q^{4} + 3 q^{3} t + 3 q^{2} t^{2} + 3 q t^{3} + 2 t^{4}} \\[-5pt]
& \textstyle \qquad\qquad\qquad {\scriptstyle + 3 q^{3} + 5 q^{2} t + 5 q t^{2} + 3 t^{3} + 3 q^{2}+ 6 q t + 3 t^{2} + 3 q + 3 t + 1}) \binom{r-2}{1} \\[-.1cm]
& \textstyle \quad + ({\scriptstyle q^{4} + q^{3} t + q^{2} t^{2} + q t^{3} + t^{4} + 4 q^{3} + 5 q^{2} t + 5 q t^{2} + 4 t^{3} + 9 q^{2} + 12 q t + 9 t^{2} + 15 q + 15 t + 12}) \binom{r-2}{2} \\[-.1cm]
& \textstyle \quad + ({\scriptstyle q^{3} + q^{2} t + q t^{2} + t^{3} + 6 q^{2} + 7 q t + 6 t^{2} + 18 q + 18 t + 29}) \binom{r-2}{3}  \\[-.1cm]
& \textstyle \quad + ({\scriptstyle q^{2} + q t + t^{2} + 8 q + 8 t + 25}) \binom{r-2}{4} + ({\scriptstyle q + t + 9}) \binom{r-2}{5} + \binom{r-2}{6} \big] s_{1111} \\[.1cm]
& + \big[ \textstyle ({\scriptstyle q^{5} + q^{4} t + q^{3} t^{2} + q^{2} t^{3} + q t^{4} + t^{5} + q^{4} + 2q^{3} t + 2q^{2} t^{2} + 2q t^{3} + t^{4} + q^{3} + 2q^{2} t + 2q t^{2} + t^{3} + q t}) \\[-.1cm]
& \textstyle \quad + ({\scriptstyle q^{4} + q^{3} t + q^{2} t^{2} + q t^{3} + t^{4} + 3q^{3} + 4q^{2} t + 4q t^{2} + 3t^{3} + 5q^{2} + 7q t + 5t^{2} + 6q + 6t + 3}) {r - 2 \choose 1} \\[-.1cm]
& \textstyle \quad + ({\scriptstyle q^{3} + q^{2} t + q t^{2} + t^{3} + 5q^{2} + 6q t + 5t^{2} + 12q + 12t + 15}) {r - 2 \choose 2} \\[-.1cm]
& \textstyle \quad + ({\scriptstyle q^{2} + q t + t^{2} + 7q + 7t + 18}) {r - 2 \choose 3} + ({\scriptstyle q + t + 8}) {r - 2 \choose 4} + {r - 2 \choose 5} \big] s_{211} \\[.1cm]
& + \big[ \textstyle ({\scriptstyle q^{4} + q^{3} t + q^{2} t^{2} + q t^{3} + t^{4} + q^{2} t + q t^{2} + q^{2} + q t + t^{2}}) \\[-.1cm]
& \textstyle \quad + ({\scriptstyle q^{3} + q^{2} t + q t^{2} + t^{3} + 2q^{2} + 3q t + 2t^{2} + 3q + 3t + 2}) {r - 2 \choose 1} \\[-.1cm]
& \textstyle \quad + ({\scriptstyle q^{2} + q t + t^{2} + 4q + 4t + 6}) {r - 2 \choose 2} + ({\scriptstyle q + t + 5}) {r - 2 \choose 3} + {r - 2 \choose 4} \big] s_{22} \\[.1cm]
& + \big[ \textstyle ({\scriptstyle q^{3} + q^{2} t + q t^{2} + t^{3} + q^{2} + q t + t^{2} + q + t}) \\[-.1cm]
& \textstyle \quad + ({\scriptstyle q^{2} + q t + t^{2} + 2q + 2t + 3}) {r - 2 \choose 1} + ({\scriptstyle q + t + 3}) {r - 2 \choose 2} + {r - 2 \choose 3} \big] s_{31} \\[.1cm]
& + s_4
\end{align*}
\end{proposition}

\begin{remark}
\label{rem:qtsymmetry}
Observe that~$\Phi_{n,r}(q,t)$ is symmetric in~$q$ and~$t$ since the diagonal harmonics space~$\DiagHarm{n}{r}$ is symmetric on the different sets of variables.
It thus follows from Theorem~\ref{thm:LLT} that~$\Psi_{n,r}(q,t)$ is symmetric in~$q$ and~$t$ for~$n \le 4$. See the expressions of Proposition~\ref{prop:FrobeniousCharacter_n_1_4}.
\end{remark}

\begin{proof}[Proof of Proposition~\ref{prop:FrobeniousCharacter_n_1_4}]
The expressions for~$\Psi_{n,r}(q,t)$ are obtained as follows. For a fixed~$n$, consider all strict Hopf chains $(\pi_1,\pi_2,\dots ,\pi_\ell)$, that is, 
Hopf chains satisfying~${\pi_1 \ne \pi_2 \ne \dots \ne \pi_\ell}$. For any fixed~$n$ there are only finitely many strict Hopf chains. 
For a given strict Hopf chain ${\bpi = (\pi_1,\pi_2,\dots ,\pi_\ell)}$, there are~${r-1 \choose \ell-1}$ distinct Hopf chains of length~$r$ that involve 
exactly the Dyck paths~$\pi_1, \pi_2, \dots, \pi_\ell$, with possible repetitions. From those chains, 
\begin{itemize}
\item ${r-2 \choose \ell-2}$ have a single occurence of~$\pi_\ell$, hence have collar statistic $\col(\bpi)$,
\item ${r-2 \choose \ell-1}$ have the chain~$\pi_\ell$ repeated more than once, hence have collar statistic $0$.
\end{itemize}
Therefore, for any~$n$ and~$r$, the function~$\Psi_{n,r}(q,t)$ of Theorem~\ref{thm:LLT} is given by
\[
\Psi_{n,r}(q,t) = \sum_{\substack{\bpi=(\pi_1,\dots ,\pi_\ell) \\ \text{strict Hopf chain}}} \Big[ {\textstyle q^{\col(\bpi)} \binom{r-2}{\ell-2} + \binom{r-2}{\ell-1} } \Big] \LLT{\pi_\ell}{t}.
\]
The expressions in the proposition are then obtained by computer, generating all strict Hopf chains for~$n \le 4$.
This was done using the sage software~\cite{Sage,SageCombinat}.
\end{proof}

\begin{proof}[Proof of Theorem~\ref{thm:LLT}]
We obtain~$\Phi_{n,r}(q,t) = \Psi_{n,r}(q,t)$ for all~$n \le 4$ and any~$r>0$ comparing the expressions in Proposition~\ref{prop:FrobeniousCharacter_n_1_4} with the symmetric function~$\Phi_{n,r}(q,t)$ obtained in~\cite{Bergeron-multivariateDiagonalCoinvariantSpaces}.
\end{proof}

\enlargethispage{-.5cm}
We now present some immediate consequences of Theorem~\ref{thm:LLT}.

\begin{corollary}
\label{coro:bigcoro}
For degree~$n \le 4$ and \textbf{any number $r$ of sets of variables}:
\begin{enumerate}
\item \label{item:bigradedHilbertSeriesAlt}
The bigraded Hilbert series of~$\Alt(\DiagHarm{n}{r})$ with respect to two sets of variables is given by
\[
\widetilde\Psi_{n,r}(q,t) \eqdef \sum_{\substack{\bpi=(\pi_1,\pi_2,\dots ,\pi_r) \\ \text{Hopf chain of } \HC{n}{r}}}  q^{\col(\bpi)} t^{\dinv(\pi_r)}.
\]

\item \label{item:dimensionAlt}
The dimension of~$\Alt(\DiagHarm{n}{r})$ equals the number of Hopf chains of length~$r$ and size~$n$. 

\item \label{item:eExpansion}
The $q$-Frobenius characteristic~$\Phi_{n,r}(q,1)$ of~$\DiagHarm{n}{r}$ is given by
\[
\Psi_{n,r}(q,1) = \sum_{\substack{\bpi=(\pi_1,\pi_2,\dots ,\pi_r) \\ \text{Hopf chain of } \HC{n}{r}}}  q^{\col(\bpi)} e_{\type(\pi_r)},
\]
where $\type(\pi_r)$ is the partition of the connected $N$ steps lengths in $\pi_r$.

\item \label{item:dimensionDH}
The dimension of~$\DiagHarm{n}{r}$ equals the number of labelled Hopf chains of length~$r$ and size~$n$. 
\end{enumerate}
\end{corollary}

\begin{proof}
Part~\eqref{item:bigradedHilbertSeriesAlt} follows by considering the coefficient of~$s_{1^n}$ in the equality~$\Phi_{n,r}(q,t) = \Psi_{n,r}(q,t)$ of Theorem~\ref{thm:LLT} since:
\begin{itemize}
\item the bigraded Hilbert series of~$\Alt(\DiagHarm{n}{r})$ is the coefficient of~$s_{1^n}$ in~$\Phi_{n,r}(q,t)$, and
\item $t^{\dinv(\pi)}$ is the coefficient of~$s_{1^n}$ in~$\LLT{\pi}{t}$.
\end{itemize}
Part~\eqref{item:dimensionAlt} is an evaluation of Part~\eqref{item:bigradedHilbertSeriesAlt} at $q = t = 1$. \\
Part~\eqref{item:eExpansion} is the specialization of Theorem~\ref{thm:LLT} at~$t = 1$, using that~$\LLT{\pi}{1} = e_{\type(\pi)}$. \\
Part~\eqref{item:dimensionDH} is obtained by applying the map~$e_\lambda\mapsto {n\choose \lambda}$ to~$\Psi_{n,r}(1,1)$.
\end{proof}

\begin{remark}
For convenience, we include the data of Corollary~\ref{coro:bigcoro} for~$n \le 4$:
\begin{itemize}
\item $\widetilde \Phi_{n,r}(q,t)$ can already be found as the coefficient of~$s_{1^n}$ in the data of Proposition~\ref{prop:FrobeniousCharacter_n_1_4},
\item the dimensions of~$\Alt(\DiagHarm{n}{r})$ and~$\DiagHarm{n}{r}$ are given by
\[
\begin{array}{|c|l|l|}
\hline
n & \dim \big( \Alt(\DiagHarm{n}{r}) \big) = & \dim \big( \DiagHarm{n}{r} \big) = \\
& \text{number of Hopf chains} & \text{number of laballed Hopf chains}\\
\hline
n=1 & {r \choose 0} & {r+1 \choose 0}\phantom{\Big]}\\[4pt]
n=2 & {r \choose 1} & {r+1 \choose 1} \\[4pt]
n=3 &  {r \choose 1} + 3{r \choose 2} + {r \choose 3} & {r +1\choose 1} + 4{r+1 \choose 2} + {r+1 \choose 3}  \\[4pt]
n=4 &  {r \choose 1} + 12{r \choose 2} + 29{r \choose 3} & {r+1 \choose 1} + 22{r+1 \choose 2} + 56{r+1 \choose 3} \\[4pt]
&\quad\quad +25{r \choose 4} +9{r \choose 5} +{r \choose 6} &\quad\quad + 40{r +1\choose 4} + 11{r+1 \choose 5} +{r+1 \choose 6}  \\[3pt]
\hline
\end{array}
\]
\item the symmetric functions $\Psi_{n,r}(q,1)$ are given by
\begin{align*}
\Psi_{1,r}(q,1) = \; & e_1 \\
\Psi_{2,r}(q,1) = \; & \textstyle e_{11} + \big[{\scriptstyle q}+ {r-2 \choose 1}\big] e_{2} \\[3pt]
\Psi_{3,r}(q,1) = \; & \textstyle e_{111} + \big[({\scriptstyle q^2+2q})+({\scriptstyle q+3}){r -2\choose 1} +{r -2\choose 2}\big] e_{21} \\[3pt] & \textstyle + \big[{\scriptstyle q^3}+({\scriptstyle q^2+2q+1}){r -2\choose 1} + ({\scriptstyle q+3}){r-2 \choose 2} +{r-2 \choose 3}\big] e_{3} \\[3pt]
\Psi_{4,r}(q,1) = \; & \textstyle e_{1111} + \big[({\scriptstyle q^{3} + 2q^{2} + 3q})+({\scriptstyle q^{2} + 3q + 6}){r -2\choose 1} +({\scriptstyle q + 4}){r -2\choose 2}+{r -2\choose 3}\big] e_{211}
	\\[4pt]
 	& \textstyle + \big[({\scriptstyle q^{4} + q^{2}})+({\scriptstyle q^{3} + 2q^{2} + 4q + 2}){r-2\choose 1} +({\scriptstyle q^{2} + 4q + 7}){r -2\choose 2} +({\scriptstyle q + 5}){r-2\choose 3}+{r -2\choose 4}\big] e_{22}
	\\[4pt]
	& \textstyle  +\big[({\scriptstyle q^{5} + q^{4} + 2q^{3}}) + ({\scriptstyle q^{4} + 3q^{3} + 6q^{2} + 8q + 4}) {r-2 \choose 1} + ({\scriptstyle q^{3} + 5q^{2} + 13q + 18}){r-2 \choose 2} \\
	& \textstyle 	\qquad\qquad + ({\scriptstyle q^{2} + 7q + 19}){r -2\choose 3} + ({\scriptstyle q + 8}){r-2 \choose 4}+{r-2 \choose 5}\big] e_{31}
	\\[4pt]
	& \textstyle + \big[ ({\scriptstyle q^{6}})+ ({\scriptstyle q^{5} + 2q^{4} + 3q^{3} + 3q^{2} + 3q + 1}){r-2 \choose 1} + ({\scriptstyle q^{4} + 4q^{3} + 9q^{2} + 15q + 12}){r -2\choose 2} \\
	& \textstyle 	\qquad\qquad + ({\scriptstyle q^3 + 6q^2 + 18q + 29}){r-2 \choose 3} + ({\scriptstyle q^2 + 8q + 25}){r -2\choose 4}+ ({\scriptstyle  q + 9}){r -2\choose 5}+{r -2\choose 6}\big] e_{4} \,.
\end{align*}
\end{itemize}
\end{remark}

\begin{remark}
As in Remark~\ref{rem:qtsymmetry}, observe that~$\widetilde\Psi_{n,r}(q,t)$ is symmetric in~$q$ and~$t$ for~$n \le 4$.
This is surprising from the combinatorial perspective since $\dinv$ only depends on the top path~$\pi_r$ while $\col(\bpi)$ depends on the full Hopf chain~$\bpi = (\pi_1, \dots, \pi_r)$. It is an open problem to find a bijection on Hopf chains (for~$n \le 4$) that would exchange the two statistics~$\dinv$ and~$\col$. This generalizes the symmetry problem of the $q,t$-Catalan numbers~\cite[Open~Problem~3.11]{Haglund-qt-catalan}.
\end{remark}

\begin{remark}
In the case of two sets of variables, the bigraded Hilbert series of~$\Alt(\DiagHarm{n}{2})$ can be combinatorially expressed in terms of the pair of statistics $\area$ and $\dinv$ or $\area$ and $\bounce$ on Dyck paths. The zeta map sends $\dinv$ to $\area$ and $\area$ to $\bounce$. Therefore, as mentioned by D.~Armstrong, N.~A.~Loehr and G.~S.~Warrington in~\cite{ArmstrongLoehrWarrington-sweep}, one point of view is that rather than having three statistics we have one statistic $\area$ and a nice map zeta. It would be interesting to find more generalizations of the zeta map on chains of Dyck paths that allows us to describe the missing statistics in multivariate diagonal harmonics using the collar statistic. A first attempt to find such generalization could be related to the steep-zeta map suggested in Remark~\ref{rem_step_zeta_map}.
\end{remark}

\begin{remark}
We have verified that Theorem~\ref{thm:LLT} and Corollary~\ref{coro:bigcoro} still hold if we replace $\pi_r = N^{i_1} E^{j_1} \dots N^{i_p} E^{j_p}$ by its transpose~$\overline{\pi}_r \eqdef N^{j_p} E^{i_p} \dots N^{j_1} E^{i_1}$ in the formulas for $\Psi_{n,r}(q,t)$, $\widetilde \Psi_{n,r}(q,t)$, and~$\Psi_{n,r}(q,1)$.
This symmetry is clear in the two variate case (since~${\area(\pi) = \area(\overline{\pi})}$) but is not obvious in the multivariate case.
In fact, this symmetry seems a coincidence and might not hold for larger values of~$n$, nor in the more general context of rectangular diagonal harmonics.
Guided by preliminary computations performed with F.~Bergeron, we believe that the right formulas do involve the transposition of~$\pi_r$.
\end{remark}


\subsection{The $n=5$ case and further conjectures}

For~$n = 4$, we had to refine Tamari chains to Hopf chains and length of a longest chain in intervals to the collar statistic.
We will see in this Section that this should be refined even further for~$n = 5$.

In comparison to the formulas in Proposition~\ref{prop:FrobeniousCharacter_n_1_4}, the formula for $\Psi_{5,r}(q,t)$ and the Frobenius characteristic~$\Phi_{5,r}(q,t)$ of~$\DiagHarm{5}{r}$ differ by very few terms.
This means that the Hopf chains for $n=5$ are not the right subcollection of Tamari chains, and that the collar statistic is not defined on the right set.
Therefore, it makes sense to set $q=t=1$ first:
\begin{align*}
\Psi_{5,r}(1,1) = \; & \textstyle e_{11111}
	\\
	& \textstyle + {\scriptstyle \big[10{r-1 \choose 1} +10{r-1 \choose 2} +5{r -1\choose 3}+{r-1 \choose 4}\big] } e_{2111} 
	\\[4pt]
	& + {\scriptstyle  \big[10{r-1 \choose 1} +48{r -1\choose 2} +65{r-1 \choose 3}+40{r -1\choose 4} } {\scriptstyle  +11{r -1\choose 5} +  {r -1\choose 6}\big] } e_{221}
	\\[4pt]
    & +{\scriptstyle \big[10{r-1 \choose 1} +62{r -1\choose 2} +115{r-1 \choose 3}+107{r -1\choose 4}+52{r -1\choose 5}+12{r -1\choose 6}+{r -1\choose 7}\big]} e_{311}
	\\[4pt]
	& \textstyle + {\scriptstyle  \big[5{r-1 \choose 1} + 68{r-1\choose 2} + 218{r-1 \choose 3}+297{r-1\choose 4}+208{r-1\choose 5}+77{r-1\choose 6}+14{r-1\choose 7}+{r-1\choose 8}\big] }e_{32}
   	\\[4pt]
 	& \textstyle + {\scriptstyle   \big[5{r-1 \choose 1} +88{r -1\choose 2} +360{r-1 \choose 3}  +652{r -1\choose 4}  +638{r -1\choose 5} +354{r -1\choose 6}+109{r -1\choose 7}+17{r -1\choose 8}+{r -1\choose 9}\big]} e_{41} \\[4pt]
  	& \textstyle + {\scriptstyle   \big[{r-1 \choose 1} +40{r -1\choose 2} +276{r-1 \choose 3}+763{r -1\choose 4}+1097{r -1\choose 5} +909{r -1\choose 6}+444{r -1\choose 7}+124{r -1\choose 8}+18{r -1\choose 9}+{r -1\choose 10}\big]} e_{5} \,.
\end{align*}
We can write the difference of the two symmetric functions as
\[
\Psi_{5,r}(1,1) - \Phi_{5,r}(1,1) = \textstyle {r+4 \choose 8} e_{41} + {r+4 \choose 9} e_{5} \,. 
\]

This computation suggests that among the Hopf chains surviving in~$\Psi_{5,r}(1,1) - \Phi_{5,r}(1,1)$ there should be some Hopf chains $\bpi=(\pi_1,\pi_2,\pi_3,\pi_4)$ with~$\type(\pi_4)=[4,1]$ that we need to eliminate.
To be more precise, we look for a set~$O$ of strict Hopf chains~$\bpi$, such that the set of Hopf chains that contain at least one subchain in~$O$ is precisely enumerated by the difference~${\Psi_{5,r}(1,1) - \Phi_{5,r}(1,1)}$. 
For instance, forbidding any of the three patterns in Figure~\ref{fig:potentialInvalidKiller} would kill the given excess.
Even more, the formula of Theorem~\ref{thm:LLT} restricted to the remaining chains and considering the restricted collar statistic would coincide with~$\Phi_{5,r}(q,t)$.
Computer exploration should give us candidates for the obstruction sets~$O$ and we look for a set~$S$ of surviving chains (that do not contain any subchains in the set of obstructions~$O$) satisfying the properties in the following question.

\begin{figure}[t]
	\capstart
	\centerline{
		\includegraphics[scale=0.5]{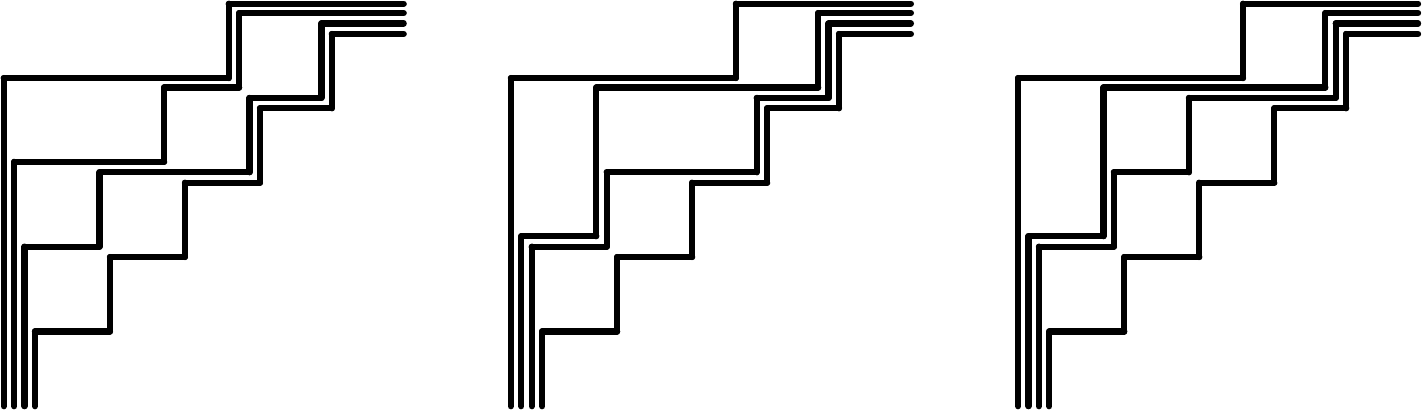}
	}
	\caption{Three potential killers for $n=5$.}
	\label{fig:potentialInvalidKiller}
\end{figure}

\begin{question}
\label{conj:Sconjecture}
Is there a subfamily~$S$ of Hopf chains such that the Frobenius characteristic $\Phi_{n,r}$ is equal to
\[
\Psi_{S,n,r}(q,t) \eqdef \sum_{\bpi=(\pi_1,\dots ,\pi_\ell)\in S} q^{\col_S(\bpi)} \LLT{\pi_\ell}{t},
\]
where~$\col_S(\pi)$ denotes the $S$-collar statistic
\[
\col_S(\bpi) = \max \set{\ell}{{(\pi_{r-1},\nu_1,\cdots,\nu_{\ell-1},\pi_r) \text{ {\bf strict} Tamari chain,} \hfill \atop (\pi_1,\pi_2,\dots ,\pi_{r-1},\nu_1,\cdots,\nu_{\ell-1},\pi_r) \in S}}.
\]
\end{question}

\begin{remark}
\label{rem:bigcoro}
A positive answer to Question~\ref{conj:Sconjecture} would have the following consequences similar to Corollary~\ref{coro:bigcoro}:
\begin{enumerate}
\item the bigraded Hilbert series of~$\Alt(\DiagHarm{n}{r})$ with respect to two sets of variables would be given by
\[
\widetilde\Psi_{S,n,r}(q_1,q_2) \eqdef \sum_{\bpi=(\pi_1,\pi_2,\dots ,\pi_r) \in S}  q_1^{\col_S(\bpi)} q_2^{\dinv(\pi_r)}.
\]
\item The dimension of~$\Alt(\DiagHarm{n}{r})$ equals the number of chains in~$S$ of length~$r$ and size~$n$. 

\item \label{item:eExpansion2}
The $q$-Frobenius characteristic~$\Phi_{n,r}(q,1)$ of~$\DiagHarm{n}{r}$ would be given by
\[
\Psi_{S,n,r}(q,1) \eqdef \sum_{\bpi=(\pi_1,\dots ,\pi_\ell)\in S} q^{\col_S(\bpi)} e_{\type(\pi_\ell)},
\]
where $\type(\pi_r)$ is the partition of the connected $N$ steps lengths in $\pi_r$.

\item The dimension of~$\DiagHarm{n}{r}$ equals the number of labelled chains in~$S$ of length~$r$ and size~$n$. 
\end{enumerate}
\end{remark}

Note that Part~\eqref{item:eExpansion2} of Remark~\ref{rem:bigcoro} is closely related to the following $e$-positivity conjecture, made by F.~Bergeron in 2011 and shared with us by personal communication~\cite{Bergeron-ePositivityConjecture}.

\begin{conjecture}[\cite{Bergeron-ePositivityConjecture}]
The Frobenius characteristic of~$\DiagHarm{n}{r}$ is $e$-positive for any fixed~$n$ and~$r$ when the parameter~$q_r=1$.
\end{conjecture} 

In fact, Part~\eqref{item:eExpansion2} of Remark~\ref{rem:bigcoro} suggests that the expansion of $\Phi_{n,r}$ is not only $e$-positive, but for a fixed~$n$ the coefficient of~$e_\lambda$ is a polynomial in~$r$ that has a $q$-positive expansion in the binomials~${r-2\choose *}$. Remark that this polynomial in $r$ is not a $q$-positive expansion in the binomials~${r-1 \choose *}$. 

To conclude, we discuss one more evidence for the connection between Hopf chains and multivariate diagonal harmonics. Namely, the following conjecture was verified by computer up to~$n = 6$.

\begin{conjecture}
\label{conj:structureTopCoeff}
For any fixed partition~$\lambda \vdash n$, the coefficient~$c_{n,r}^\lambda(q)$ of~$e_\lambda$ in~$\Phi_{n,r}(q,1)$ 
is a polynomial in~$r$ that has a $q$-positive expansion in the binomials~${r-2\choose *}$. Moreover, it is of the form
\[
c_{n,r}^\lambda(q) = \textstyle \binom{r-2}{m} + (q+b_\lambda) \textstyle \binom{r-2}{m-1} + \text{lower terms,}
\]
where~$b_\lambda$ is a positive integer and~$m = \sum_{i=1}^n (n-i+1)(\lambda_i-1)$ is the maximal area of a Dyck path with connected $N$ steps of type~$\lambda$. In the formula for~$m$, we assume~$\lambda_i = 0$ for $i$ larger than the length of~$\lambda$.
\end{conjecture}

The next statement is the counterpart of the last conjecture for Hopf chains.
We say that a path \defn{leaves the diagonal at most once} if it is of the form~$(NE)^a \rho (NE)^b$ where~$a,b \in \N$ and~$\rho$ is a Dyck path whose endpoints are the only diagonal points.

\begin{theorem}
\label{thm:maxchain}
For any Dyck path~$\pi$ that leaves the diagonal at most once, there is a unique strict Hopf chain~$(\pi_1,\ldots,\pi_m)$ of maximal length where $m = area(\pi)+1$ and~$\pi_m = \pi$.
\end{theorem}
 
\begin{proof}
It is easy to check that the chain~$\bpi = (\pi_1,\ldots,\pi_m)$ where~$\pi_{m-\ell}$ is obtained from $\pi$ by removing~$\ell$ squares from left to right, row by row, from top to bottom, is a strict Hopf chain. It has the desired length and is certainly maximal. \fref{fig:maxStrictHopfChain} illustrates~$\bpi$ with an example.

\begin{figure}[ht]
	\capstart
	\centerline{
		\includegraphics[scale=0.5]{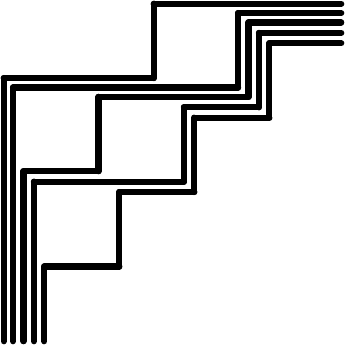}
	}
	\caption{The maximal chain~$\bpi$.}
	\label{fig:maxStrictHopfChain}
\end{figure}

Now assume we have another Tamari chain~$(\pi'_1, \dots, \pi'_m) = \bpi' \ne \bpi$ of the same length and with~$\pi'_m = \pi$. We show that~$\bpi'$ cannot satisfy the Hopf chain condition, thus concluding the proof. For~$k \in [m]$, we denote by~$B_k$ (resp.~$B'_k$) the unique box of the grid located between the paths~$\pi_k$ and~$\pi_{k+1}$ (resp.~$\pi'_k$ and~$\pi'_{k+1}$). Let~$k$ be such that~$\pi'_{k} \ne \pi_{k}$ while~$\pi'_\ell = \pi_\ell$ for any~$\ell > k$. Consider the boxes~$B_k = (x,y)$ and~$B'_k = (x',y')$. Note that $B'_k$ is a box above~$\pi'_k$ in row~$y'$. We can thus consider the rightmost box~$U$ above~$\pi'_k$ in row~$y'$. Moreover, by definition of~$\bpi$, we have~$y > y'$ so that $B_k$ is a box below~$\pi'_k$ in some row above~$y$. Since~$\pi$ leaves the diagonal at most once, it follows that there is at least a box below~$\pi'_k$ in row~$y'+1$. Therefore, we can consider the leftmost box~$V$ below~$\pi'_k$ in row~$y'+1$. Let~$p$ and~$q$ be such that~$B'_p = U$ and~$B'_q = V$. By definition, $\pi'_p$ passes below both~$U$ and~$V$, $\pi'_q$ passes below~$U$ but above~$V$, and $\pi'_{q+1}$ passes above both~$U$ and~$V$. Therefore, $\pi'_q \not< \pi'_{q+1}$ in the~$\pi'_p$-Tamari order. We conclude that~$\bpi'$ is not a Hopf chain as expected.
\end{proof}

Note that for a fixed partition~$\lambda$, there is a unique Dyck path~$\pi_\lambda$ of type~$\lambda$ and maximal area~$m = \sum_{i=1}^n (n-i+1)(\lambda_i-1)$, and moreover this path~$\pi_\lambda$ leaves the diagonal at most once.
Therefore, if the connection between diagonal harmonics and certain modified Hopf chains suggested in Question~\ref{conj:Sconjecture} was established, Conjecture~\ref{conj:structureTopCoeff} would follow from Theorem~\ref{thm:maxchain} applied to~$\pi_\lambda$.


\section*{Acknowledgments}

We thank Fran\c{c}ois Bergeron for fruitful discussions on multivariate diagonal harmonics.
These discussions were essential for the development of Part~\ref{part:multivariateDiagonalHarmonics} and lead in particular to Theorem~\ref{thm:LLT}.
We also thank Nicolas Thi\'ery for sharing his computations.
We are grateful to Myrto Kallipoliti, Robin Sulzgruber and Eleni Tzanaki for pointing out the observation in Remark~\ref{rem_bouncek_heightk}. 
Thanks to their observation we discovered the connection between the Steep-Bounce Conjecture and the zeta map in Section~\ref{sec_steepbounce_conjecture}.
The computation and tests needed along the research were done using the open-source mathematical software \texttt{Sage}~\cite{Sage} and its combinatorics features developed by the \texttt{Sage-combinat} community~\cite{SageCombinat}.
Finally, we are grateful to an anonymous referee for relevant suggestions.


\bibliographystyle{alpha}
\bibliography{BCP_HopfAlgebra_PipeDreams}
\label{sec:biblio}


\end{document}